\documentclass[10pt]{amsart}
\usepackage{marginnote}
\usepackage[normalem]{ulem}
\DeclareRobustCommand{\sout}{\bgroup\markoverwith{\textcolor{red}{\rule[.5ex]{2pt}{0.4pt}}}\ULon}

\usepackage{enumitem}

\usepackage{mathtools, stmaryrd}

\usepackage{xparse} \DeclarePairedDelimiterX{\Iintv}[1]{\llbracket}{\rrbracket}{\iintvargs{#1}}
\NewDocumentCommand{\iintvargs}{>{\SplitArgument{1}{,}}m}
{\iintvargsaux#1} %
\NewDocumentCommand{\iintvargsaux}{mm} {#1\mkern1.5mu,\mkern1.5mu#2}

\usepackage{graphicx}
\usepackage{color}
\usepackage[colorlinks=true]{hyperref}
\usepackage{mathtools}

\definecolor{RedOrange}{cmyk}{ 0, 0.77, 0.87, 0}
\definecolor{RoyalPurple}{cmyk}{ 0.84, 0.53, 0, 0}
\definecolor{YellowGreen}{cmyk}{ 0.44, 0, 0.74, 0}
\definecolor{Fuchsia}{cmyk}{ 0.47, 0.91, 0, 0.08}
\definecolor{Blue}{cmyk}{ 0.84, 0.53, 0, 0}
\definecolor{BlueViolet}{cmyk}{ 0.84, 0.53, 0, 0}
\definecolor{Black}{cmyk}{ 0.75, 0.68, 0.67, 0.9}
\usepackage{xcolor}
\usepackage{verbatim}
\usepackage{amsmath}
\usepackage{amssymb, mathrsfs}
\usepackage{amsbsy}
\usepackage{bbm}
\usepackage{amscd}
\usepackage{amsthm}

\usepackage[english]{babel}
\usepackage{todonotes}
\usepackage[T1]{fontenc}

\usepackage{color}

\newcommand{\lf}{\lfloor}
\newcommand{\rf}{\rfloor}
\newcommand{\R}{\mathbb{R}}

\newcommand{\N}{\mathbb{N}}

\newcommand{\e}{\varepsilon}

\newcommand{\E}{\mathbb{E}}
\newcommand{\Z}{\mathbb{Z}}

\newcommand{\pp}{\mathbb{P}}
\renewcommand{\P}{\mathbb{P}}
\newcommand{\kD}{\mathcal{D}}
\newcommand{\eps}{\epsilon}
\newcommand{\kA}{\mathcal{A}}
\newcommand{\kB}{\mathcal{B}}
\newcommand{\kC}{\mathcal{C}}

\newcommand{\kR}{\mathcal{R}}
\newcommand{\kO}{\mathcal{O}}

\newcommand{\kF}{\mathcal{F}}

\newcommand{\kM}{\mathcal{M}}
\newcommand{\kE}{\mathcal{E}}
\newcommand{\kH}{\mathcal{H}}
\newcommand{\kI}{\mathcal{I}}
\newcommand{\kS}{\mathcal{S}}
\newcommand{\kN}{\mathcal{N}}

\newcommand{\kL}{\mathcal{L}}

\newcommand{\rmT}{\mathrm{T}}
\newcommand{\rmB}{\mathrm{B}}

\newcommand{\dd}{\mathrm{d}}

\newcommand{\by}{\boldsymbol{y}}
\newcommand{\bb}{\boldsymbol{b}}

\newcommand{\kJ}{\mathcal{J}}

\newcommand{\tf}{\tilde{f}}

\newcommand{\fhd}{f^{\textrm{hd},\kJ}}
\newcommand{\fsd}{f^{\textrm{sd},\kJ}}

\DeclareMathOperator{\1}{\mathbbm{1}}

\newcommand{\lin}{\left[\kern-0.15em\left[}
\newcommand{\rin} {\right]\kern-0.15em\right]}
\newcommand{\linf}{[\kern-0.15em [}
\newcommand{\rinf} {]\kern-0.15em ]}
\newcommand{\ilin}{\left]\kern-0.15em\left]}
\newcommand{\irin} {\right[\kern-0.15em\right[}

\def\be#1{\begin{equation*}#1\end{equation*}}
\def\ben#1{\begin{equation}#1\end{equation}}
\def\bea#1{\begin{eqnarray*}#1\end{eqnarray*}}
\def\bean#1{\begin{eqnarray}#1\end{eqnarray}}
\def\al#1{\begin{align*}#1\end{align*}}
\def\aln#1{\begin{align}#1\end{align}}

\usepackage{constants}

\newconstantfamily{c}{symbol=c}

\newconstantfamily{a}{symbol=\alpha}

\newcommand{\secno}[1]{\thesection.\arabic{#1}}
\newconstantfamily{kE}{
symbol=\mathcal{E},
format=\secno,
reset={section}
}

\renewcommand{\tilde}{\widetilde}

\DeclareMathOperator{\T}{T}

\newtheorem{lem}{Lemma}[section]

\newtheorem{prop}[lem]{Proposition}
\newtheorem{theo}[lem]{Theorem}

\newtheorem{cor}[lem]{Corollary}

\newtheorem {rem}[lem] {Remark}

\usepackage{color}
\definecolor{lilas}{RGB}{182, 102, 210}

\newcommand{\NK}{\color{lilas}}

\numberwithin{equation}{section}

\def\be#1{\begin{equation*}#1\end{equation*}}
\def\ben#1{\begin{equation}#1\end{equation}}
\def\bea#1{\begin{eqnarray*}#1\end{eqnarray*}}
\def\bean#1{\begin{eqnarray}#1\end{eqnarray}}

\begin{document}
\title[Upper tail large deviation for 1D-frog model]{Upper tail large deviation for the one-dimensional frog model}

\author[V.~H.~CAN]{Van Hao CAN}
\address[V.~H.~CAN]{Institute of Mathematics, Vietnam Academy of Science and Technology, 18 Hoang Quoc Viet, Cau Giay, Hanoi, Vietnam.}
\email{cvhao@math.ac.vn}

\author[N.~KUBOTA]{Naoki KUBOTA}
\address[N.~KUBOTA]{College of Science and Technology, Nihon University, Chiba 274-8501, Japan.}
\email{kubota.naoki08@nihon-u.ac.jp}

\author[S.~NAKAJIMA]{Shuta NAKAJIMA}
\address[S.~NAKAJIMA]{Graduate School of Science and Technology, Meiji University, Kanagawa 214-8571, Japan.}
\email{njima@meiji.ac.jp}


\begin{abstract}
In this paper, we study the upper tail large deviation for the one-dimensional frog model. In this model,  sleeping and active frogs are assigned to  vertices on $\Z$.   While  sleeping frogs do not move, the active ones move as  independent  simple random walks and activate any  sleeping frogs. 
The main object of interest in this model is the asymptotic behavior of the first passage time $\rmT(0,n)$, which is the time needed to activate the frog at the vertex $n$, assuming there is only one active frog at $0$ at the  beginning.  While the law of large numbers and central limit theorems have been well established, the intricacies of large deviations remain elusive. Using renewal theory,  B\'erard and Ram\'irez \cite{BR} have pointed out a slowdown phenomenon where the probability that the first passage time $\rmT(0,n)$ is significantly larger than its expectation decays sub-exponentially and lies between $\exp(-n^{1/2+o(1)})$ and $\exp(-n^{1/3+o(1)})$. In this article, using a novel covering process approach, we confirm that $1/2$ is the correct exponent, i.e., the rate of upper large deviations is given by $n^{1/2}$. Moreover, we obtain an explicit rate function that is characterized by properties of Brownian motion and is strictly concave.
\end{abstract}


\maketitle


\section{Introduction} 

In this paper, we treat the interacting particle system consisting of ``active'' and ``sleeping'' states as follows:
First of all, we place infinitely many particles in some space, according to a deterministic rule.
Active particles can randomly move around in the space and sleeping particles do not move at first. However, sleeping particles become active and start moving around as soon as they are touched by active particles.
Initially, only one particle is active and the others are sleeping.
When the system starts, the first active particle gradually generates active particles by touching sleeping ones, and they propagate across space, with time.

We call the interacting particle system above the \emph{frog model} and regard particles as frogs in the present paper.
However, the frog model has several names circumstantially.
The frog model was originally introduced in the image of information spreading (see the introduction of \cite{AlvMacPop02}):
every active frog has some information and shares it with sleeping frogs touched by active ones.
In \cite[Section~{2.4}]{TelWor99} (which is the first published article on the frog model), the frog model is called the egg model.
It is said that R.~Durrett coined the name ``frog model'' proposing a discrete space-time version (see the introduction of \cite{AlvMacPop02} again).
On the other hand, a continuous-time version of the frog model is interpreted as a combustion phenomena described by a system composed of two types of frogs.
In this case, the frog model is often called the ``combustion model'' or ``the reaction $A+B \to 2A$'' (see for instance \cite{RamSid04}).
One of interest object of the frog model is the diffusion speed of active particles, and it has been investigated in the view of the probabilistic theory for several decades: the law of large numbers, the central limit theorem and the large deviation principle, which provide the asymptotic behavior, the fluctuation around the average behavior, and the decay rate of the tail probability for the diffusion speed, respectively.

The study of the diffusion speed has mainly made progress in the case where an underlying space is the $d$-dimensional lattice $\Z^d$ ($d \geq 1$) and each site of $\Z^d$ initially has one frog (the genetic active frog is put on the origin $0$ of $\Z^d$).
We hereafter focus on this frog model.
Alves et al.~~\cite{AlvMacPop02} and Ram\'{\i}rez and Sidoravicius~\cite{RamSid04} completely solved the law of large numbers for the diffusion speed in all dimensions and both discrete and continuous-time settings.
On the other hand, the central limit theorem and the large deviation principle on $\Z$ are studied in \cite{ComQuaRam07} and \cite{BerRam10}, respectively.

This paper deals with a part of the large deviation principle for the diffusion speed in the discrete-time frog model on $\Z$.
B\'{e}rard and Ram\'{\i}rez~\cite{BerRam10} investigated this topic in the continuous-time setting.
In particular, they observed the so-called \emph{slowdown phenomenon} for the propagation of active frogs by giving some partial estimates for the upper tail large deviation probability for the diffusion speed, which is the probability that the diffusion speed deviates upward from its typical behavior (see the discussion above Theorem~\ref{mth1} for more details).
The argument used in \cite{BerRam10} might work for the discrete-time setting as well.
However, in this paper, we take a different approach using a novel energy coming from the one-dimensional Brownian motion (see \eqref{eq:def_energy}), and completely solve the upper tail large deviation probability for the diffusion speed in the discrete-time setting.

\subsection{The main result}
We first state the dynamics of frogs and define the diffusion speed precisely.
Let $d \geq 1$.
The dynamics of frogs are given by independent simple, symmetric random walks on $\Z^d$ (we drop the adjective ``symmetric'' below, as is customary).
For each $x \in \Z^d$, write $(S_n^x)_{n=0}^\infty$ for these random walks on $\Z^d$ with $S_0^x=x$.
This describes the trajectory of the frog initially sitting on $x$ after becoming active.
For any $x,y \in \Z^d$, the \emph{first passage time} from $x$ to $y$ is defined by
\begin{align*}
    \T(x,y):=\inf\left\{ \sum_{i=0}^{k-1}t(x_i,x_{i+1}):
    \begin{minipage}{13em}
        $k \geq 1$ and $x_0,x_1,\dots,x_k \in \Z^d$\\
        with $x_0=x$ and $x_k=y$
    \end{minipage}\right\},
\end{align*}
where
\begin{align*}
    t(x_i,x_{i+1}):=\inf\bigl\{ n \geq 0:S_n^{x_i}=x_{i+1} \bigr\}.
\end{align*}
The main object of interest in this paper is the first passage time $\T(0,\cdot)$, which represents the diffusion speed at which active frogs propagate from the origin $0$.
Let us now explain the dynamics of our frog model and the intuitive meaning of the first passage time $\T(0,y)$:
First, we put the particle on all sites of $\Z^d$.
The behavior of the frog sitting on a site $x$ is controlled by the simple random walk $S_\cdot^x$, but not all particles move around from the beginning.
At first, the only frog sitting on $0$ are active and perform a simple random walk.
On the other hand, the other frogs are sleeping and do not move.
Each sleeping frog becomes active and starts to perform a simple random walk once it is touched by an active frog.
When we repeat this procedure for the remaining sleeping frogs, $\T(0,y)$ represents the minimum time at which an active frog reaches $y$.
It is clear that the first passage time is subadditive in the following sense:
\begin{align*}
    \T(x,z) \leq \T(x,y)+\T(y,z),\qquad x,y,z \in \Z^d.
\end{align*}
Furthermore, Alves et al.~\cite[Section~3]{AlvMacPop02} showed the integrability of the first passage time $\T(x,y)$. Combining these with the subadditive ergodic theorem (see for instance \cite[Theorem~{6.4.1}]{Dur19_book}) yields the following asymptotic behavior of the first passage time:
there exists a (nonrandom) norm $\mu(\cdot)$ on the $d$-dimensional Euclidean space $\R^d$ such that $\P$-a.s.,
\begin{align}\label{eq:shape_thm}
    \lim_{|y|_1 \to \infty} \frac{\T(0,y)-\mu(y)}{|y|_1}=0,
\end{align}
where $|\cdot|_1$ denotes the $\ell^1$-norm on $\R^d$.
Furthermore, $\mu(\cdot)$ is invariant under permutations of the coordinates and under reflections in the coordinate hyperplanes.
The norm $\mu(\cdot)$ is called the \emph{time constant} and $\T(0,y)$ asymptotically behaves like $\mu(y)$ as $|y|_1 \to \infty$.

From now on, we assume $d=1$ and set $\mu:=\mu(1)$.
Before stating our main result, we shall explain the motivation for the present work.
As stated at the beginning of this section, B\'{e}rard and Ram\'{\i}rez~\cite{BerRam10}\footnote{Actually, \cite{BerRam10} adopts a little different setting from the one-particle-per-site frog model on $\Z$ as follows:
Initially, every site of the left of $0$ has a random number of active frogs, and every site on non-negative integers has a common fixed number of sleeping frogs.
Although it seems that arguments used in \cite{BerRam10} also run along the same lines for the discrete-time, one-particle-per-site frog model on $\Z$, we use a completely different approach in the present article to complete the upper tail large deviation estimate.
} studied large deviation principles for the continuous-time frog model on $\Z$.
In particular, they observed that the slowdown phenomenon for the propagation of active frogs, i.e., the upper tail large deviation probability for the first passage time decays slower than exponential~\cite[Theorem~2]{BerRam10}:
for any $\xi>0$,
\begin{align}\label{eq:BR_slowdown}
    e^{-t^{1/2+o(1)}}
    \leq \P\bigl( \T(0,\lfloor t \rfloor) \geq (\mu+\xi)t \bigr)
    \leq e^{-t^{1/3+o(1)}}
    \qquad \text{as } t \to \infty,
\end{align}
where $\lfloor t \rfloor$ is the greatest integer less than or equal to $t$.
However, this estimate is not optimal, and hence the present work is motivated by the desire to complete the upper tail large deviation estimate.

Let us prepare some notation to state our main result. First, for any $\xi>0$, denote by $\kC(\xi)$ the set of all functions $f:\R \to [0,\infty)$ satisfying the following three conditions:
\begin{itemize}
    \item $f$ is non-increasing and non-decreasing over $(-\infty,0]$ and $[0,\infty)$, respectively;
    \item $\lim_{x \to 0} f(x)=0$ holds;
    \item $\|f\|_\infty:=\sup_{x \in \R}f(x) \leq \xi$ and $\lim_{x \to \infty} f(x)=\xi$.
\end{itemize}
Next, for each $x \in \R$, let $\P_x^{\rm BM}$ be the law of one-dimensional Brownian motion starting at $x$.
In addition, write $(B_t)_{t \geq 0}$ for the trajectory of Brownian motion, and $\tau_y:=\inf\{ t \geq 0:B_t=y \}$ stands for the hitting time to $y \in \R$.
Then, for any $f \in \kC(\xi)$, \emph{the energy} of $f$ is defined by
\begin{align}\label{eq:def_energy}
    E(f):=-\int_\R \log \pp^{\rm BM}_x\bigl( \tau_y\geq f(y)-f(x) \quad \forall y\in \R \bigr) \,\dd x.
\end{align}
In addition, set for any $\xi>0$,
\begin{align}\label{deor}
    r(\xi):=\inf\bigl\{ E(f):f \in \kC(\xi) \bigr\},
\end{align}
which is the minimum energy over $\kC(\xi)$.

\begin{rem}
As we will see later (Step~2 in Section~\ref{subsect:sketch}), the key of the present work is to observe a localization phenomenon of the upper tail large deviation event $\{ \T(0,n) \geq (\mu+\xi)n \}$.
More precisely, $\{ \T(0,n) \geq (\mu+\xi)n \}$ is mainly affected by frogs siting on a bad interval, whose length has order $\sqrt{n}$ but whose passage time is more than $\xi n$.
The ensemble $\kC(\xi)$ can be referred as the set of functions recording all possible profiles of the space-time rescaled first passage time on a bad interval whose left endpoint is $0$, i.e., $f(u)=\T(0,u\sqrt{n})/n$ for $u \in \R$.
Moreover, as explained Section~\ref{Section: 4 Heuristic}, given $f\in \kC(\xi)$, the event that the scaled passage time $\T(x\sqrt{n},y\sqrt{n})/n$ is approximated by the $(f(y)-f(x))_+$ has probability roughly equal to $e^{-\sqrt{n}E(f)}$.
Hence, we refer to $E(f)$ as the energy of $f$, following the custom of statistical mechanics.
\end{rem}


We are now in a position to state our main result.
For the discrete-time, one-particle-per-site frog model on $\Z$, the following theorem completely provides the upper tail large deviation.

\begin{theo}\label{mth1}
Let $r_*:=r(1)$ and $\mu:=\mu(1)$.
Then, $r_*$ is positive and finite, and the first passage time of the frog model on $\Z$ satisfies the upper tail large deviation with speed $\sqrt{n}$ and rate function $\xi \longmapsto r_*\sqrt{\xi}$ ($\xi>0$).
More precisely,
for all $\xi>0$,
\begin{align*}
    \lim_{n \rightarrow \infty} \frac{1}{\sqrt{n}} \log\P\bigl( {\rm T}(0,n) \geq (\mu+\xi) n \bigr)
    = -r_* \sqrt{\xi}.
\end{align*}
\end{theo}

\subsection{Related works}\label{subsect:related}
Let us finally comment on earlier literature for the frog model.
The recurrence/transience problem was the first published result on the frog model.
Telcs and Wormald~\cite[Section~{2.4}]{TelWor99} treated the one-particle-per-site frog model and proved that it is recurrent for all $d \geq 1$, i.e., almost surely, active particles infinitely often visit $0$ (the frog model is said to be transient if it is not recurrent).
This result was developed into the relation between the strength of transience for a single random walk and the number of frogs.
Actually, Popov~\cite{Pop01} introduced the frog model with random initial configurations and exhibited phase transitions of the recurrence and transience in terms of the density of initial configuration.
After that, Alves et al.~introduced the frog model on an arbitrary graph with random initial configurations and random lifetimes, and studied phase transitions of the survival of frogs and recurrence/transience (see \cite{AlvMacPop02} and \cite{Pop03} for more details).
Recently, Kosygina and Zerner~\cite{KosZer17} obtained a zero-one law of recurrence and transience for the frog model with random initial configurations.
Furthermore, in \cite{DobGanHofPopWei18,GanSch09} and \cite{HofJohJun16,HofJohJun17,MicRos19_arXiv,MicRos20,MulWie20}, the recurrence/transience problem is also studied for the frog model with drift (this means that active frogs perform asymmetric random walks) and on (some $d$-ary, Galton--Watson and non-amenable) trees, respectively.

The spread of active frogs, which is the interest of the present paper, was first investigated for the frog model on $\Z^d$ in the discrete- and continuous-time settings (see \cite{AlvMacPop02} and \cite{RamSid02}).
More precisely, they showed that the asymptotic behavior of the first passage time can be controlled by the time constant, which is a (nonrandom) norm on $\R^d$ (see \eqref{eq:shape_thm}).
Hence, the central limit theorem and the large deviation principle are the next step to understand the behavior of the first passage time in details.
However, there are a few results for these topics.
In the multi-dimensional case, the authors proved that the first passage time has a sublinear variance and satisfies a concentration inequality, and its tail probability decays sub-exponentially (see \cite{CanNak19} and \cite{Kub19}).
These results give some clues for the central limit theorem and the large deviation principle for the first passage time, but are not enough to solve these problems completely.
Moreover, since the main tools used in \cite{CanNak19} and \cite{Kub19} are percolation arguments, those do not work in the one-dimensional case.
This means that we need completely different approaches to study the central limit theorem and the large deviation principle in the frog model on $\Z$.
Actually, Ram\'{\i}rez et al.~\cite{BerRam10,ComQuaRam07} discussed the central limit theorem and the large deviation principle for the frog model on $\Z$ by using a renewal structure, which is developed in the study of random walks in random environments.
Moreover, in the present article, we complete the upper tail large deviation estimate by using the energy $E(f)$, which is never seen in previous works.

In forthcoming papers, we will study large deviations for the first passage time in higher dimensions.
Then, for the upper tail large deviation, it is also useful to observe a localization phenomenon:
let $\mathbf{e}_1$ be the first coordinate vector of $\Z^d$, and the upper tail large deviation event $\{ \T(0,n\mathbf{e}_1) \geq (\mu+\xi)n \}$ is affected by frogs sitting on a bad region which is the ball centered at $n\mathbf{e}_1$ and of radius $\kO(\sqrt{n})$.
Due to the geometry of the bad ball, the results in higher dimensions are different.
Particularly, when $d=2$ (resp.~$d\geq 3$), the speed is $n/\log n$ (resp.~$n$) and the rate function is linear $\xi \mapsto c_d \xi$ (where $c_d$ is a constant depending on $d$).
In contrast, it appears that the lower tail large deviation event $\{\T(0,n\mathbf{e}_1) \leq (\mu-\xi)n\}$ is affected by overall frogs and the lower tail large deviation probability decays exponentially regardless of the dimension.

Similar localization phenomena have been observed in other models, including First-passage percolation and the chemical distance in percolation. In a study of First-passage percolation with weights under tail estimates, Cosco and Nakajima \cite{CN23} established a specific rate function for upper tail large deviations, known as the discrete p-capacity. Furthermore, Dembin and Nakajima \cite{DN23+} demonstrated the existence of the rate function for upper tail large deviations of the chemical distance in super-critical percolation when the dimension is three or higher. This is characterized by a space-time cut-point that all paths between the endpoints must pass through later than a specified time. These upper tail large deviations share similarities in the sense that the passage times are abnormally large due to the environments surrounding the endpoints. In our current study of the one-dimensional frog model, we also confirm the appearance of localization phenomena on the event of upper-tail large deviations. However, we do not know its location, which makes the structure complicated. Moreover, the analysis of the models mentioned above heavily depends on the slab argument, a concept introduced by Kesten, which is not applicable in our current study due to its one-dimensional nature. Instead, we introduce a new argument of covering process (see Section~\ref{sec:propts2} for details), which will be also used for upper tail large deviations in the two-dimensional frog model in the forthcoming paper.

\subsection{Sketch of proof}\label{subsect:sketch}
In this subsection, we summarize the main steps in the proof of Theorem~\ref{mth1}.
The symbol $o_M(1)$ stands for some constants satisfying $\lim_{M\to\infty}\limsup_{n\to\infty}|o_M(1)|=0$, which may change from line to line.

\vspace{1em}

\noindent
\textbf{Step~0: Scaling invariance of energy.} In Lemma \ref{thm2} (see Section~\ref{subsect:BMSRW} below), we demonstrate the scaling invariance of the rate function taking advantage of the scaling invariance of Brownian motion:
\ben{
r(\xi) := \inf_{f\in\kC(\xi)}E(f) =r_*\sqrt{\xi}, \qquad r_*:=r(1).
}

\vspace{1em}

\noindent
\textbf{Step~1: Localization of upper tail large deviation event.}
Let us next observe that a certain localization phenomenon affects the upper tail large deviation event.
Intuitively, a good strategy to delay the transmission on $\Z$ is to retard the propagation of active frogs on a \textit{bad interval} whose length is of order $\sqrt{n}$, but the passage time is of order $n$: for all sufficiently large $M \in \N$, as $n \to \infty$,
\begin{align*}
    \frac{-1}{\sqrt{n}}
    \log \P(\T(0,n) \geq (\mu+\xi)n)
    \approx
    \frac{-1}{\sqrt{n}}
    \log \P\big( \T(0,\lfloor M\sqrt{n} \rfloor{\NK )} \geq \xi n \bigr) + o_M(1). 
\end{align*}
However, this strategy does not work directly.
Instead, we show the following inequalities in Proposition~\ref{prop:ts2} (see Section~\ref{sec:po1d} below), which are weaker versions of the above approximation:
\begin{align*}
    &\pp(\rmT(0,n)\geq (\mu+\xi)n)\geq \exp(o_M(1) \sqrt{n})  \, \pp\big(\rmT(0,\lfloor M\sqrt{n} \rfloor) \geq (\xi+o_M(1))n\big),\\
    &\pp(\rmT(0,n)\geq (\mu+\xi)n) \leq\exp(o_M(1)\sqrt{n})  \sup_{\substack{\xi_1,\ldots,\xi_M\geq 0 \\ \xi_1+\ldots+\xi_M=\xi}} \, \prod_{i=1}^{M} \pp\big(\rmT(0,\lfloor M\sqrt{n} \rfloor)\geq (\xi_i- o_M(1)/M)n\big).
\end{align*}
These inequalities tell us that the upper tail large deviation event can be localized around several bad intervals whose length is of order $\sqrt{n}$ and where the total passage time is approximately greater than $\xi n$.

\vspace{1em}

\noindent \textbf{Step~2: Upper tail estimate for the first passage time on the bad interval.}
Thanks to Step~1, it suffices to take care of the upper tail probability of the form $\P(\T(0,\lfloor M\sqrt{n} \rfloor) \geq \xi n)$.
Our task is now to prove that for all sufficiently large $M \in \N$, as $n \to \infty$,
\begin{align}\label{eq:step2}
    \frac{-1}{\sqrt{n}} \log \pp\big( \rmT(0,\lfloor M\sqrt{n} \rfloor)\geq \xi n \bigr)
    \approx \inf_{f\in\kC(\xi)}E(f) +o_M(1).
\end{align}
The heuristic ideas of the above approximation is as follows:
Let $f(u):={\rm T}(0,u\sqrt{n})/n$ be the space-time rescaled passage time on the bad interval.
Then, $\rmT(0, M\sqrt{n})\geq \xi n $ if and only if $f \in \kC_M(\xi):=\{f \in \kC(\xi): f(M) \geq \xi \}$.
In addition, it always holds by the triangle inequality that
\begin{align*}
  t(x,y) \geq {\rm T}(0,y)-{\rm T}(0,x)= n\bigl( f(y/\sqrt{n})-f(x/\sqrt{n}) \bigr) \quad \forall \, \, x, y \in \Z.
\end{align*}
Since $t(x,y)$'s are the hitting times of simple random walks, one can expect by Donsker's invariance principle that as $n \to \infty$,
\begin{align*}
    \P\bigl( \forall x,y\in \Iintv{-M\sqrt{n},M\sqrt{n}},\,t(x,y) \geq n(f(y/\sqrt{n})-f(x/\sqrt{n})) \bigr)
    \approx \exp\bigl( -\sqrt{n}(E(f) + o_M(1)) \bigr),
\end{align*}
where for any $a,b \in \R$, $\Iintv{a,b}$ is the set of integers on the interval $[a,b]$ (see  Section~\ref{Section: 4 Heuristic} for more detailed explanation of this approximation).
These observations imply \eqref{eq:step2}.
It should be noted that the final approximation is essentially heuristic in nature, even though we have chosen $f$ as a random quantity included in the class $\mathcal{C}_M(\xi)$ in our previous discussion.
Specifically, certain approximations are only valid when applied to step functions.
Consequently, we will establish the final outcome through a two-step proof.

\vspace{1em}

\noindent
{\bf Step 2a: Arising of energy functional.}
Let $\kC^{\rm Step}(\xi)$ be the set of all step functions in $\kC(\xi)$.
In Proposition~\ref{lem:tmxn} (see Section~\ref{sec:po1d} below), we estimate the left-hand side of \eqref{eq:step2} precisely:
\begin{align*}
    \limsup_{n\to\infty}\frac{-1}{\sqrt{n}}
    \log\P\bigl(\T(0,\lfloor M\sqrt{n} \rfloor) \geq \xi n \bigr)
    \leq \inf_{f\in\kC^{\rm Step}(\xi)}E(f) +o_M(1)
\end{align*}
and
\begin{align*}
    \liminf_{n \to\infty}\frac{-1}{\sqrt{n}}
    \log\P\left(\T(0,\lfloor M\sqrt{n} \rfloor) \geq \xi n\right)
    \geq \inf_{f\in\kC(\xi)}E(f) -o_M(1).
\end{align*}

\vspace{1em}

\noindent \textbf{Step 2b:	Energy approximation.} In Proposition~\ref{prop:s} (see Section~\ref{sec:po1d} below), we show  that the ground state energy in $\kC(\xi)$ can be approximated by the energy of step functions:
\ben{
\inf_{f\in\kC(\xi)}E(f)=\inf_{f\in\kC^{\rm Step}(\xi)}E(f).
}

\vspace{1em}

\noindent
\textbf{Conclusion:} It directly follows from Steps~0 and 2 that as $n \to \infty$,
\ben{
\frac{-1}{\sqrt{n}}\log\pp \left(\rmT(0,\lfloor M\sqrt{n} \rfloor) \geq \xi n\right) \approx r_*\sqrt{\xi} + o_M(1).
}
Plugging this into Step~1, we arrive at 
\al{
&\limsup_{n \to\infty}\frac{-1}{\sqrt{n}}\log\pp ({\rm T}(0,n) \geq (\mu+\xi)n) \leq  r_*\sqrt{\xi},\\
&\liminf_{n \to\infty}\frac{-1}{\sqrt{n}}\log\pp \left(\rmT(0,n) \geq (\mu+\xi)n\right)
\geq \liminf_{M \to \infty} \liminf_{n \to \infty} \left\{ \sup_{\substack{\xi_1,\ldots,\xi_M\geq 0 \\ \xi_1+\ldots+\xi_M=\xi}} r_*\bigl( \sqrt{\xi_1}+\ldots+\sqrt{\xi_M} \bigr) {\NK -} o_M(1) \right\}  \geq r_*\sqrt{\xi},
}
by using the inequality $\sqrt{x_1} + \ldots +\sqrt{x_m} \geq \sqrt{x_1 + \ldots + x_m}$.
Therefore, the proof of Theorem~\ref{mth1} is complete.

\subsection{Organization of the paper}\label{subsect:organization}
Let us describe how the present article is organized.
The rest of the paper is devoted to justifying the steps in Section~\ref{subsect:sketch}.
First of all, in Section~\ref{sect:prelim}, we summarize properties of the simple random walk, the Brownian motion and the frog model, which are used throughout this paper.
In particular, by using the scaling invariance of Brownian motion, Lemma~\ref{thm2} gives the scaling invariance of energy and completes Step~0 in Section~\ref{subsect:sketch}.

Section~\ref{sec:po1d} is devoted to the proof of our main result (Theorem~\ref{mth1}).
Proposition~\ref{prop:ts2} guarantees the validity of Step~1 in Section~\ref{subsect:sketch}.
In particular, we can also observe a slowdown phenomenon for the first passage time (see Proposition~\ref{prop:ts2}-(iii)).
This is a weaker version of Theorem~\ref{mth1} but tells us that $r_*=r(1) \in (0,\infty)$.
Furthermore, Proposition~\ref{lem:tmxn} and \ref{prop:s}
deal with Steps~{2a} and {2b} in Section~\ref{subsect:sketch}, respectively.
We just give the statements of the above propositions in Section~\ref{sec:po1d}, and complete the proof of Theorem~\ref{mth1} for now.

The aim of Section~\ref{sect:slowdown} is to show Proposition~\ref{lem:tmxn}.
This section consists of two parts: In Subsections~\ref{subs:liminf1d}, we optimize the energy for the lower bound of the upper tail probability (see Proposition~\ref{lem:tmxn}-(i)).
Note that at this stage, the energy on $\mathcal{C}^{\textrm{Step}}(\xi)$ (not $\mathcal{C}(\xi)$) is used for the lower bound (In Section~\ref{sec:props}, we check that the energy on $\mathcal{C}(\xi)$ is approximated by that on $\mathcal{C}^{\textrm{Step}}(\xi)$).
On the other hand, in Subsections~\ref{subs:liminf1d}, we optimally estimate the upper tail probability from above by using the energy on $\mathcal{C}(\xi)$ (see Proposition~\ref{lem:tmxn}-(ii)).

Section~\ref{sec:propts2} gives the proof of Proposition~\ref{prop:ts2}.
This proposition consists of parts (i)--(iii), and those proofs are mainly given in Sections~\ref{subsect:prop3.1-2}, \ref{subsect:prop3.1-3} and \ref{subsect:prop3.1-1}.
The proof of part~{(ii)} is the most difficult, and the key lemma is Lemma~\ref{lem:keyup} stated in Section~\ref{subsect:prop3.1-2}.
Since its proof is a little bit long, we postpone it into Section~\ref{polku}.

In Section~\ref{sec:props}, we prove Proposition~\ref{prop:s}, which guarantees that the energy on $\mathcal{C}(\xi)$ is approximated by that on $\mathcal{C}^{\textrm{Step}}(\xi)$.
More precisely, we shall use a delicate multi step deformation that gradually transforms a function in  $\mathcal{C}(\xi)$ to another in $\mathcal{C}^{\textrm{Step}}(\xi)$ with approximated energy.  

In Appendix, we prove some technical lemmas used in the paper. In particular, we show  a version of conditional FKG inequality for Brownian motion, which may be of independent interest.

We close this section with some general notation:
\begin{itemize}
\item For any $a \in \R$, $\lfloor a \rfloor$ is the greatest integer less than or equal to $a$.
In addition, $\lceil a \rceil$ is the smallest integer larger than or equal to $a$.

\item For any $a,b \in \R$, $\Iintv{a,b}$ is the set of integers on the interval $[a,b]$.

\item For any subset $A$ of $\Z$, denote by $|A|$ the cardinality of $A$.

\item For any $A,B \subset \R$, ${\rm d}(A,B)$ denotes the Euclidean distance between two sets $A$ and $B$:
\begin{align*}
    {\rm d}(A,B) :=\inf\{ |x-y|:x \in A,y \in B \}.
\end{align*}

\item For $x \in A \subset \Z$ and $y \in \Z$, we denote by $\rmT_A(x,y)$ the first passage time from $x$ to $y$ using only frogs inside $A$, or more formally
\begin{align*}
    \T_A(x,y):=\inf\left\{ \sum_{i=0}^{k-1}t(x_i,x_{i+1}):
    \begin{minipage}{12.5em}
        $k \geq 1$ and $x_1,\dots,x_{k-1} \in A$,\\
        and  $x_0=x$ and $x_k=y$
    \end{minipage}\right\}.
\end{align*}
Moreover, define for $A, U, V \subset \Z^d$,
\begin{align*}
    \T_A(U,V) := \inf \{\T_A(x,y): x \in U \cap A, y \in V\}.
\end{align*}
Note that we have $\T_A \geq \T$ for all subsets $A$ of $\Z^d$.

\item Throughout the paper, $c$, $c'$, $C$, $C'$, $c_i$ and $C_i$ ($i=0,1,\dots$) denote some constants with $0<c,c',C,C',c_i,C_i<\infty$.
If it is necessary to emphasize the dependence on some parameter $a$, then we write $c(a)$ for instance.
\end{itemize}

\section{Preliminaries}\label{sect:prelim}

In this paper, we use often properties for the Brownian motion, the simple random walk and the frog model.
Hence, this section summarizes those properties.
In particular, Lemma~\ref{thm2} proves the scaling invariance of energy, which corresponds to Step~0 in Section~\ref{subsect:sketch}.

\subsection{Some properties of simple random walk and Brownian motion}\label{subsect:BMSRW}
We start with some simple estimates for the hitting time and the range of the simple random walk.

\begin{lem} \label{lemrw}
For simplicity of notation, we use the notation $(S_i)_{i=0}^\infty$ to express the simple random walk $(S_i^0)_{i=0}^\infty$ starting at $0$.
Then, the following results hold:
\begin{itemize}
\item [(i)] Let $\kR_t$ be the range of the random walk up to time $t$, i.e., $\kR_t:=\{S_i:i\in \Iintv{0,t}\}$.
There exists a positive constant $c$ such that  for all $t\geq m^2$,
    \be{
    \pp(|\kR_t| \leq m) \leq \exp(-ct/m^2).
    }

\item[(ii)] For any $a >0$, there exists a constant $c=c(a) \in (0,1)$ such that for all $n\in \N$,
    \be{
    \sup_{ |x| \leq a \sqrt{n}} \pp(t(0,x) \geq n) \leq c.
    }
    
\item[(iii)] There exist positive constants $c$ and $C$ such that for all  $x \in \Z$ and $n\in\N$,
    \be{
    \pp(t(0,x) \leq n) \leq C \exp(-cx^2/n).
    }
\end{itemize}
\end{lem}
\begin{proof}
By the usual central limit theorem, 
 \be{
\lim_{k \rightarrow \infty} \pp(S_{k^2} \leq k)= \pp(Z \leq 1) <1, 
 }
where $Z$ is a standard normal random variable. Therefore, 
\be{
\sup_{k \geq 1} \pp(|\kR_{k^2}| \leq k) \leq  \sup_{k \geq 1} \pp(S_{k^2} \leq k) =:\alpha < 1.
}
This combined with the Markov property implies that for $t\geq m^2$,
\begin{align*}
    \P(|\kR_t| \leq m)
    \leq \P\Bigl( \bigl| \{ S_i:i \in \Iintv{jm^2+1,  (j+1)m^2} \} \bigr| \leq m \quad \forall j \in \Iintv{0, t/m^2} \Bigr)
    \leq \alpha^{t/2m^2},
\end{align*}
and part~(i) follows.
For part~(ii), we use Donsker's invariance principle: as $n \to \infty$,
\begin{align*}
  \sup_{ |x| \leq a \sqrt{n}} \pp(t(0,x) \geq n) =\P\bigl( t(0,\lfloor a\sqrt{n}\rfloor) \geq n \bigr) \longrightarrow \P_0^{\rm BM} (\tau_a \geq 1) \in (0,1),
\end{align*}
where $\P_0^{\textrm{BM}}$ is the law of Brownian motion starting at $0$ and $\tau_a$ is the hitting time of Brownian motion to $a$ (see also above \eqref{eq:def_energy}).
This implies part~(ii) immediately. 
Let us finally prove part~(iii).
From \cite[Proposition 2.1.2-(b)]{LL}, there exist constants $c$ and $C$ such that for all $x \in \Z \setminus \{ 0\}$ and $n \in \N$,
\begin{align*}
    \pp(t(0,x) \leq n) \leq \pp \left(\max_{1 \leq i \leq n} |S_i| \geq |x|\right) \leq C \exp(-cx^2/n),
\end{align*}
and part~(iii) follows.
\end{proof}

Let us next state some properties for the hitting time of the Brownian motion.

\begin{lem}\label{thm2} 
Given $\xi>0$ and $f \in \mathcal{C}(\xi)$, we define the rescaled function $f_\xi(x):=\xi^{-1}f(\sqrt{\xi}x)$.
Then, $f_\xi\in \mathcal{C}(1)$ and  $E(f)=\sqrt{\xi}E(f_\xi)$ for any $\xi>0$ and $f \in \mathcal{C}(\xi)$.
As a corollary, the function $r$ defined in \eqref{deor} satisfies
\begin{align*}
    r(\xi)=\sqrt{\xi}\,r(1) \quad \forall \xi>0.
\end{align*}
\end{lem}
\begin{proof}
Let $\xi>0$ and $f\in \mathcal{C}(\xi)$.
Since $\lim_{x \to \infty} f(x)=\xi$, we have $\lim_{x \to \infty} f_\xi(x)=1$, and $f_{\xi} \in \kC(1)$ holds.
Moreover,  by the scaling invariance of Brownian motion, i.e., the laws of $(B_t)_{t\geq 0}$ and $(B_{\xi t}/\sqrt{\xi})_{t\geq 0}$ coincide
(see for instance \cite[Lemma~9.4]{KarShr91_book}), we have 
	\begin{align*}
		-E(f)
		&= \int_\R \log \mathbb{P}_x^{\rm BM} \left( \tau_y \geq f(y)-f(x)\quad\forall y \in \R \right) \dd x\\
		&= \int_\R \log \mathbb{P}_{x/\sqrt{\xi}}^{\rm BM} \left( \tau_{y/\sqrt{\xi}}\geq \xi^{-1}(f(y)-f(x))\quad\forall y \in \R \right) \,\dd x.
	\end{align*}
	By the change of variables $t=x/\sqrt{\xi}$ and $s=y/\sqrt{\xi}$, the rightmost side above is equal to
	\begin{align*}
		&\sqrt{\xi} \int_\R
		\log \mathbb{P}_t^{\rm BM} \left( \tau_s \geq \xi^{-1} f(\sqrt{\xi}s)-\xi^{-1} f(\sqrt{\xi}t)\quad\forall s \in \R \right) \dd t\\
		&= \sqrt{\xi} \int_\R \log \mathbb{P}_t^{\rm BM}(\tau_s \geq f_\xi(s)-f_\xi(t)\quad\forall s \in \R) \,\dd t
		=- \sqrt{\xi}E(f_\xi).
	\end{align*}
	Therefore, we obtain the equality $E(f)=\sqrt{\xi}E(f_\xi)$, which also implies that $r(\xi) =\sqrt{\xi} r(1)$.
\end{proof}

\begin{lem}\label{lem:maxbrown}
We have for any $u>0$,
\begin{align*}
    \frac{{\rm d}}{ {\rm d} u}	\pp_0^{\rm BM}\left( \max_{0 \leq s \leq 1} B_s  \leq u \right)
    = \sqrt{\frac{2}{\pi}}e^{-u^2/2}. 
\end{align*}
As a consequence, for all $t,u>0$,
\begin{align*}
    \pp_0^{\rm BM} (\tau_u \geq t )= \sqrt{\frac{2}{\pi}}\int_0^{u/\sqrt{t}} e^{-s^2/2} \, \dd s\asymp 1\land \frac{u}{\sqrt{t}}.
\end{align*}
\end{lem}
\begin{proof}
The first statement is proved in \cite[page 96]{KarShr91_book}.
Hence, the scaling invariance of Brownian motion shows that for all $t,u>0$,
\begin{align*}
    \pp_0^{\rm BM} (\tau_u \geq t ) = 		\pp_0^{\rm BM} \left(\max_{0 \leq s \leq t} B_s \leq u \right) = \pp_0^{\rm BM} \left(\max_{0 \leq s \leq 1} B_s \leq \dfrac{u}{\sqrt{t}} \right) = \sqrt{\frac{2}{\pi}}\int_0^{u/\sqrt{t}} e^{-s^2/2} \, \dd s.
\end{align*}
The last asymptotic formula follows from that for all $x>0$,
\begin{align*}
    \frac{x \wedge 1}{3} \leq \int_0^{x \wedge 1} e^{-s^2/2} \, \dd s \leq \int_0^{x} e^{-s^2/2} \, \dd s \leq x \land \sqrt{\pi},
\end{align*}
and the proof is complete.
\end{proof}

In Step~2 of Section~\ref{subsect:sketch}, we use Donsker's invariance principle and approximate the trajectory of the simple random walk by that of the Brownian motion.
To do this procedure rigorously, for any function $f$ on $\R$ and $\epsilon>0$, we define two perturbed versions $f^{\pm,\epsilon}$ of $f$ by
\begin{align}\label{per:f}
    f^{+,\eps}(u)=\sup_{v\in [u-\eps,u+\eps]}f(v),\quad f^{-,\eps}(u)=\inf_{v\in [u-\eps,u+\eps]}f(v) \qquad \forall u \in \R.
\end{align}
In particular, the following lemmas play a key role in approximating the energy by using step functions in Step~{2a} of Section~\ref{subsect:sketch}.

\begin{lem}\label{lem:kmt}
Let $f$ and $g$ be bounded functions on $\R$ with $f(x) \leq g(x)$ for all $x \in \R$. For all positive constants $\epsilon_1,\epsilon_2,M$, 
 there exists $n_0 \in \N$ such that for all $n\geq n_0$ and $x \in \Iintv{-M\sqrt{n},M\sqrt{n}}$, 
\begin{align}\label{eq:kmt}
\begin{split}
	&\pp^{{\rm BM}}_{x/\sqrt{n}} \Bigl( \tau_v \geq f^{+,\epsilon_1}(v)-g(x/\sqrt{n})\quad\forall v \in [-M,M] \Bigr)-\epsilon_2\\
	&\leq \pp\Bigl( t(x,y) \geq n( f(y/\sqrt{n})-g(x/\sqrt{n}) )\quad\forall y \in \Iintv{-M\sqrt{n},M\sqrt{n}} \Bigr)\\
	&\leq \pp^{{\rm BM}}_{x/\sqrt{n}} \Bigl( \tau_v \geq f^{-,\epsilon_1}(v)-g(x/\sqrt{n})\quad\forall v \in [-M,M] \Bigr)+\epsilon_2.	
\end{split}
\end{align}
\end{lem}
\begin{proof}
Let $\epsilon_1,\epsilon_2,M>0$. Without loss of generality, we  assume that $g(x/\sqrt{n})=0$ by subtracting $g(x/\sqrt{n})$ from $f$ and $g$ if needed. 
Let $(B_t)_{t \geq 0}$ be the Brownian motion starting at $x/\sqrt{n}$, and  
 $(S_i)_{i=0}^\infty$ the simple random walk starting at $x$.
The Koml\'{o}s--Major--Tusn\'{a}dy approximation (see \cite[Lemma~17]{DenWac15} for instance) allows us to couple $(\frac{1}{\sqrt{n}}S_{\lfloor nt \rfloor})_{t\geq 0}$ and $(B_t)_{t\geq 0}$ on a common probability space with probability measure $\mathbf{P}$ in such a way that for all $n$ large enough depending on $\epsilon_1,\epsilon_2,f$,
\begin{align}\label{eq:KMTapprox}
	\mathbf{P}(\kE_n^c) <\eps_2,\text{ where }
	\kE_n:=\left\{ \sup_{0 \leq t \leq \|f\|_\infty} \left|B_t-\frac{1}{\sqrt{n}}S_{\lfloor nt \rfloor}\right|< \frac{\epsilon_1}{2} \right\}.
\end{align}

We first consider the second inequality in \eqref{eq:kmt}. Suppose that $\kE_n$ occurs, and assume further that for all $y \in \Iintv{-M\sqrt{n},M\sqrt{n}}$,
\begin{align}\label{additional assumption}
	t(x,y) \geq n f(y/\sqrt{n}).
\end{align}
If $y\geq x$, then \eqref{additional assumption} implies $\sup\{S_t:t\leq n f(y/\sqrt{n})\}\leq y$. 
Thus, by $\kE_n$, 
choosing $y:=\lfloor (v-(\epsilon_1/2))\sqrt{n}\rfloor$, one has for all $v \in [x/\sqrt{n}+\epsilon_1, M]$,
\be{
\sup \left\{ B_t: t \leq f^{-,\epsilon_1}(v) \right\}\leq \sup \left\{ B_t: t \leq f(y/\sqrt{n}) \right\} <\frac{y}{\sqrt{n}}+ \frac{\epsilon_1}{2}	\leq v, 
}
which implies $\tau_v \geq f^{-,\epsilon_1}(v).$ 
By symmetry, we have $\tau_v\geq  f^{-,\epsilon_1}(v)$ for all  $v \in [-M,x/\sqrt{n}-\epsilon_1]$. 
Observe that for all $v \in \R$ with $|v-x/\sqrt{n}| \leq \epsilon_1$, by the assumption that $f\leq g$, we have  $f^{-,\epsilon_1}(v) \leq f(x/\sqrt{n}) \leq g(x/\sqrt{n})=0$.
Hence, we have $\tau_v \geq f^{-,\epsilon_1}(v)$ for all $v \in [-M,M]$.
 In conclusion, 
\be{
\kE_n\cap \Bigl\{ t(x,y) \geq n f(y/\sqrt{n}) \quad\forall y \in \Iintv{-M\sqrt{n},M\sqrt{n}} \Bigr\} \subset \Bigl\{ \tau_v \geq f^{-,\epsilon_1}(v)\quad\forall v \in [-M,M] \Bigr\}.
}
This combined with \eqref{eq:KMTapprox} yields the second inequality of \eqref{eq:kmt}.

Next, we consider the first inequality in \eqref{eq:kmt}.
Suppose that $\kE_n$ occurs and that $\tau_v \geq  f^{+,\eps_1}(v)$ holds for all $v \in [-M,M]$. 
If $v\geq x/\sqrt{n}$, then $\sup\{B_t:t\leq n f^{+,\eps_1}(v)\}\leq v$. 
Thus, by  $\kE_n$,  for all $y\in\Iintv{x,M\sqrt{n}}$, with $v:=y/\sqrt{n}-\eps_1/2$, $$\sup \left\{ S_n: n \leq f(y/\sqrt{n}) \right\}< \sup \left\{ B_t: t \leq f^{+,\eps_1}(v) \right\}+\frac{\eps_1} 2 \leq v+\frac{\eps_1} 2	=\frac y {\sqrt{n}},$$ 
 which implies $t(x,y) \geq  f(y/\sqrt{n})$.
By symmetry, we have $t(x,y) \geq  f(y/\sqrt{n})$ for all $y\in\Iintv{-M\sqrt{n},x}$.
Putting things together, we reach 
\be{
\kE_n\cap \Bigl\{ \tau_v \geq f^{+,\epsilon_1}(v)\quad\forall v \in [-M,M] \Bigr\} \subset \Bigl\{ t(x,y) \geq n f(y/\sqrt{n}) \quad\forall y \in \Iintv{-M\sqrt{n},M\sqrt{n}} \Bigr\}.
}
This combined with \eqref{eq:KMTapprox} yields the first inequality of \eqref{eq:kmt}.
\end{proof}

\begin{lem}\label{lem:h_decreasing}
Fix $\xi>0$.
Let $f \in \kC^{\rm Step}(\xi)$ (which is the set of all step functions in $\kC(\xi)$) and let $\Gamma$ be the set of all the points of discontinuity of $f$.
Then, for any $u \in \Gamma^c$,
\begin{align*}
    \lim_{\delta \searrow 0}\pp^{\rm BM}_{u} \bigl( \tau_v \geq f^{+,\delta}(v)-f(u)\quad\forall v \in \R \bigr)
    = \pp^{\rm BM}_{u}(\tau_v \geq f(v)-f(u)\quad\forall v \in \R).
\end{align*}
\end{lem}

{
\begin{proof}
Since $f$ is bounded, we can take a constant $M$ such that the support of $f$ is in $[-M,M]$.
Hence, it suffices to show that for any $u \in \Gamma^c$,
\begin{align}\label{eq:newproof_goal}
    \lim_{\delta \searrow 0} \pp^{\rm BM}_{u} \bigl( \tau_v \geq f^{+,\delta}(v)-f(u)\quad\forall v \in [-M,M] \bigr)
    = \pp^{\rm BM}_{u}(\tau_v \geq f(v)-f(u)\quad\forall v \in [-M,M]).
\end{align}
To this end, we fix $u \in \Gamma^c$.
Due to the monotonicity of $f^{+,\delta}$ in $\delta$, the limit of the above expression exists and
\begin{align*}
    \lim_{\delta \searrow 0} \pp^{\rm BM}_{u} \bigl( \tau_v \geq f^{+,\delta}(v)-f(u)\quad\forall v \in [-M,M] \bigr)
    \leq \pp^{\rm BM}_{u}(\tau_v \geq f(v)-f(u) \quad \forall v \in [-M,M]).
\end{align*}
For the opposite inequality, let $\Gamma$ be the set of all the points of discontinuity of the function $f$.
Note that $\Gamma$ is finite since $f$ is a step function.
Moreover, set $f^+(x):=\lim_{\delta \searrow 0}f^{+,\delta}(x)$ and consider the event
\begin{align*}
    \kE:=\left\{ \forall w\in \Gamma,\,\tau_w\neq f^+(w)-f(u),\,\lim_{n\to\infty}\tau_{w\pm 1/n}=\tau_w \right\}.
\end{align*}
Clearly, $\P^{\rm BM}_{u}(\kE)=1$ holds due to the finiteness of $\Gamma$.
Hence, the desired opposite inequality follows once we prove that
\begin{align}\label{elimh}
    \bigl\{ \forall \delta>0,\,\exists v \in [-M,M],\,\tau_v< f^{+,\delta}(v)-f(u) \bigr\} \cap \kE
    \subset \{ \exists v\in [-M,M],\,\tau_v<f(v)-f(u) \}.
\end{align}
Indeed, \eqref{elimh} combined with the monotonicity of $f^{+,\delta}$ in $\delta$ implies that
\begin{align*}
    \lim_{\delta \searrow 0} \pp^{\rm BM}_{u} \bigl( \tau_v \geq f^{+,\delta}(v)-f(u)\quad\forall v \in [-M,M] \bigr)
    &= \pp^{\rm BM}_{u} \bigl( \bigl\{ \exists \delta>0,\,\tau_v \geq f^{+,\delta}(v)-f(u)\quad\forall v \in [-M,M] \bigr\} \cup \kE^c \bigr)\\
    &\geq \pp(\tau_v \geq f(v)-f(u) \quad \forall v\in [-M,M]),
\end{align*}
and the desired opposite inequality follows.

To prove \eqref{elimh}, suppose that the event in the left-hand side of \eqref{elimh} occurs.
Then, for each $\delta \in (0,1]$, we can take $v_\delta \in [-M,M]$ such that $\tau_{v_\delta}<f^{+,\delta}(v_\delta)-f(u)$.
By the compactness of $[0,1]$ and $[-M,M]$, there exist a sequence $(\delta_k)_{k=1}^\infty$ on $(0,1]$ and $v_*\in [-M,M]$ such that $\delta_k \searrow 0$ and $v_{\delta_k}\to v_*$ as $k \to \infty$.
If $v_* \notin \Gamma$, then $f^{+,\delta_k}(v_{\delta_k})=f(v_{\delta_k})$ holds for a sufficiently large $k$.
This means that we have $\tau_{v_{\delta_k}}<f(v_{\delta_k})-f(u)$, and thus the event in the right-hand side of \eqref{elimh} occurs.
Let us next treat the case where $v_*\in \Gamma$.
Since $f$ is a step function, $f^{+,\delta_k}(v_{\delta_k}) \leq f^+(v_*)$ holds for all large $k$.
Hence,
\begin{align*}
    \limsup_{k \to \infty} \tau_{v_{\delta_k}}
    \leq \limsup_{k \to \infty} f^{+,\delta_k}(v_{\delta_k})-f(u)
    \leq f^+(v_*)-f(u).
\end{align*}
This, together with the fact that 
$\tau_{v_*} \not= f^+(v_*)-f(u)$ and $\tau_{v_{\delta_k}} \to \tau_{v_*}$ as $k \to \infty$ on the event $\kE$, implies $\tau_{v_*}<f^+(v_*)-f(u)$.
Hence, there almost surely exists a (random) constant $\e>0$ such that $\tau_v<f^+(v_*)-f(u)$ for all $v \in [v_*-\e,v_*+\e]$.
Moreover, since $f$ is a step function and $v_* \in \Gamma$, there exists an open interval $I \subset [-M,M]$ such that $v_*$ is an endpoint of $I$ and $f(v)=f^+(v_*)$ for all $v \in I$.
Since $[v_*-\varepsilon,v_*+\varepsilon] \cap I$ is not empty, we can take $v \in [v_*-\varepsilon,v_*+\varepsilon] \cap I \subset [-M,M]$ such that
\begin{align*}
    \tau_v
    <f^+(v_*)-f(u)
    = f(v)-f(u),
\end{align*}
and the event in the right-hand side of \eqref{elimh} also occurs in the case where $v_* \in \Gamma$.
Therefore, \eqref{elimh} is proved.
\end{proof}
}

\subsection{Some prior estimates of the one-dimensional frog model}
In this part, we present some lemmas on large deviation behaviors of the first passage time of the frog model around the starting point.

\begin{lem} \label{lem:t1}
There exists $c \in (0,1/64)$ such that for any $\delta>0$, $A \geq 4c\sqrt{\delta}$, and $n$ large enough, we have
\ben{\label{lem:t1 (1)}
\P(|\rmB_{\delta n}| \leq  c \delta \sqrt{n}/A) \leq \exp(-A\sqrt{n}),
}
where for $t\geq 0$,
 \be{
 \rmB_t:=\{x \in \Z: \T(0,x) \leq t\}.
 }
Consequently, for any $\eta, A >0$ satisfying $A\geq 32 c \eta$,
\ben{\label{lem:t1 (2)}
    \P\Bigl(\rmT(0,\lf \eta\sqrt{n}\rf) \wedge \rmT(0,-\lf\eta \sqrt{n}\rf) \geq 2A \eta n/c \Bigr) \leq \exp(-A\sqrt{n}).
	}
\end{lem}
\begin{proof}
	
	For any $k\geq 1$, let us define 
	$${\rm T}_k:=\min\{t\in\N:  |\rmB_t|\geq 2^k\}.$$
By Lemma \ref{lemrw}-(i), there exists a positive constant $c_0$ such that  for $t\geq m^2$ 
	\ben{ \label{krtm}
		\pp(|\kR_t| \leq m) \leq \exp(-2c_0t/m^2),
	} 
where $\kR_t$ is the range of the simple random walk starting at $0$ up to time $t$ (see also the statement of part~{(i)} in Lemma~\ref{lemrw}).
Let	\be{
		\eps:=(c_0\delta/4A)^2, \qquad k_n:=\lf\log_2(\sqrt{\eps n})\rf.
	}
 	Next, we consider the filtration $(\kF_j)_{j\geq 1}$ defined by
	$$\kF_j:=\sigma(\rmT_1,\cdots,\rmT_j).$$
By the Markov inequality, for any $\alpha >0$
	\begin{align} \label{tnn2}
		\P(|\rmB_{\delta n}| \leq \sqrt{\eps n})&\leq\pp({\rm T}_{k_n}\geq \delta n) \leq e^{-\alpha  \delta n} \E e^{\alpha {\rm T}_{k_n}} 
		= e^{-\alpha  \delta n} \E \left[ \prod^{k_n}_{j=1} \E [e^{\alpha ({\rm T}_j-{\rm T}_{j-1})} \mid \mathcal{F}_{j-1}] \right].
	\end{align}
	Since $|\rmB_{{\rm T}_{j-1}}|\geq 2^{j-1}$, using \eqref{krtm},   we have for $t\geq 2^{2j}$
	\begin{align} \label{todtj}
		\pp({\rm T}_j-{\rm T}_{j-1}> t \mid \mathcal{F}_{j-1})&\leq \left( \pp ( |\kR_{t}| < 2^{j})\right)^{2^{j-1}} \leq e^{-c_0 t/2^j}.
	\end{align}
Taking $\alpha:=c_0/(2\sqrt{\eps n})=2A/(\delta\sqrt{n})$, we have $c_0/2^j \geq c_0/2^{k_n} \geq c_0/\sqrt{\e n}=2\alpha$ for all $j \in \Iintv{1,k_n}$.
Hence, for all $j \in \Iintv{1,k_n}$,
	\begin{align*}
		\E [e^{\alpha ({\rm T}_j-{\rm T}_{j-1})} \mid \mathcal{F}_{j-1}]& \leq  \alpha \int_{0}^\infty e^{\alpha t} \pp({\rm T}_j-{\rm T}_{j-1}> t \mid \mathcal{F}_{j-1})dt\\
		&\leq \alpha 2^{2j} e^{\alpha 2^{2j}}+ \alpha\int_{ 2^{2j}}^\infty e^{\alpha t} e^{-c_0 t/2^j} dt \leq e^{\alpha 2^{2j+1}} + \alpha \int_{ 2^{2j}}^\infty e^{-\alpha t} dt  \leq  2e^{\alpha  2^{2j+1}}.
	\end{align*}
	Thus, since $2^{2k_n+2}\leq 4\eps n \leq \delta n/4$ and $\alpha\delta/2=A/\sqrt{n}$ by the choice of $\eps,\alpha$ and  the condition $A \geq 4c\sqrt{\delta}$, we have for $n$ large enough depending on $c_0,\e,A$,
	\begin{align*}
		e^{-\alpha  \delta n} \E  \left(\prod^{k_n}_{j=1} \E [e^{\alpha ({\rm T}_j-{\rm T}_{j-1})} \mid \mathcal{F}_{j-1}] \right)  &\leq e^{-\alpha  \delta n}  \prod^{k_n}_{j=1} \left(e^{\alpha  2^{2j+1}}+1\right) \leq e^{-\alpha  \delta n} \prod^{k_n}_{j=1} 2e^{\alpha  2^{2j+1}}\\
		&\leq \exp \left( -\alpha \delta n + k_n \log{2} +  \alpha     2^{2k_n+2} \right)\\
                &\leq \exp(- \alpha \delta n/ 2 )=
\exp{(-A\sqrt{n})}.
	\end{align*}
This combined with \eqref{tnn2} yields \eqref{lem:t1 (1)}.

Next we treat \eqref{lem:t1 (2)}.
Given $\eta, A>0$, we set 
		$\delta := 2A\eta/c$. Since $A^2 \geq 32 A c\eta  = 16 c^2 \delta$ by the assumption $A \geq 32 c \eta$, Part (i)  gives  
	\be{
		\pp\bigl( |\rmB_{\delta n}| > 2 \eta \sqrt{n} \bigr) \geq 1- \exp(-A\sqrt{n}).
	}
	Moreover, if $|\rmB_{\delta n}| > 2 \eta \sqrt{n}$, then either $\rmT(0,\lf \eta \sqrt{n}\rf) < \delta n$ or $\rmT(0,-\lf \eta \sqrt{n}\rf) <\delta n$. Hence,
	\be{
		\pp\Bigl( \rmT(0,\lf \eta\sqrt{n})\rf) \wedge \rmT(0,-\lf\eta \sqrt{n}\rf)<\delta n \Bigr) \geq 1-\exp(-A\sqrt{n}),
	}
and the proof is complete.
\end{proof}

\begin{lem} \label{lem:taln}
The following results hold:
\begin{itemize}
\item [(i)] For any $\alpha>0$, there exists $c=c(\alpha)>0$ such that if $n \in \N$ is large enough, then for all $x\in\Iintv{-\sqrt{n},\sqrt{n}}$,
\begin{align*}
    \pp\Bigl( {\rm T}_{[-\sqrt{\alpha n}, \sqrt{\alpha n}] }(0,x) \geq  n \Bigr) \leq \exp(-c\sqrt{n}).
\end{align*} 

\item[(ii)] There exists a universal constant $c>0$ such that for any $\alpha, \beta>0$, if $n$ is sufficiently large depending on $\alpha$ and $\beta$, then
\begin{align*}
    \pp\bigl( \T(0,\lfloor \alpha \sqrt{n} \rfloor) \geq \beta n \bigr)
    \leq \exp\bigl( -c\alpha^{-1}\beta\sqrt{n} \bigr).
\end{align*}
\end{itemize}
\end{lem}
\begin{proof}
First of all, let $c\in(0,1/64)$ be the constant as in Lemma \ref{lem:t1} and recall $\rmT_{k}=\inf\{t:|\rmB_t|\geq 2^{k}\}$.
Fix $\alpha >0$ and $x \in \Z$ with $|x| \leq \sqrt{n}$, and define
\be{
\e := \min \{\alpha, c^2/4\}, \quad k_n:= \lf\log_2\sqrt{ \e n}\rf.
}
Then, for all $n$ sufficiently large (independently of $c$, $\alpha$ and $\eps$),
\ben{ \label{tkne}
\pp({\rm T}_{k_n} >  n/2) \leq  \pp(|\rmB_{n/2}|\leq \sqrt{ \e n}) \leq \pp(|\rmB_{n/2}|\leq c\sqrt{ n}/2) \leq \exp(-\sqrt{n}),
}
where we have used Lemma~\ref{lem:t1} with $\delta=1/2$, $A = 1\geq 4c\sqrt{\delta}$ in the last inequality.
Note that, at the time $\T_{k_n}$, the number of active frogs is larger than $\sqrt{\e n}/2$ and  the set of sites visited by active frogs is a subset of $\Iintv{-\sqrt{\e n},\sqrt{\e n}}$.
Therefore, by the strong Markov property, we have for any $x \in \Z$ with $|x| \leq \sqrt{n}$,
\begin{align*}
    \pp\Bigl( \T_{[-\sqrt{\e n}, \sqrt{\e n}]}(0, x)\geq  n,
    \T_{k_n}\leq n/2 \Bigr)
    \leq \left\{ \max_{|z|\leq \sqrt{\e n}}\pp(t(z,x) \geq n/2) \right\}^{\sqrt{\e n}/2}\leq \exp(-c'\sqrt{n}),
\end{align*}
for some constant $c'=c'(\e)>0$.
Here the last inequality follows from Lemma~\ref{lemrw}-(ii) and that $\max \{|z-x|:|z| \leq \sqrt{\e n}, |x| \leq \sqrt{n} \} \leq (1+\sqrt{\e}) \sqrt{n}$. Combining this with \eqref{tkne} yields that
\begin{align*}
    \pp\Bigl( \T_{[-\sqrt{\e n}, \sqrt{\e n}]}(0, x)\geq n \Bigr) \leq \exp(-\sqrt{n}) + \exp(-c'\sqrt{n}).
\end{align*}
Hence, part~{(i)} of the lemma follows since $\rmT_{[-\sqrt{\e n}, \sqrt{\e n}]}(0,x) \geq \rmT_{[-\sqrt{\alpha n}, \sqrt{\alpha n}]}(0,x)$ by $\varepsilon \leq \alpha$.

Next, we consider part~{(ii)}.
Fix $\alpha,\beta>0$.
Then, letting $n$ large enough and taking $m:=\lceil \alpha^2 n \rceil$ and $\gamma:=\beta/(2\alpha^2)$, one has
\begin{align*}
    \pp\bigl( \T(0, \lfloor \alpha \sqrt{n} \rfloor)\geq \beta n \bigr)
    \leq \pp\bigl( \rmT(0,  \lfloor \sqrt{m} \rfloor) \geq  \gamma m \bigr).
\end{align*}
Hence, for part~{(ii)}, it suffices to check that there exists a positive constant $c$ (which is independent of $\alpha$ and $\beta$) such that if $n$ is large enough depending on $\alpha$ and $\beta$, then
\begin{align}\label{eq:tsn_goal}
    \pp\bigl( \T(0,  \lfloor \sqrt{m} \rfloor)\geq  \gamma m \bigr)
    \leq \exp(-c\alpha^{-1}\beta\sqrt{n}).
\end{align}
Since $\T_I(0,\cdot) \geq \rmT(0,\cdot)$ for all $I \subset \Z$, part~(i) implies that there exists a positive constant $c_1$ (which is independent of $\alpha$ and $\beta$) such that for all $\ell$ sufficiently large,
\begin{align}\label{tsn}
    \pp\bigl( \T(0,\lfloor \sqrt{\ell} \rfloor) \geq \ell \bigr) \leq \exp(-c_1\sqrt{\ell}).
\end{align}
In the case $\gamma \geq 1$, \eqref{eq:tsn_goal} is a direct consequence of \eqref{tsn}.
However, \eqref{eq:tsn_goal} is not trivial in the case $\gamma<1$ since $\{ \T(0,\lfloor \sqrt{m} \rfloor)\geq \gamma m \} \supset \{ \T(0,\lfloor \sqrt{m} \rfloor)\geq m \}$.
To overcome this problem, we use the subadditivity to obtain
\begin{align*}
    \T(0,\lfloor \sqrt{m} \rfloor)
    \leq \sum_{i=0}^{K-1} \T\Bigl( i \lfloor (\gamma/2)\sqrt{m} \rfloor,(i+1)i \lfloor (\gamma/2)\sqrt{m} \rfloor \Bigr)
\end{align*}
with $K:=\lceil 2/\gamma \rceil$.
Therefore, the union bound, the translation invariance and \eqref{tsn} show that there exists a positive constant $c$ (which is independent of $\alpha$ and $\beta$) such that if $n$ is large enough depending on $\alpha$ and $\beta$, then
\begin{align*}
    \pp\bigl( \T(0,\lfloor \sqrt{m} \rfloor) \geq \gamma m \bigr)
    &\leq K \times \pp\Bigl( \T(0,\lfloor (\gamma/2)\sqrt{m} \rfloor) \geq \frac{\gamma m}{K} \Bigr)\\
    &\leq K \times \pp\Bigl(\T(0,\lfloor \sqrt{(\gamma^2/4)m} \rfloor) \geq (\gamma^2/4)m \Bigr)\\
    &\leq \Bigl( \frac{2}{\gamma}+1 \Bigr)\exp\Bigl( -\frac{c_1}{4}\alpha^{-1}\beta \sqrt{n} \Bigr)
    \leq \exp(-c\alpha^{-1}\beta \sqrt{n}).
\end{align*}
This is the desired conclusion \eqref{eq:tsn_goal}, and part~{(ii)} is proved.
\end{proof}

\begin{lem}\label{Task3}
There exists a universal constant $\alpha>0$ such that  for all $L \in \N $ and $a,h>0$ satisfying $a^2 \geq h$, we have
\begin{align*}
    \P\bigl( \T_{[-a,L]}(0,L) \geq h \bigr)
    \leq \exp (\alpha h/a) \,\P(\T(0,L) \geq h).
\end{align*}
\end{lem}
\begin{proof}
We notice that
\begin{align*}
    \{ \T_{[-a, L]}(0,L) \geq h\} \cap \{\forall x<-a,\,t(x,0)\geq h\} \subset \{ \T(0,L) \geq h \},
\end{align*}
and that the two events appearing the left-hand side are independent.
Thus,
\begin{align*}
    \P(\forall x<-a, \,t(x,0) \geq h)
    \,\mathbb{P}\bigl( {\rm T}_{[-a,L]}(0,L) \geq h \bigr) \leq \mathbb{P}({\rm T}(0,L) \geq h).
\end{align*}
Hence, once it is proved that there exists a universal constant $\alpha>0$ such that for all $a,h>0$ with $a^2 \geq h$,
\begin{align}\label{eq:toxha_main}
    \P(\forall x<-a, \,t(x,0) \geq h)
    \geq \exp(-\alpha h/a),
\end{align}
the desired conclusion follows immediately.
To prove \eqref{eq:toxha_main}, we use Lemma~\ref{lemrw}-(iii):
there exist  positive constants $c$ and $C$ independent of $x,h$ such that
\begin{align*}
    \pp(t(x,0) \geq h) \geq 1 -C\exp(-cx^2/h).
\end{align*}
Moreover, if $x^2 \geq h$, then 
\be{
1 -C\exp(-cx^2/h) \geq \exp\left(-c'\exp(-cx^2/h) \right),
}
with some $c'=c'(c,C)>0$. 	Hence, since 	
$$\sum_{x = -a-1}^{-\infty} \exp(-cx^2/h)  \leq  \int_{a}^{\infty} \exp(-ct^2/h) dt \leq  \exp(-ca^2/h)\int_0^\infty \exp(-cs^2/h) ds\leq \sqrt{h/c}\, \exp(-ca^2/h),$$ we have
\begin{align*}
    \pp( \forall \, x < -a,\,t(x,0) \geq h) \geq \exp \left(-\sum_{x = -a-1}^{-\infty}c'\exp(-cx^2/h) \right) \geq \exp \left(-c''\sqrt{h}\exp(-ca^2/h)\right),
\end{align*}
where $c''=c'/\sqrt{c}$. In addition, using $e^{-t}\leq 1/(1+t)$ for $t>0$ and the Cauchy--Schwarz inequality, one has
	\be{
		\sqrt{h}\exp(-ca^2/h) \leq  \frac{\sqrt{h}}{1+(c a^2/h)} \leq  \frac{\sqrt{h}}{2\sqrt{c a^2/h}}= \frac{h}{2a\sqrt{c}}, 
	}
With these observations, for all $a,h>0$ with $a^2 \geq h$,
\begin{align*}
    \pp( \forall \, x < -a,\,t(x,0) \geq h)
    \leq \exp\left(-\frac{c''h}{2a\sqrt{c}} \right),
\end{align*}
and \eqref{eq:toxha_main} follows by taking $\alpha:=c''/(2\sqrt{c})$.
\end{proof}

The following lemma gives a rough upper tail large deviation estimate, which is a counterpart of the result for the continuous-time frog model \cite[Theorem~2-(b)]{BerRam10}.

\begin{lem} \label{lem:tgcn}
There exist positive constants $c$ and $C$ such that for all $n$ sufficiently large,
\be{
\pp(\rmT(0,n) \geq Cn) \leq \exp(-cn^{1/4}).
} 
\end{lem}
\begin{proof}
By the subadditivity and the fact that $\rmT_A \geq \rmT$ holds for all subsets $A$ of $\Z$, we have for all $C>0$,
\begin{align}\label{tocn}
    \pp\left( \rmT(0,n) \geq Cn \right)
    \leq n \pp\bigl( \rmT_{\Delta_n(0)}(0,1) \geq \sqrt{n} \bigr)+\pp \left( \sum_{i=0}^{n-1} \rmT_{\Delta_n(i)}(i,i+1)\1\{\T_{\Delta_n(i)}(i,i+1) \leq \sqrt{n}\} \geq Cn \right),
\end{align}
where $\Delta_t(i):=\Iintv{i-\sqrt{t}/4, i+\sqrt{t}/4}$ for $t>0$ and $i \in \N$.
Let us first estimate the first probability in the right-hand side of \eqref{tocn}.
Note that by Lemma~\ref{lem:taln}-(i), there exists a constant $c_0>0$ such that for all $t \geq 1$,
\begin{align}\label{lem2.7}
    \pp\bigl( \rmT_{\Delta_t(0)}(0,1) \geq t \bigr) \leq c_0^{-1}\exp(-c_0\sqrt{t}).
\end{align}
This implies that for all $n \in \N$,
\begin{align}\label{hao}
    \pp\bigl( {\rm T}_{\Delta_n(0)}(0,1) \geq \sqrt{n} \bigr)
    \leq \pp\Bigl( {\rm T}_{\Delta_{\sqrt{n}}(0)}(0,1) \geq \sqrt{n} \Bigr)
    \leq c_0^{-1}\exp(-c_0n^{1/4}).
\end{align}

We next estimate the second probability in the right-hand side of \eqref{tocn}.
Define for $i=0,\ldots n-1$,
\begin{align*}
    X_i:= \T_{\Delta_n(i)}(i,i+1) \1\{\rmT_{\Delta_n(i)}(i,i+1) \leq \sqrt{n}\}.
\end{align*}
By definition, $X_i$'s are bounded from above by $\sqrt{n}$ almost surely.
Moreover, the translation invariance and \eqref{lem2.7} imply that for all $i=0,\ldots,n-1$ and $t \in [1,\sqrt{n}]$,
\begin{align*}
    \pp(X_i \geq t)
    = \pp({\rm T}_{\Delta_n(0)} \geq t)
    \leq \pp({\rm T}_{\Delta_t(0)} \geq t)
    \leq c_0^{-1}\exp(-c_0 \sqrt{t}).
\end{align*}
This means that for all $i=0,\dots,n-1$,
\begin{align*}
    \E\biggl[ \exp\biggl( \frac{2c_0}{3}\sqrt{X_i} \biggr) \biggr]
    \leq 1+c_0^{-1}\int_1^\infty t^{-3/2}\,\dd t
    < \infty.
\end{align*}
With these observations, the mean-value theorem proves that there exists a constant $\alpha=\alpha(c_0)>0$ such that for all $i=0,\dots,n-1$,
\begin{align}\label{eosx}
\begin{split}
    \E\biggl[ \exp\biggl( \frac{c}{3n^{1/4}} X_i \biggr) \biggr]
    &\leq 1+\frac{c_0}{3n^{1/4}}\E\biggl[ X_i \exp\biggl( \frac{c_0}{3n^{1/4}}X_i \biggr) \biggr]\\
    &\leq 1+\frac{c_0}{3n^{1/4}}\E\biggl[ X_i \exp\biggl( \frac{c_0}{3}\sqrt{X_i} \biggr) \biggr]\\
    &\leq 1+\frac{6}{c_0n^{1/4}}\E\biggl[ \exp\biggl( \frac{2c_0}{3}\sqrt{X_i} \biggr) \biggr]
    \leq 1+\frac{\alpha}{n^{1/4}}.
\end{split}
\end{align}
Here we used the fact that $X_i \leq \sqrt{n}$ almost surely in the second inequality.
Furthermore, the third inequality follows from the fact that $\exp(c_0\sqrt{t}/3) \geq c_0^2t/18$ for all $t \geq 0$.
To estimate the sum of $(X_i)_{i=0}^{n-1}$ by using \eqref{eosx}, we divide $\Iintv{0,n-1}$ into $\lfloor \sqrt{n} \rfloor$ groups as follows:
\begin{align*}
    \Iintv{0,n-1}
    =\bigcup_{j=0}^{\lfloor \sqrt{n} \rfloor-1} \kI_j,\qquad
    \kI_j
    := \big  \{i \in \Iintv{0,n-1}:i \equiv j \pmod{\lfloor \sqrt{n} \rfloor} \big\}.
\end{align*}
Remark that $X_i$ depends only on the frogs $(S^x_{\cdot})_{|x-i|\leq \sqrt{n}/4}$.
Thus, for each $j \leq \lfloor \sqrt{n} \rfloor$, the random variables $(X_i)_{i \in \kI_j}$ are independent. 
Therefore, Markov's inequality and \eqref{eosx} show that if $h\geq 6\alpha |\kI_j|/c_0$, then for each $j=0,\dots,\lfloor \sqrt{n} \rfloor$,
\begin{align*}
    \pp\Biggl( \sum_{i \in \kI_j} X_i \geq h \Biggr)
    &= \pp\Biggl( \sum_{i \in \kI_j} \frac{c_0}{3n^{1/4}}X_i \geq \frac{c_0}{3n^{1/4}}h \Biggr)\\
    &\leq \exp\biggl( -\frac{c_0h}{3n^{1/4}} \biggr) \prod_{i \in \kI_j} \E\biggl[ \exp\biggl( \frac{c}{3n^{1/4}} X_i \biggr) \biggr]\\
    &\leq \exp\biggl( -\frac{c_0h}{3n^{1/4}} \biggr) \biggl( 1+\frac{\alpha}{n^{1/4}} \biggr)^{|\kI_j|}
    \leq \exp\biggl( -\frac{c_0h}{6n^{1/4}} \biggr).
\end{align*}
For each $j=0,\ldots, \lfloor \sqrt{n} \rfloor-1$, we have $12\alpha\sqrt{n}/c_0 \geq 6\alpha|\kI_j|/c_0$ due to $|\kI_j| \leq 2 \sqrt{n}$.
This enables us to use the above estimate with $h=12\alpha\sqrt{n}/c_0$ to obtain
\begin{align*}
    \pp\Biggl( \sum_{i=0}^{n-1} X_i \geq \frac{12\alpha}{c_0} n \Biggr)
    \leq \sum_{j=0}^{\lfloor \sqrt{n} \rfloor-1}\pp\Biggl( \sum_{i \in \kI_j} X_i \geq \frac{12\alpha}{c_0} \sqrt{n} \Biggr)
    \leq \sqrt{n}\exp(-\alpha n^{1/4}).
\end{align*}
Therefore, combining this estimate with \eqref{tocn} and \eqref{hao}, we get the desired conclusion.
\end{proof}
  
\section{Large deviation of the first passage time: Proof of the main theorem} \label{sec:po1d}
The aim of this section is to show Theorem~\ref{mth1}.
To this end, we fix $\xi>0$ and prove
\begin{align}
    &r_* \in (0,\infty),\label{rsnz}\\
    &\limsup_{n\to\infty}\frac{-1}{\sqrt{n}}\log{\mathbb{P}({\rm T}(0,n)\geq (\mu +\xi) n)} \leq r(\xi) =r_*\sqrt{\xi},\label{liminf1d}\\
    &\liminf_{n\to\infty}\frac{-1}{\sqrt{n}}\log{\mathbb{P}({\rm T}(0,n)\geq (\mu +\xi) n)}\geq r(\xi)=r_*\sqrt{\xi}. \label{limsup1d}
\end{align}
The proofs of these claims are based on the following three key propositions.
The first proposition treats Step~1 of Section~\ref{subsect:sketch}, which asserts a localization phenomenon of upper tail large deviation.
We also observe a slowdown phenomenon for the upper tail large deviation probability of the first passage time.

\begin{prop}\label{prop:ts2}
The following results hold:
\begin{itemize}
\item [(i)] For any $\xi,\delta>0$, there exists $c=c(\xi,\delta)>0$ such that for any $M\in \N$, if $n\in\N$  is  large enough depending on $M$, then
\begin{align*}
    \P(\T(0,n)\geq (\mu+\xi) n)
    \geq \P\bigl( \T(0,\lfloor M\sqrt{n} \rfloor) \geq (\xi+\delta) n \bigr)-\exp(-cn^{2/3}).
\end{align*}

\item[(ii)] For any $\xi,c,\delta,A>0$, there exists $M_0=M_0(\xi,c,\delta,A)>2+\xi$ such that for any $M\geq M_0$, if $n\in\N$ large enough depending on $M$,
\begin{align*}
    \P(\T(0,n)\geq (\mu+\xi) n)
    \leq \exp(-A\sqrt{n})+\exp(c\sqrt{n}) \,\sum_{m =1}^M \,\,\sum_{(h_i)_{i=1}^m \in \kH^{\delta}_{m,n}} \,\,\prod_{i=1}^m \pp\bigl( \T(0, \lfloor M\sqrt{n} \rfloor) \geq h_i \bigr),
\end{align*}
where
\begin{align}\label{hdmn}
    \kH^{\delta}_{m,n}
    := \left\{ (h_i)_{i=1}^m \in \N^m:  (\xi- \delta)n\leq \sum_{i = 1}^m h_i \leq  M n \right\}.
\end{align}

\item[(iii)](Slowdown phenomenon) For any $\xi >0$, there exists a positive constant $c=c(\xi)$ such that
\begin{align*}
    \pp(\T(0,n) \geq (\mu +\xi)n) \leq \exp(-c\sqrt{n}).
\end{align*}
\end{itemize}
\end{prop}

The following two propositions justify Steps~{2a} and {2b} of Section~\ref{subsect:sketch}, which surfaces energy functionals and claims that the ground state energies in $\mathcal{C}(\xi)$ and $\mathcal{C}^{\rm Step}(\xi)$ coincide, respectively
(see above \eqref{eq:def_energy} for the definition of $\mathcal{C}(\xi)$ and recall that $\kC^{\rm Step}(\xi)$ is the set of all step functions in $\kC(\xi)$).

\begin{prop}\label{lem:tmxn}
For any $\xi>0$, the following results hold:
\begin{itemize}
\item [(i)] We have
\begin{align*}
     \limsup_{M\to\infty}\limsup_{n\to\infty}
    \frac{-1}{\sqrt{n}}\log\P\bigl( \T(0,\lfloor M\sqrt{n} \rfloor) \geq \xi n \bigr)
	\leq \inf_{f\in\kC^{\rm Step}(\xi)}E(f).
\end{align*}

\item[(ii)] For any $M$ sufficiently large,
\begin{align*}
    \liminf_{n \to\infty}
    \frac{-1}{\sqrt{n}}\log\P\bigl( \rmT(0,\lfloor M\sqrt{n} \rfloor) \geq \xi n\bigr)
    \geq \inf_{f\in\kC(\xi)}E(f) - M^{-2}.
\end{align*}
\end{itemize}
\end{prop}

\begin{prop}\label{prop:s}
We have
\begin{align*}
    \inf\left\{E(f):f\in\kC^{\rm Step}(\xi) \right\}
    = \inf\left\{ E(f):f\in\kC(\xi) \right\}.
\end{align*}
\end{prop}

The proofs of Propositions~\ref{prop:ts2},  \ref{lem:tmxn} and \ref{prop:s} will be presented in the subsequent sections, and let us here complete the proofs of \eqref{rsnz}, \eqref{liminf1d} and \eqref{limsup1d}.

\begin{proof}[\bf Proof of (\ref{rsnz}) and (\ref{liminf1d})]
We first prove $r_*<\infty$.
Use the fact that $\log (1-t) \geq -t/2$ for $t \in (0,1)$ and Lemma~\ref{lem:maxbrown} to obtain
\begin{align*}
    \log\P_0^{\rm BM}(\tau_x \geq 1)
    \geq -\frac{1}{2} \P_0^{\rm BM}(\tau_x <1) \geq -\frac{1}{\sqrt{2\pi}}e^{-x^2/2}.
\end{align*}
Since the function $g(x)=\1\{ x \geq 1 \}$ belongs to $\mathcal{C}(1)$, we have
\begin{align}\label{rfi}
\begin{split}
    r_*
    &\leq E(g)=-\int_{-\infty}^1 \log \pp^{\rm BM}_x(\tau_1 \geq 1) \,\dd x
    = -\int_0^{\infty} \log\pp_0^{\rm BM}(\tau_x \geq 1) \,\dd x\\
    &\leq \frac{1}{\sqrt{2\pi}} \int_0^{\infty} e^{-x^2/2} \, \dd x=\frac{1}{2}<\infty.
\end{split}
\end{align}

Let us next prove \eqref{liminf1d}.
Lemma~\ref{thm2} and Propositions~\ref{prop:s}, \ref{lem:tmxn}-(i) and \ref{prop:ts2}-(i) show that for any $\delta>0$,
\begin{align}\label{uors}
\begin{split}
    r_*\sqrt{\xi+\delta}=\inf_{f \in \mathcal{C}^{\rm Step}(\xi+\delta)}E(f)
    &\geq \limsup_{M \to \infty}\limsup_{n\to\infty}\frac{-1}{\sqrt{n}} \log\P\bigl( \T(0,\lfloor M\sqrt{n} \rfloor) \geq (\xi+\delta)n \bigr)\\
    & \geq \limsup_{n\to\infty} \min \left\{ \frac{-1}{\sqrt{n}}\log\P(\T(0,n) \geq (\mu+\xi)n),cn^{1/6} \right\}\\
    &= \limsup_{n\to\infty} \frac{-1}{\sqrt{n}}\log\pp(\T(0,n) \geq (\mu+\xi)n),
\end{split}
\end{align}
and \eqref{liminf1d} follows by letting $\delta \searrow 0$.

It remains to show $r_*>0$.
Due to \eqref{uors} and Proposition~\ref{prop:ts2}-(iii) with $\xi=\delta=1$,
\begin{align*}
    -r_*\sqrt{2}
    \leq \liminf_{n\to\infty} \frac{1}{\sqrt{n}}\log\P(\T(0,n) \geq (\mu+1)n) \leq -c(1)<0,
\end{align*}
where $c(1)$ is a positive constant as in Proposition~\ref{prop:ts2}-(iii).
Therefore, $r_*>0$ holds and we obtain $r_* \in (0,\infty)$.
\end{proof}

\begin{proof}[\bf Proof of (\ref{limsup1d})]
By Proposition~\ref{prop:ts2}-(ii), for any $c,\delta,A>0$, there exists $M_0>2+\xi$ such that for all $M\geq M_0$,
\begin{align*}
    &\liminf_{n \to \infty \frac{-1}{\sqrt{n}}} \log \mathbb{P}({\rm T}(0,n)\geq (\mu+\xi) n)\\
    &\geq \min\left\{ A,-c+\min_{1 \leq m \leq M} \liminf_{n \to \infty} \min_{(h_i)_{i=1}^m \in \kH_{m,n}^{\delta}}
    \frac{-1}{\sqrt{n}}\sum_{i=1}^m \log\pp\bigl( \T(0,\lfloor M\sqrt{n} \rfloor) \geq h_i \bigr) \right\}.
\end{align*}
Thus, it suffices to prove that for any $\delta \in (0,\xi/2)$, $M\geq M_0$ and $m\in\Iintv{1,M}$,
\begin{align}\label{eq:tacxn_goal}
    \liminf_{n \to \infty} \min_{(h_i)_{i=1}^m \in \kH_{m,n}^\delta}
    \frac{-1}{\sqrt{n}}\sum_{i=1}^m\log \pp\bigl( \T(0,\lceil M\sqrt{n} \rceil) \geq h_i \bigr)
    \geq r_*\sqrt{\xi-2\delta}-2M^{-1}.
\end{align}
Indeed, by \eqref{eq:tacxn_goal},
\begin{align*}
    \liminf_{n \to \infty} \frac{-1}{\sqrt{n}} \log \mathbb{P}({\rm T}(0,n)\geq (\mu+\xi) n)
    \geq \min\left\{ A,-c+r_*\sqrt{\xi-2\delta}-2M^{-1} \right\}.
\end{align*}
Hence, letting $M \to \infty$, and then $c,\delta \searrow 0$ and $A \to \infty$ proves \eqref{limsup1d}.

To prove \eqref{eq:tacxn_goal}, we fix $M \geq M_0$ and $m \in \Iintv{1,M}$.
By Proposition~\ref{lem:tmxn}-(ii), if $n$ is large enough depending on $M,\delta$, for any $i\in \Iintv{1,M^2/\delta}$,
\begin{align*}
    \frac{-1}{\sqrt{n}}\log\P\bigl( \T(0,\lfloor M\sqrt{n} \rfloor) \geq i \delta n/M \bigr)
    \geq r_*\sqrt{i \delta/M}-2M^{-2}.
\end{align*}
Therefore, for $n$ large enough, we have
\begin{align*}
    \min_{(h_i)_{i=1}^m \in \kH_{m,n}^\delta} \frac{-1}{\sqrt{n}}\sum_{i=1}^m
    \log \P\bigl( \T(0,\lfloor M\sqrt{n} \rfloor) \geq h_i \bigr)
    &\geq \min_{(h_i)_{i=1}^m \in \kH_{m,n}^\delta}
    \frac{-1}{\sqrt{n}}\sum_{i=1}^m\log \P\left(\T(0,\lfloor M\sqrt{n} \rfloor) \geq \left\lfloor \frac{M h_i}{\delta n}\right\rfloor \frac{\delta n}{M} \right)\\
    &\geq \min_{(h_i)_{i=1}^m \in \kH_{m,n}^\delta}
    \sum_{i=1}^m \left(r_*\sqrt{\left\lfloor \frac{M h_i}{\delta n}\right\rfloor \frac{\delta}{M}}-2M^{-2}\right)\\
    &\geq r_* \times \min_{(h_i)_{i=1}^m \in \kH_{m,n}^\delta}
    \sum_{i=1}^m \sqrt{\left( \frac{h_i}{n}-\frac{\delta}{M} \right)_+}-2M^{-1}.
\end{align*}
Since $r_* \geq 0$ and $\sqrt{a_1}+\dots+\sqrt{a_\ell} \geq \sqrt{a_1+\dots+a_\ell}$ for any $a_1,\dots,a_\ell \geq 0$, this is further bounded from below by
\begin{align*}
    r_* \times \min_{(h_i)_{i=1}^m \in \kH_{m,n}^\delta} \sqrt{\sum_{i=1}^m\left( \frac{h_i}{n}-\frac{\delta}{M} \right)_+}-2M^{-1}
    \geq r_*\sqrt{\xi-2\delta}-2M^{-1},
\end{align*}
which is the desired bound \eqref{eq:tacxn_goal}.
\end{proof}


\section{Arising of energy functional: Proof of Proposition \ref{lem:tmxn}}\label{sect:slowdown}
In this section, we focus on proving that
\ben{\label{LDP around 0}
\lim_{M\to\infty}\lim_{n\to\infty}\frac{-1}{\sqrt{n}}
\log\mathbb{P}\bigl( {\rm T}(0,\lfloor M\sqrt{n} \rfloor) \geq \xi n \bigr)
= \inf_{f\in\kC(\xi)}E(f).
}
The section is divided into four parts. First,  we explain the heuristic of the proof in Section~\ref{Section: 4 Heuristic}. Second, we prove the lower bound (Proposition \ref{lem:tmxn}-(i)) in Section~\ref{subs:liminf1d}. Third, we prove the upper bound (Proposition \ref{lem:tmxn}-(ii)) in Section~\ref{subs:limsup1d}. Finally, we prove an auxiliary result, Lemma~\ref{lem:EpE}, that is used in the proof of Proposition~\ref{lem:tmxn}-(ii).

\subsection{Heuristic behind the proof} \label{Section: 4 Heuristic}
We explain here general ideas to prove Proposition \ref{lem:tmxn}. 
  Observe that
    \aln{
{\rm T}(0, \lfloor M\sqrt{n} \rfloor)\geq \xi n \quad & \Rightarrow  \ \quad  \exists \,  f\in \kC(\xi);~\forall x,y\in \Iintv{-M\sqrt{n},M\sqrt{n}},\,t(x,y)\geq n(f(x/\sqrt{n})-f(y/\sqrt{n})),\label{ass1}\\
{\rm T}(0, \lfloor M\sqrt{n} \rfloor)\geq \xi n \quad  &\Leftarrow  \ \quad  \exists \,  f\in \kC^{\rm Step}(\xi);~\forall x,y\in \Iintv{-M\sqrt{n},M\sqrt{n}},\,t(x,y)\geq n(f(x/\sqrt{n})-f(y/\sqrt{n})),\label{assB}
}
Indeed, if ${\rm T}(0, \lfloor M\sqrt{n} \rfloor)\geq \xi n$, letting $f(u):={\rm T}(0,\sqrt{n} u)/n$ for $u\in[-M,M]$ so that $f\in \kC(\xi)$, then 
$$t(x,y)\geq {\rm T}(0,y)-{\rm T}(0,x)= n(f(y/\sqrt{n})-f(x/\sqrt{n})),$$
where we have used  the triangular inequality.
This implies \eqref{ass1}. On the other hand, Lemma~\ref{lem:tfs_lower} below shows that if the right-hand side of \eqref{assB} holds for some step function $f \in \kC^{\rm Step}(\xi)$, then $\T(0, \lfloor M\sqrt{n} \rfloor)\geq \xi n$. Additionally, as it will be shown in Section~\ref{sec:props}, we can interchange the spaces $\kC(\xi)$ and  $\kC^{\rm Step}(\xi)$ in the computation of the rate function, and hence we essentially have the equivalence relation in \eqref{ass1}.

As a consequence,
\begin{align}
\label{local equiv hitting time}
  \P\bigl( \T(0, \lfloor M\sqrt{n} \rfloor)\geq \xi n \bigr)
  \approx \P\Bigl( \exists f\in \kC(\xi);~\forall x,y\in \Iintv{-M\sqrt{n},M\sqrt{n}},\,t(x,y)\geq n(f(y/\sqrt{n})-f(x/\sqrt{n})) \Bigr).
\end{align}
  By the Laplace principle, one  expects that the right-hand side of \eqref{local equiv hitting time} is approximated by
\begin{align*}
    &\sup_{f\in  \kC(\xi)}\P\Bigl( \forall x,y\in \Iintv{-M\sqrt{n},M\sqrt{n}},\,t(x,y)\geq n(f(y/\sqrt{n})-f(x/\sqrt{n})) \Bigr)\\
    &=\sup_{f\in  \kC(\xi)}\prod_{x\in  \Iintv{-M\sqrt{n},M\sqrt{n}}}\P\Bigl( \forall y\in \Iintv{-M\sqrt{n},M\sqrt{n}},\, t(x,y) \geq n(f(y/\sqrt{n})-f(x/\sqrt{n})) \Bigr).
\end{align*}
By Donsker's invariant principle, one further expects that
\begin{align*}
    &\P\Bigl( \forall y \in \Iintv{-M\sqrt{n},M\sqrt{n}},\,t(x,y)\geq n(f(x/\sqrt{n})-f(y/\sqrt{n})) \Bigr)\\
    &\approx \P_{x/\sqrt{n}}^{\rm BM}\bigl( \forall v\in [-M,M],\,\tau_v \geq f(v)-f(x/\sqrt{n}) \bigr).
\end{align*}
Hence, one would get
\al{
  &\prod_{x\in  \Iintv{-M\sqrt{n},M\sqrt{n}}}\P\Bigl( \forall y\in \Iintv{-M\sqrt{n},M\sqrt{n}},\, t(x,y) \geq n(f(y/\sqrt{n})-f(x/\sqrt{n})) \Bigr)\\
  &\approx \prod_{x\in  \Iintv{-M\sqrt{n},M\sqrt{n}}}\P_{x/\sqrt{n}}^{\rm BM}\bigl( \forall v\in [-M,M],\,\tau_v \geq f(v)-f(x/\sqrt{n}) \bigr)\\
  &= \exp{\left(\sqrt{n}\cdot \frac{1}{\sqrt{n}}\sum_{x\in  \Iintv{-M\sqrt{n},M\sqrt{n}}}\log{ \P_{x/\sqrt{n}}^{\rm BM}(\forall v\in [-M,M],\,\tau_v \geq f(v)-f(x/\sqrt{n}))}\right)}\\
  &\approx \exp\left(\sqrt{n} \int_{-M}^M\log{ \P_{u}^{\rm BM}\bigl( \forall v\in [-M,M],\,\tau_v \geq f(v)-f(u) \bigr)\dd u} \right)\\
  &= \exp\bigl( -\sqrt{n}(E(f) + o_M(1)) \bigr).
}
Combining these approximations we arrive at the desired equation \eqref{LDP around 0}.

In fact, some of them are not straightforward and  we will only prove that  
\al{
\inf_{f\in\kC(\xi)}E(f) -o_M(1) \leq \lim_{n\to\infty}\frac{-1}{\sqrt{n}}
\log\mathbb{P}\bigl( {\rm T}(0,\lfloor M\sqrt{n} \rfloor) \geq \xi n \bigr)
 \leq  \inf_{f\in\kC^{\rm Step}(\xi)}E(f) +o_M(1).
}

\subsection{Proof of Proposition~\ref{lem:tmxn}-(i)}\label{subs:liminf1d}
Let us start with a simple observation.

\begin{lem}\label{lem:tfs_lower}
Let $\xi >0$ and $f \in \mathcal{C}^{\rm Step}(\xi)$.
If $M$ is large enough, then for all $n \in \N$,
\begin{align*}
    &\log \P\Bigl( \rmT_{[-M\sqrt{n},M\sqrt{n}]}(0,\lfloor M\sqrt{n} \rfloor)\geq \xi n \Bigr)\\
    &\geq \sum_{x \in \Iintv{-M\sqrt{n},M\sqrt{n}}} \log \pp\Bigl( t(x,y) \geq n( f(y/\sqrt{n})-f(x/\sqrt{n}) )
    \quad \forall y\in \Iintv{-M\sqrt{n},M\sqrt{n}} \Bigr).
\end{align*}
\end{lem}
\begin{proof}
Fix $\xi >0$ and $f \in \mathcal{C}^{\rm Step}(\xi)$.
By definition, $f(0)=0$ holds and we can take a sufficiently large $M_0 \in \N$ such that $f$ is constant outside $[-M_0,M_0]$.
Note that $f(M_0)=\|f\|_\infty=\xi$, where $\|f\|_\infty=\sup_{u \in \R}f(u)$ (see also above \eqref{eq:def_energy}).
Let $M \geq M_0+1$ and $n \in \N$. Suppose that
\be{
t(x,y) \geq n( f(y/\sqrt{n})-f(x/\sqrt{n}) )\quad
	\forall x, \, y\in \Iintv{-M\sqrt{n},M\sqrt{n}}.
}
Then, since $\xi=f(M-1) \leq f(\lfloor M\sqrt{n} \rfloor/\sqrt{n}) \leq f(M)=\xi$ holds, one has for all sequence $(x_i)_{i=0}^\ell$ on $\Z$ with $x_0=0$ and $x_\ell= \lfloor M\sqrt{n} \rfloor$,
\begin{align*}
    \sum_{i=1}^\ell t(x_{i-1},x_i)\geq \sum_{i=1}^\ell n(f(x_i/\sqrt{n})-f(x_{i-1}/\sqrt{n}))=n(f(x_{\ell}/\sqrt{n})-f(0))= \xi n,
\end{align*}
which implies $\T(0,\lfloor M\sqrt{n} \rfloor)\geq \xi n$.
Therefore,
\begin{align*}
    &\log \P\Bigl( \T_{[-M\sqrt{n},M\sqrt{n}]}(0,\lfloor M\sqrt{n} \rfloor)\geq \xi n \Bigr)\\
    &\geq \log \P\Bigl( t(x,y) \geq n( f(y/\sqrt{n})-f(x/\sqrt{n}) )\quad\forall x,y \in \Iintv{-M\sqrt{n},M\sqrt{n}} \Bigr)\\
    &= \sum_{x \in \Iintv{-M\sqrt{n},M\sqrt{n}}}
		\log \P\Bigl( t(x,y) \geq n( f(y/\sqrt{n})-f(x/\sqrt{n}) ) \quad \forall y\in \Iintv{-M\sqrt{n},M\sqrt{n}} \Bigr),
\end{align*}
where the last equation follows from the independence of the simple random walks.
Hence, the lemma follows.
\end{proof}

We are now in a position to prove Proposition~\ref{lem:tmxn}-(i).

\begin{proof}[\bf Proof of Proposition~\ref{lem:tmxn}-(i)]
Using Lemma~\ref{Task3} with $a=M\sqrt{n}$, $L=\lfloor M\sqrt{n} \rfloor$, and $h=\xi n$, one has
\begin{align*}
    \frac{-1}{\sqrt{n}}\log\P\bigl( \T(0,\lfloor M\sqrt{n} \rfloor) \geq  \xi n \bigr)
    \leq \frac{\alpha \xi}{M}-\frac{1}{\sqrt{n}}\log\P\Bigl( \T_{[-M\sqrt{n},M\sqrt{n}]}(0,\lfloor M\sqrt{n} \rfloor)\geq \xi n \Bigr),
\end{align*}
where $\alpha$ is a universal, positive constant as in Lemma \ref{Task3}. Therefore, it suffices to prove
\begin{align}\label{tmcn}
    \limsup_{M\to\infty}\limsup_{n\to\infty}
    \frac{-1}{\sqrt{n}} \log\P\Bigl( \T_{[-M\sqrt{n},M\sqrt{n}]}(0,\lfloor M\sqrt{n} \rfloor) \geq \xi n\Bigr)
    \leq \inf_{f\in\kC^{\rm S
    tep}(\xi)}E(f).
\end{align}
Let $\eta>0$ and $M \in \N$ be sufficiently small and large, respectively.
We take $f_* \in \kC^{\rm Step}(\xi)$ such that 
\begin{align}\label{eq:f_star}
    E(f_*) \leq \inf_{f\in\kC^{\rm Step}(\xi)}E(f)+\eta.
\end{align}
Lemma~\ref{lem:tfs_lower} yields that for all $n \in \N$,
\ben{\label{eq:f*_lower}
\log \mathbb{P}\Bigl( {\rm T}_{[-M\sqrt{n},M\sqrt{n}]}(0,\lfloor M\sqrt{n} \rfloor)\geq \xi n \Bigr) \geq I_n,
}
where
\begin{align*}
    I_n :=\sum_{x \in \Iintv{-M\sqrt{n},M\sqrt{n}}}
    \log \P\Bigl( t(x,y) \geq n( f_*(y/\sqrt{n})-f_*(x/\sqrt{n}) )
		\quad \forall y \in \Iintv{-M\sqrt{n},M\sqrt{n}} \Bigr).
\end{align*}
We enumerate the points of discontinuity of $f_*$ as $u_{-\ell'}<u_{-1}<0<u_1<\dots<u_\ell$ and set $u_0:=0$.
Given $\epsilon\in(0,1)$ and $n \in \N$, let
\begin{align*}
	K^{(1)}_{\epsilon,n}:=(\sqrt{n}K^{(1)}_\epsilon) \cap \Z, \quad K^{(2)}_{\epsilon,n}:=(\sqrt{n}K^{(2)}_\epsilon) \cap \Z,
\end{align*}
where
\begin{align*}
	K^{(1)}_\epsilon:=[-M,M] \setminus \bigcup_{i=-\ell'}^{\ell} \left[ u_i-2\epsilon,u_i+2\epsilon \right],\quad
K^{(2)}_\epsilon:=\bigcup_{i=-\ell'}^{\ell} \left[ u_i-2\epsilon,u_i+2\epsilon \right].
\end{align*}
Now, decompose 
\begin{align*}
	I_n= I_{\epsilon,n}^{(1)}+I_{\epsilon,n}^{(2)},
\end{align*}
where
\begin{align*}
	&I_{\epsilon,n}^{(1)}:=\sum_{x \in K^{(1)}_{\epsilon,n}} \log \pp\Bigl(
		t(x,y) \geq n( f_*(y/\sqrt{n})-f_*(x/\sqrt{n}) ) \quad \forall y\in \Iintv{-M\sqrt{n},M\sqrt{n}} \Bigr),\\
	&I_{\epsilon,n}^{(2)}:=\sum_{x \in K^{(2)}_{\epsilon,n}} \log \pp\Bigl(
		t(x,y) \geq n( f_*(y/\sqrt{n})-f_*(x/\sqrt{n}) ) \quad \forall y\in \Iintv{-M\sqrt{n},M\sqrt{n}} \Bigr).
\end{align*}
Once we prove
\begin{align}\label{eq:I1}
    &\limsup_{\epsilon \searrow 0} \limsup_{n \to \infty}
    \frac{-1}{\sqrt{n}}I_{\epsilon,n}^{(1)}
    \leq -\int_{-M}^M\log \P^{\rm BM}_u(\tau_v \geq f_*(v)-f_*(u)\quad\forall v \in \R)\,\dd u,\\
    &\limsup_{\epsilon \searrow 0} \limsup_{n \to \infty} \frac{-1}{\sqrt{n}}I_{\epsilon,n}^{(2)} = 0,\label{eq:I2}
\end{align}
these combined with \eqref{eq:f_star} and \eqref{eq:f*_lower} imply
\begin{align*}
    &\limsup_{M\to\infty}\limsup_{n\to\infty}\frac{-1}{\sqrt{n}}\log \mathbb{P}\Bigl( {\rm T}_{[-M\sqrt{n},M\sqrt{n}]}(0,\lfloor M\sqrt{n} \rfloor)\geq \xi n \Bigr)\\
    &\leq \limsup_{M\to\infty}\limsup_{\epsilon\to 0}\limsup_{n\to\infty}\frac{-1}{\sqrt{n}} (I_{\epsilon,n}^{(1)}+I_{\epsilon,n}^{(2)})\\
    &\leq -\int_\R \log\P^{\rm BM}_u(\tau_v \geq f_*(v)-f_*(u)\quad\forall v \in \R)\,\dd u
    = E(f_*) \leq \inf_{f\in\kC(\xi)}E(f)+\eta,
\end{align*}
and hence \eqref{tmcn} follows by letting $\eta \searrow 0$.

It remains to prove \eqref{eq:f_star} and \eqref{eq:f*_lower}.
We first check \eqref{eq:I1}.
Let $\delta,\delta'\in (0,\e)$ be sufficiently small.
Given $u \in \R$, we define
$$h_*^\delta(u):=\pp^{\rm BM}_{u} \Bigl( \tau_v \geq f_*^{+,\delta}(v)-f_*(u) \quad \forall v \in \R \Bigr). $$
 Remark that since $f_*$ is a step function, we can take $M$ sufficiently large such that $f_*|_{(-\infty,-M+1]}$ and $f_*|_{[M-1,\infty)}$ are both constant functions, and thus for any $u\in \R$,
\ben{ \label{hsd}
h_*^\delta(u)=\pp^{\rm BM}_{u} \Bigl( \tau_v \geq f_*^{+,\delta}(v)-f_*(u) \quad \forall v \in [-M+1,M-1] \Bigr). }
Lemma~\ref{lem:kmt} with $f=g=f_*$ and $\epsilon_1=\delta$, $\epsilon_2=\delta'$ yields that if $n$ is large enough and $x \in K^{(1)}_{\epsilon,n}$, then
\begin{align}\label{hep1}
\begin{split}
	&\pp \Bigl( t(x,y) \geq n( f_*(y/\sqrt{n})-f_*(x/\sqrt{n}) )
	\quad \forall y \in \Iintv{-M\sqrt{n},M\sqrt{n}} \Bigr)
	\geq  h_*^\delta(x/\sqrt{n})-\delta'.
\end{split}
\end{align}
Moreover, due to the definition of $K^{(1)}_\epsilon$ and the fact that $\|f_*\|_\infty=\xi$, we have	for all $u \in K^{(1)}_\epsilon$,
\begin{align}\label{hgeqc}
	h_*^\delta(u)
	\geq \pp^{\rm BM}_{u}(\tau_{u-\epsilon} \wedge \tau_{u+\epsilon} \geq \xi)
	= \pp^{{\rm BM}}_{0}\left( \max_{0 \leq t \leq \xi} |B_t| \leq \epsilon \right)
	=: c(\xi,\epsilon) \in (0,1).
\end{align}
By \eqref{hep1} and \eqref{hgeqc}, for all $x \in K^{(1)}_{\epsilon,n}$,
\begin{align*}
	&\pp\Bigl( t(x,y) \geq n( f_*(y/\sqrt{n})-f_*(x/\sqrt{n}) )
		\quad \forall  y \in \Iintv{-M\sqrt{n},M\sqrt{n}} \Bigr)
	\geq  \left( 1-\frac{\delta'}{c(\xi,\epsilon)} \right)h_*^\delta(x/\sqrt{n}).
\end{align*}
It follows that
\begin{align}\label{eq:h2_approx}
	I_{\epsilon,n}^{(1)}
	\geq \sum_{x \in K^{(1)}_{\epsilon,n}} \log h_*^\delta(x/\sqrt{n})+(2M\sqrt{n}+1) \log\left( 1-\frac{\delta'}{c(\xi,\epsilon)} \right).
\end{align}
Notice that if $0 \leq h \leq \delta$, then for all $x \in K^{(1)}_{\epsilon,n}$ and $v \in \R$,
\begin{align*}
	f_*(x/\sqrt{n}+h)=f_*(x/\sqrt{n}),\quad
	f_*^{+,2\delta}(v+h) \geq f_*^{+,\delta}(v). 
\end{align*}
This together with \eqref{hsd} yields
\begin{align*}
	h_*^{2\delta}(x/\sqrt{n}+h)
 & \leq  \pp^{\rm BM}_{x/\sqrt{n}+h} \Bigl( \tau_{v} \geq f_*^{+,2\delta}(v)-f_*(x/\sqrt{n}+h)\quad\forall v \in [-M,M] \Bigr)\\
	&= \pp^{\rm BM}_{x/\sqrt{n}} \Bigl( \tau_{v-h} \geq f_*^{+,2\delta}(v)-f_*(x/\sqrt{n})\quad\forall v \in [-M,M] \Bigr)\\
	&= \pp^{\rm BM}_{x/\sqrt{n}} \Bigl( \tau_v \geq f_*^{+,2\delta}(v+h)-f_*(x/\sqrt{n})\quad\forall v \in [-M-h,M-h] \Bigr)\\
 &\leq \pp^{\rm BM}_{x/\sqrt{n}} \Bigl( \tau_v \geq f_*^{+,\delta}(v)-f_*(x/\sqrt{n})\quad\forall v \in [-M-h,M-h] \Bigr)
 \leq h_*^\delta(x/\sqrt{n}).
\end{align*}
Hence, if $n$ is large enough so that $1/\sqrt{n}\leq \delta<\eps$, since $ \log h_*^{2\delta}(u)=0$ for any $u\geq M-1$, then we have
\begin{align*}
	\sqrt{n} \int_{K^{(1)}_{\epsilon/2}} \log h_*^{2\delta}(u)\,\dd u
	&\leq \sum_{x \in K^{(1)}_{\epsilon,n}} \sup_{0 \leq h \leq 1/\sqrt{n}} \log h_*^{2\delta}(x/\sqrt{n}+h)\leq \sum_{x \in K^{(1)}_{\epsilon,n}} \log h_*^\delta(x/\sqrt{n}).
\end{align*}
This together with \eqref{eq:h2_approx} proves that
\begin{align}\label{eq:tt_case1}
\begin{split}
	\limsup_{n \to \infty} \frac{-1}{\sqrt{n}}I_{\epsilon,n}^{(1)}
	&\leq \limsup_{n \to \infty} \frac{-1}{\sqrt{n}}\sum_{x \in K^{(1)}_{\epsilon,n}} \log h_*^\delta(x/\sqrt{n})-2M\log\left( 1-\frac{\delta'}{c(\xi,\epsilon)} \right)\\
	&\leq -\int_{K^{(1)}_{\epsilon/2}} \log h_*^{2\delta}(u)\,\dd u-2M\log\left( 1-\frac{\delta'}{c(\xi,\epsilon)} \right).
\end{split}
\end{align}
Since $f_* \in \kC^{\rm step}(\xi)$, Lemma~\ref{lem:h_decreasing} and the monotone convergence theorem imply
\begin{align*}
    \lim_{\delta \searrow 0} \int_{K^{(1)}_{\epsilon/2}} \log h_*^{2\delta}(u)\,\dd u
    \geq \int_{-M}^M \log\pp^{\rm BM}_{u} \Bigl( \tau_v \geq f_*(v)-f_*(u) \quad \forall v \in \R \Bigr)\,\dd u.
\end{align*}
Combining this with \eqref{eq:tt_case1}, we have for any $\epsilon>\delta'>0$,
\begin{align*}
    \limsup_{n \to \infty} \frac{-1}{\sqrt{n}}I_{\epsilon,n}^{(1)}
    &\leq -\int_{-M}^M \log\pp^{\rm BM}_{u} \Bigl( \tau_v \geq f_*(v)-f_*(u) \quad \forall v \in \R \Bigr)\, \dd u
    -2M\log\left( 1-\frac{\delta'}{c(\xi,\epsilon)} \right).
\end{align*}
Therefore, \eqref{eq:I1} follows by letting $\delta' \searrow 0$ and then $\epsilon \searrow \infty$.

Let us finally check \eqref{eq:I2}.
Define for $ i\in \Iintv{-\ell',\ell}$ and $\delta >0$,
\begin{align*}
	L_{\delta,n}(i):=[\sqrt{n}(u_i-2\delta),\sqrt{n}(u_i+2\delta)] \cap \Z.
\end{align*}
Let $\delta\in (0,1)$ be small enough (depending on $f_*$) so that for any $n\in\N$, $(L_{\delta,n}(i))_{i \in \Iintv{-\ell' ,   \ell}}$ are disjoint.
Take an arbitrary $\eps \in (0,\delta/2)$ and note that $K^{(2)}_{\epsilon,n}=\bigsqcup_{i=-\ell'}^\ell L_{\epsilon,n}(i)$.
 
We first consider  $i \in \Iintv{1,\ell}$ and $x \in L_{\epsilon,n}(i)$. Since $f_*$ is increasing in $[0,\infty)$ and decreasing in $(-\infty,0]$ and satisfies $0 \leq f_*(u) \leq \xi$ for all $u \in \R$, we have
\bea{
\pp(t(x,0) \geq \xi n, t(x,y_x) \geq \xi n)\leq \pp\Bigl( t(x,y) \geq n( f_*(y/\sqrt{n})-f_*(x/\sqrt{n}) ) \quad \forall y \in \Iintv{-M\sqrt{n},M\sqrt{n}} \Bigr),
}
where $\Delta_n:=\lceil \delta \sqrt{n}\rceil$ and
\be{
y_x := \begin{cases}
    \lceil u_i \sqrt{n} \rceil &\textrm{if } (u_i-2 \epsilon) \sqrt{n} \leq x < u_i \sqrt{n},\\
\lceil u_i \sqrt{n}\rceil+\Delta_n &\textrm{if }  u_i \sqrt{n}  \leq x  \leq (u_i+2 \epsilon) \sqrt{n}.
\end{cases}
}
 The strong Markov property shows
\begin{align*}
    \pp(t(x,0) \geq \xi n, t(x,y_x) \geq \xi n)
    & =\pp_x(\tau_0  \geq \xi n, \, \tau_{y_x} \geq n\xi)\\
    &\geq \pp_x\Bigl( \tau_{\Delta_n} < \tau_{y_x},\max_{k\in \Iintv{0,\xi n}}|S_{\tau_{\Delta_n}+k}-S_{\tau_{\Delta_n}}|<\Delta_n  \Bigr)\\
    &=  \pp_x\big( \tau_{\Delta_n} < \tau_{y_x} \big) \,\pp \left(\max_{0 \leq k \leq n\xi}|S^0_k|< \Delta_n \right).
\end{align*}
A standard result for the simple random walk (see for instance \cite[(1.20)]{Law91_book}) and the fact $y_x \geq 2\Delta_n$ give
\begin{align*}
    \pp_x\big( \tau_{\Delta_n} < \tau_{y_x} \big) = \frac{y_x-x}{y_x-\Delta_n} \geq \frac{y_x-x}{\delta\sqrt{n}}. 
\end{align*}
Moreover, by  Donsker's invariance principle, for $n$ large enough,
\begin{align*}
	\pp \left(\max_{0 \leq k \leq n\xi}|S^0_k|<\Delta_n\right)
	&\geq \frac{1}{2}\pp\left(\sup_{0 \leq t \leq \xi}|B_t|<\frac{\delta}{2} \right)=:c(\xi,\delta)>0.
\end{align*}
Therefore,
\be{
\pp\Bigl( t(x,y) \geq n( f_*(y/\sqrt{n})-f_*(x/\sqrt{n}) ) \quad \forall y \in \Iintv{-M\sqrt{n},M\sqrt{n}} \Bigr) \geq c_* \frac{y_x-x}{\sqrt{n}},
}
with $c_*:=c_*(\xi,\delta):=c(\xi,\delta)/\delta$.
Hence,
\begin{align}\label{Delta n approx}
\begin{split}
    &\limsup_{n \to \infty}\frac{-1}{\sqrt{n}}\sum_{x \in K^{(2)}_{\e,n}(i)} \log \pp\Bigl( t(x,y) \geq n( f_*(y/\sqrt{n})-f_*(x/\sqrt{n}) )\quad\forall y \in \Iintv{-M\sqrt{n},M\sqrt{n}} \Bigr)\\
    &\leq -4\epsilon \log (c_*)-2 \liminf_{n \to \infty}\frac{1}{\sqrt{n}}\sum_{k=1}^{\lceil 2 \epsilon n \rceil} \log \left(\frac{k}{\sqrt{n}} \right)\\
    &= -4\epsilon \log (c_*)-2 \int_0^{2 \epsilon} \log t\, \dd t
    = -4\epsilon  (  \log(2 c_*\epsilon) -1).
\end{split}
\end{align}
In the case $i \in \Iintv{-\ell',1}$, the above argument also works by symmetry, and \eqref{Delta n approx} is valid for $i \in \Iintv{-\ell',1}$.
In the case $i=0$, setting for $x \in L_{\e,n}(0)$,
\begin{align*}
    y_x := \begin{cases}
    0 &\textrm{if }   x <0, \\
\Delta_n &\textrm{if }  x \geq 0,
\end{cases}
\end{align*}
one can apply the above argument again and obtain \eqref{Delta n approx} for $i=0$.
In conclusion, we reach
\begin{align*}
    0 \leq \limsup_{n \to \infty} \frac{-1}{\sqrt{n}}{I_{\eps,n}^{(2)}}
    \leq -4\epsilon (\ell+\ell'+1) \left( \log (2c_* \epsilon) -1 \right), 
\end{align*}
and letting $\epsilon \searrow 0$ proves \eqref{eq:I2}.
\end{proof}

\subsection{Proof of Proposition~\ref{lem:tmxn}-(ii)}\label{subs:limsup1d}
For any $\xi,M,\eta>0$, we define
\begin{align*}
	\kC_M(\xi)
		&:= \{ f \in \kC(\xi):f|_{(-\infty,-M]}\equiv \text{const} \text{ and } f|_{[M,\infty)}\equiv \text{const} \},\\
	\kC_{M,\eta}(\xi)
		&:= \{ f \in \kC(\xi):f|_{(-\infty,-M]}\equiv \text{const}, f|_{[M,\infty)}\equiv \text{const} \text{ and } f|_{[-\eta,\eta]}\equiv 0 \}.
\end{align*}
Recall the notations $f^{\pm,\delta}$ from \eqref{per:f}. For any $\xi>0$ and $f\in \kC(\xi)$, set
\begin{align*}
	E_M(f)&:=-\int_{-M}^M \log \pp^{\rm BM}_u\left(\tau_v \geq  f(v)-f(u)\quad\forall v \in \R \right) \dd u,\\
	E^+_{\delta,M}(f)&:=-\int_{-M}^M \log\pp^{\rm BM}_u\Bigl( \tau_v \geq f^{-,\delta}(v)-f^{+,\delta}(u)
		\quad \forall  v \in \R \Bigr) \dd u.
\end{align*} 
\begin{lem} \label{lem:EpE}
 For any  $\xi_0>0$, there exists $M_0\in\N$ such that for any $M\geq M_0$ and $\xi\in (0,\xi_0)$,
\begin{align*}
	\inf_{f\in \kC_{M}(\xi)}E_M(f) \geq \inf_{f\in \kC(\xi)}E(f)-M^{-2}.
\end{align*}
Furthermore, for all $M \geq 1$ and $\eta\in (0,1)$,
\begin{align*}
	\lim_{\delta\to 0} \inf_{f\in \kC_{M,\eta}(\xi)}E^+_{\delta,M}(f)=\inf_{f\in \kC_{M,\eta}(\xi)}E_M(f).
\end{align*}
\end{lem}

The proof of Lemma~\ref{lem:EpE} is postponed until Section~\ref{subsect:lem_Epe}, and we complete the proof of Proposition~\ref{lem:tmxn}-(ii) for now.

    \begin{proof}[\bf Proof of Proposition \ref{lem:tmxn}-(ii)]
Fix $\xi>0$ throughout the proof. By Lemma~\ref{lem:t1} with $2\eta$ in place of $\eta$ and $A=r(\xi)$,
there exists $C=C(\xi)>0$ such that for any $\eta\in(0,r(\xi))$, if $n \in \N$ is large enough, then
\begin{align}
&\pp\Bigl(\rmT(0,\lceil 2\eta \sqrt{n} \rceil) \wedge \rmT(0,-\lceil 2\eta \sqrt{n} \rceil) \geq C \eta n \Bigr)
	\leq e^{-r(\xi)\sqrt{n}}.\label{eq:local}
\end{align}
By Lemma~\ref{lem:taln}-(ii) with $\alpha=2M$ and $\beta=4M^2$, there exists a universal constant $c>0$ such that for any $M>0$, if $n \in \N$ is large enough, then
\begin{align}\label{eq:global}
&\mathbb{P}\Bigl( {\rm T}(0,\lfloor 2M\sqrt{n} \rfloor)\vee\rmT(0,-\lfloor 2M\sqrt{n} \rfloor)\geq 4M^2n \Bigr)
	\leq e^{-cM \sqrt{n}},
 \end{align}
Set $M_1:=|r(\xi)|/c$ and $\eta_0:=\min\{ \xi/(2C),1/4 \}$.
We fix $M \geq M_1$ and $\eta \in (0,\eta_0)$. For simplicity, we are not explicitly mentioning the dependence on 
$M$ and $\eta$ in the absence of any ambiguity. We define for  $n \in \N$,
\begin{align*}
	\mathcal{E}_n:=\left\{
	\begin{minipage}{19.5em}
	\begin{itemize}
	\setlength{\leftskip}{-2.3em}
		\item $\rmT(0,\lfloor M\sqrt{n} \rfloor) \geq \xi n$,
		\item $\rmT(0,\lfloor 2\eta\sqrt{n} \rfloor) \wedge \rmT(0,-\lfloor 2\eta\sqrt{n} \rfloor) \leq C\eta n$,
		\item $\rmT(0,\lfloor 2M\sqrt{n} \rfloor)\vee \rmT(0,-\lfloor 2M\sqrt{n} \rfloor) \leq 4M^2n$
	\end{itemize}
	\end{minipage}
	\right\}.
\end{align*}
First, assume
\begin{align}\label{eq:energy_key}
    \liminf_{n \to \infty} \frac{-1}{\sqrt{n}}
    \log\P(\mathcal{E}_n)
    \geq \inf_{f \in \kC_{M+2,\eta}(\xi-C \eta)} E_{M+2}(f),
\end{align}
and complete the proof of Proposition~\ref{lem:tmxn}-(ii).
By \eqref{eq:local} and \eqref{eq:global} and the choice of $M$, we have
\begin{align*}
	\liminf_{n \to\infty}\frac{-1}{\sqrt{n}}\log\pp \left(\rmT(0,\lfloor M\sqrt{n} \rfloor) \geq \xi n\right)
	&\geq \liminf_{n \to\infty}\frac{-1}{\sqrt{n}}\log\left\{ 2e^{-r(\xi)\sqrt{n}}+\pp(\mathcal{E}_n) \right\}\\
	&\geq \min\left\{ r(\xi), \,  \inf_{f \in \kC_{M+2,\eta}(\xi-C \eta)} E_{M+2}(f) \right\}.
\end{align*}
By Lemma~\ref{lem:EpE}, since $\kC_{M+2,\eta}(\xi-C \eta)\subset \kC(\xi-C \eta)$,   if $M \geq M_0(\xi)$ the constant as in this lemma,   then 
\begin{align*}
	  \inf_{f \in \kC_{M+2,\eta}(\xi-C \eta)} E_{M+2}(f)&\geq  \inf_{f \in \kC(\xi-C \eta)} E_{M+2}(f) \geq \inf_{f\in \kC(\xi-C\eta)}E(f)+M^{-2} =r(\xi-C\eta)-M^{-2}.
\end{align*}
Letting $\eta \searrow 0$ with continuity of $r$ (Lemma~\ref{thm2}), as desired in Proposition \ref{lem:tmxn}-(ii), we have  for $M\geq \max \{M_1,M_0\}$
\al{
\liminf_{n \to\infty}\frac{-1}{\sqrt{n}}\log\pp
\left(\rmT(0,\lfloor M\sqrt{n} \rfloor) \geq \xi n \right) 
\geq r(\xi) -M^{-2}.
} 

Our task is now to prove \eqref{eq:energy_key}.
To this end, let $\delta \in (0,\eta)$ be sufficiently small, and set $J:=2\lceil 2M/\delta \rceil$. 
For $n \in \N$ sufficiently large, we consider the sequence  $x_{i,n}:= \lfloor i(\delta/2)\sqrt{n} \rfloor$ for $i\in \Iintv{-J,J}$. Furthermore, the subset $\kA_n$ of $(\N_0)^{2J+1}$ is defined by
\begin{align*}
	\kA_n:=\left\{ (t_i)_{i=-J}^J \in (\N_0)^{2J+1}:\hspace{-2em}
	\begin{minipage}{7truecm}
	\begin{itemize}
		\item $t_{J/2} \geq \xi n$,
		\item $t_\sigma \wedge t_{-\sigma} \leq C\eta n$,
		\item $0 \leq t_i \leq 4M^2 n$ for all $i\in\Iintv{-J,J}$,
		\item $t_i \leq t_{i+1}$ for all $i\in\Iintv{0,J-1}$,
		\item $t_i \geq t_{i+1}$ for all $i\in\Iintv{-J,-1}$
	\end{itemize}
	\end{minipage}
	\right\},
\end{align*}
where $\sigma:=\lfloor 4\eta/\delta \rfloor$ and $C$ is the constant appearing in \eqref{eq:local}.
This describes the space induced by the configuration of $(\T(0,x_{i, n}))_{i=-J}^J$ conditioned on the event $\mathcal{E}_n$, and
note that $|\mathcal{A}_n|$ is at most $(4M^2 n)^{2J+1}$.
Hence,
\begin{align} \label{pent}
	\pp(\mathcal{E}_n)
	&\leq \sum_{(t_i)_{i=-J}^J \in \kA_n}\pp(\T(0,x_{i,n})=t_i\quad\forall i\in \Iintv{-J, J}) \notag \\
	&\leq (4M^2n)^{2J+1} \max_{(t_i)_{i=-J}^J \in \kA_n}\pp(\T(0,x_{i,n})=t_i\quad\forall i\in \Iintv{-J, J}).
\end{align}
For each $n \in \N$, let $(t_{i,n})_{i=-J}^J$ be an element of $\mathcal{A}_{n}$ attaining the above maximum.
To derive the desired bound \eqref{eq:energy_key}, we take the following steps (1) and (2):
\begin{enumerate}
    \item For all sufficiently small $\delta>0$, if $n\in\N$ is large enough, then we can construct a step function $\phi=\phi_n$ on $\R$ satisfying that $\phi(u) \geq \xi$ for all $u \geq M+1$, $\phi(\eta) \wedge \phi(-\eta) \leq C\eta$, and \begin{align}\label{eq:int_to_E}
	\begin{split}
		&\frac{1}{\sqrt{n}}\log\pp(\T(0,x_{i,n})=t_{i,n}\quad\forall i\in \Iintv{-J, J})\\
		&\leq \int_{-(M+3)}^{M+3} \log^-\Bigl[ \pp_u^{\rm BM}\Bigl( \tau_v \geq \phi^{-,\delta}(v)-\phi^{+,\delta}(u)
			\quad \forall  v \in [-(M+3),M+3] \Bigr)+\delta\Bigr] \,\dd u,
        \end{split}
    \end{align}
    where $\log^-u:=\log{(u \wedge 1)}\leq 0$ for $u>0$.
 
 \item We build a function $\psi=\psi_n \in \mathcal{C}_{M+2,\eta}(\xi-C\eta)$ such that
    \begin{align}\label{eq:const_psi}
    \begin{split}
		&\int_{-(M+3)}^{M+3} \log\Bigl[ \pp_u^{\rm BM}\Bigl( \tau_v \geq \phi^{-,\delta}(v)-\phi^{+,\delta}(u)\quad\forall v \in [-(M+3),M+3] \Bigr)+\delta \Bigr] \,\dd u\\
		&\leq \int_{-(M+2)}^{M+2} \log\Bigl[ \pp_u^{\rm BM}\Bigl( \tau_v \geq \psi^{-,\delta}(v)-\psi^{+,\delta}(u)\quad\forall v \in \R \Bigr)+\delta \Bigr] \,\dd u.
    \end{split}
    \end{align}
\end{enumerate}
These guarantee that for all sufficiently small $\delta>0$,
\begin{align*}
    \liminf_{n \to \infty}\frac{-1}{\sqrt{n}}
    \log\pp(\T(0,x_{i,n})=t_{i,n}\quad\forall i\in \Iintv{-J, J})
    \geq \liminf_{n \to \infty}E_{\delta,M+2}^+(\psi_n)
    \geq \inf_{f \in \mathcal{C}_{M+2,\eta}(\xi-C\eta)} E_{\delta,M+2}^+(f).
\end{align*}
This together with \eqref{pent} leads to
\begin{align*}
    \liminf_{n \to \infty} \frac{-1}{\sqrt{n}} \log\pp(\mathcal{E}_n)
    \geq \lim_{\delta \searrow 0} \inf_{f \in \mathcal{C}_{M+2,\eta}(\xi-C\eta)} E_{\delta,M+2}^+(f),
\end{align*}
and \eqref{eq:energy_key} follows. We shall complete the proof by carrying out steps~(1) and (2).

\vspace{0.5em}

\noindent
\underline{\bf Step~(1)}\quad
For simplicity of notation, we write $t_i=t_{i,n}$ and $x_i=x_{i,n}$. We define for $u \in \R$,
\begin{align*}
    \phi(u)=\phi_n(u)
    := \frac{t_{-J}}{n}\1\{ u \leq u_{-J}\}
    &+\sum_{i=-J+1}^0 \frac{t_{i}}{n} \1\{ u_{i-1}<u \leq u_i\}\\
    &+\sum_{i=0}^{J-1} \frac{t_{i}}{n}\1\{ u_i \leq u<u_{i+1} \}(u)
    +\frac{t_{J}}{n}\1\{ u_J \leq u \},
\end{align*}
where $u_i=u_{i,n}:=x_{i,n}/\sqrt{n}$ for $i \in \Iintv{-J,J}$. From the definition of $\kA_n$,  $\phi$ is a step function on $\R$ satisfying that $\phi(u) \geq \xi$ for all $u \geq M+1$ and $\phi(\eta) \wedge \phi(-\eta) \leq C\eta$.

We suppose $\T(0,x_{i})=t_{i}$ for any $i\in \Iintv{-J, J}$ . 
 Note that for all $x,y \in \Iintv{-2M\sqrt{n},2M\sqrt{n}}$,
\begin{align*}
	t(x,y) \geq \T(0,y) -\T(0,x) \geq n \big( \phi^{-,\delta}(y/\sqrt{n})-\phi^{+,\delta}(x/\sqrt{n}) \big).
\end{align*}
Hence, by the independence of the  simple random walks,
\begin{align*}
	&\pp\left( \T(0,x_{i})=t_{i}\quad\forall i\in \Iintv{-J, J} \right)\\
	&\leq \pp\Bigl( t(x,y) \geq n\left(\phi^{-,\delta}(y/\sqrt{n})-\phi^{+,\delta}(x/\sqrt{n})\right)\quad\forall x,y \in \Iintv{-2M\sqrt{n},2M\sqrt{n}} \Bigr)\\
	&= \prod_{x \in \Iintv{-2M\sqrt{n},2M\sqrt{n}}}\pp\Bigl( t(x,y) \geq n\left(\phi^{-,\delta}(y/\sqrt{n})-\phi^{+,\delta}(x/\sqrt{n})\right)\quad\forall y \in \Iintv{-2M\sqrt{n},2M\sqrt{n}} \Bigr).
\end{align*}
By  Lemma~\ref{lem:kmt} with $f=\phi^{-,\delta}$, $g=\phi^{+,\delta}$, $\epsilon_1=\epsilon_2=\delta$, for all large $n$, this is bounded from above by 
\begin{align*}
\prod_{x \in \Iintv{-2M\sqrt{n},2M\sqrt{n}}}\, (1\land h_{\delta}(x/\sqrt{n})),
\end{align*}
where for $u \in \R$,
\begin{align*}
	h_{\delta}(u):=
	\pp^{{\rm BM}}_u \Bigl( \tau_v \geq \phi^{-,2\delta}(v)-\phi^{+,2\delta}(u)
	\quad \forall  v \in [-2M+\delta,2M-\delta] \Bigr)+2\delta.
\end{align*}
Note that if $x \in \Iintv{-2M\sqrt{n},2M\sqrt{n}}$, $v \in \R$ and $0 \leq |h| \leq 1/\sqrt{n}<\delta$, then
\begin{align*}
	&\phi^{+,4\delta}(x/\sqrt{n}+h) \geq \phi^{+,2\delta}(x/\sqrt{n}),\quad \phi^{-,4\delta}(v+h)\leq 
	\phi^{-,2\delta}(v),\quad \forall v \in \R.
\end{align*}
Hence, for all large $n$,
\begin{align*}
	&\frac{1}{\sqrt{n}}\sum_{x \in \Iintv{-2M\sqrt{n},2M\sqrt{n}}} \log^- h_{\delta}(x/\sqrt{n})
\leq \sum_{x \in \Iintv{-2M\sqrt{n},2M\sqrt{n}}} \int_{x/\sqrt{n}}^{(x+1)/\sqrt{n}}
		\log^-  h_{2\delta}(u) \,\dd u\\
	&\leq \int_{-(M+3)}^{M+3}
		\log^-\left[\pp^{{\rm BM}}_u \Bigl( \tau_v \geq \phi^{-,4\delta}(v)-\phi^{+,4\delta}(u)
	\quad \forall  v \in [-(M+3),M+3] \Bigr)+4\delta\right]\,\dd u.
\end{align*}
Therefore, \eqref{eq:int_to_E} follows by  replacing $4\delta$ with $\delta$.\\

\noindent \underline{\bf Step~(2)}\quad
We write
\begin{align*}
	\phi_\diamond(u):=((\phi(u)\land \xi)-C\eta)_+,\qquad u \in \R.
\end{align*}
Since
\begin{align*}
	(\phi_\diamond^{-,\delta}(v)-\phi_\diamond^{+,\delta}(u))_+ \leq (\phi^{-,\delta}(v)-\phi^{+,\delta}(u))_+,\qquad u,v \in \R,
\end{align*}
we have
\begin{align*}
	&\int_{-(M+3)}^{M+3} \log^-\left[\pp_u^{\rm BM}\Bigl( \tau_v \geq \phi^{-,\delta}(v)-\phi^{+,\delta}(u)
		\quad \forall  v \in [-(M+3),M+3] \Bigr)+\delta\right] \,\dd u\\
	&\leq \int_{-(M+3)}^{M+3} \log^-\left[\pp_u^{\rm BM}\Bigl( \tau_v \geq \phi_\diamond^{-,\delta}(v)-\phi_\diamond^{+,\delta}(u)
			\quad \forall  v \in [-(M+3),M+3] \Bigr)+\delta\right] \,\dd u.
\end{align*}
Note that $\min\{\phi(\eta),\phi(-\eta)\} \leq C\eta$ from step~(1). We  define a function $\psi=\psi_n$ as follows: 
\begin{align*}
\psi (\cdot) = 	\phi_\diamond(\cdot+\eta) \quad  \text{if } \phi(\eta) \leq C\eta; \qquad  \psi (\cdot) = 	\phi_\diamond(\cdot-\eta) \quad  \text{otherwise}. 
\end{align*}
Then, $\psi  \in\kC_{M+2,\eta}(\xi-C\eta)$. 
In the case $\phi(\eta) \leq C\eta$, applying the change of variables $w=u-\eta$, we have for any $\delta \in (0,\eta/5)$,
\begin{align*}
	&\int_{-(M+3)}^{M+3} \log^-\left[\pp_u^{\rm BM}\Bigl( \tau_v \geq \phi_\diamond^{-,\delta}(v)-\phi_\diamond^{+,\delta}(u)
		\quad \forall  v \in [-(M+3),M+3] \Bigr) +\delta\right]\,\dd u\\
	&\leq \int_{-(M+3)+2\eta}^{M+3-3\eta} \log^-\left[\pp_w^{\rm BM}\Bigl( \tau_v \geq \psi^{-,\delta}(v)-\psi^{+,\delta}(w)
			\quad \forall  v \in [-M+3-\eta,M+3-\eta] \Bigr) +\delta\right]\,\dd w\\
	&\leq \int_{-(M+2)}^{M+2} \log\left[\pp_w^{\rm BM}\Bigl( \tau_v \geq \psi^{-,\delta}(v)-\psi^{+,\delta}(w)
		\quad \forall  v \in \R \Bigr)+\delta\right] \,\dd w=-E^+_{\delta,M+2}(\psi).
\end{align*}
When $\phi(-\eta) \leq C\eta$, by the change of variables $w=u+\eta$, we have the same. Therefore, \eqref{eq:const_psi} follows.
\end{proof}

\subsection{Proof of Lemma~\ref{lem:EpE}}\label{subsect:lem_Epe}
This subsection is devoted to the proof of Lemma~\ref{lem:EpE}.
We first introduce a variant of the L\'{e}vy distance in $\kC_{M,\eta}(\xi)$: given $f,g \in \kC_{M,\eta}(\xi)$, we define
\begin{align*}
	D(f,g):=\inf\left\{ \eps>0:f(x)>g^{-,\eps}(x)-\eps \text{ and } g(x)>f^{-,\eps}(x)-\eps\quad\forall x \in \R \right\}.
\end{align*}
The next lemma provides the compactness of  the  distance. 
\begin{lem}\label{lem:levy}
Let $(f_k)_{k=1}^\infty$ be a sequence on $\kC_{M,\eta}(\xi)$. 
Then, there exist a subsequence $(f_{k(n)})_{n=1}^\infty$  and $f_* \in \kC_{M,\eta}(\xi)$ such that $D(f_{k(n)},f_*) \to 0$ as $n \to \infty$.
\end{lem}
\begin{proof}
Since the same argument works when dividing these functions by $\xi$, without loss of generality, we suppose $\xi=1$.
Given $f\in\kC(1)$, we define the new functions 
$$f_+(x)=
\begin{cases}
f(x)& \text{ if $x\geq 0$}\\
0& \text{ if $x< 0$}
\end{cases},
\quad f_-(x)=
\begin{cases}
1& \text{ if $x> M+1$}\\
f(-x)& \text{ if $0<x\leq M+1$}\\
0& \text{ if $x< 0$}
\end{cases}.
$$
Given increasing functions $F,G$ with $\lim_{x\to\infty}F(x)=\lim_{x\to\infty}G(x)=1$,   the L\'evy distance is defined as 
$$L(F,G):=\inf\{\e>0 : \forall x\in \R,\,F(x)>G(x-\e)-\e,\,G(x)>F(x-\e)-\e\}.$$
By the  definition of $D$, we have
$$D(f,g)\leq L(f^+,g^+)+L(f^-,g^-).$$
By Prokhorov's theorem and the fact that weak convergence implies convergence of L\'evy distance (c.f., \cite[Theorem 5.1 and Remark (iv) on page 72]{B}), there exist a subsequence $(f_{n_k})_{k\geq 1}$ and $f^+_*$ such that $L(f^+_{n_k},f_*^+)\to 0$ as $k\to\infty$. Applying the same results again, there exist a subsequence $(n'_k)_{k\geq 1}$ of $(n_k)_{k\geq 1} $ and $f^-_*$ such that $L(f^-_{n'_k},f_*^-)\to 0$ as $k\to\infty$. Letting
$$f_*(x):=
\begin{cases}
 1& 
\text{ if $x \geq M$}, \\
f_*^+(x)& 
\text{ if $ \eta < x < M$},\\
0& \text{ if $|x|\leq \eta$},\\
f_*^-(-x)& \text{ if $x<-\eta$},
\end{cases}$$
we have $D(f_{n'_k},f_*)\to 0$ as $k\to\infty$. By definition, $f_*$ is non-decreasing in $[0,\infty)$ and non-increasing in $(-\infty,0]$. Moreover, by the convergences in L\'evy distance,  $f_*|_{[M,\infty)}\equiv 1$, $f_*|_{[-\eta,\eta]}\equiv 0$, and hence  $f_* \in \kC_{M,\eta}(1)$.
\end{proof}

We next consider a modification of the energy $E_{\delta,M}^+(f)$:
for any $f \in \kC(\xi)$,
\begin{align*}
	\tilde{E}_{\delta,M}^+(f)
	:=- \int_{-M}^M \log\left[\pp^{\rm BM}_x\left(\tau_y\geq  f^{-,\delta}(y)-f^{+,\delta}(x) -\delta
	\quad \forall  y \in \R \right)+\delta\right]\dd x.
\end{align*}
By definition, $\tilde{E}_{\delta,M}^+(f) \leq E_{\delta,M}^+(f) \leq E_M(f)$ holds for all $\delta>0$ and $f \in \kC(\xi)$.

    \begin{lem}\label{prop:EpE_contra}
Let $M \geq 1$ and $\eta>0$. If  
\begin{align*}
	\lim_{\delta\to 0} \inf_{f\in \kC_{M,\eta}(\xi)}\tilde{E}^+_{\delta,M}(f) <\inf_{f\in \kC_{M,\eta}(\xi)}E_{M}(f),
\end{align*}
then we can  find  a sequence $(\delta_k)_{k=1}^\infty \downarrow 0$  and $f_* \in \kC_{M,\eta}(\xi)$ such that
\begin{align*}
	\lim_{k \to \infty} \tilde{E}^+_{\delta_k,M}(f_*)& < \inf_{f\in \kC_{M,\eta}(\xi)}E_{M}(f).
\end{align*}
\end{lem}
\begin{proof}
We fix $\e>0$ such that
\aln{\label{edpl2}
\lim_{\delta \to 0} \inf_{f\in \kC_{M,\eta}(\xi)}\tilde{E}^+_{\delta, M}(f)+\e
	<\inf_{f\in \kC_{M,\eta}(\xi)}E_{M}(f)-\e.
 }
 Since $\delta \mapsto  \inf_{f\in \kC_{M,\eta}(\xi)}\tilde{E}^+_{\delta,M}(f)$ is non-decreasing, there exists $\delta_0>0$ such that for any $\delta\in (0,\delta_0)$, 
we can find $f_{\delta}\in \kC_{M,\eta}(\xi)$ satisfying 
\begin{align}\label{edpl}
	\tilde{E}^+_{\delta,M}(f_{\delta}) \leq \lim_{\delta\to 0} \inf_{f\in \kC_{M,\eta}(\xi)}\tilde{E}^+_{\delta,M}(f)+\e.
\end{align}
By Lemma~\ref{lem:levy}, there exist a sequence $(\delta_l)_{l=1}^\infty\subset (0,\delta_0)$ and $f_* \in \kC_{M,\eta}(\xi)$  such that  $\delta_l\downarrow 0$ and  $D(f_{\delta_l},f_*) \to 0$ as $l \to \infty$.
Fixing $k\in\N$, we take $l\geq k$ large enough satisfying that $\delta_l<\delta_k/4$ and $D(f_{\delta_l},f_*)<\delta_k/4$.
We show that for all $x \in \R$,
\begin{align}\label{eq:f&f*}
	f^{-,\delta_k}_*(x)<f_{\delta_l}^{-,\delta_l}(x)+\frac{\delta_k}{4},\quad
	f_*^{+,\delta_k}(x)>f_{\delta_l}^{+,\delta_l}(x)-\frac{\delta_k}{4}.
\end{align}
The first inequality directly follows from $D(f_{\delta_l},f_*)<\delta_k/4$. Since $f_{\delta_l},f_* \in \kC_{M,\eta}(\xi)$ and $\max\{D(f_{\delta_l},f_*),\delta_l\}<\delta_k/4$, 
\begin{align*}
	f_*^{+,\delta_k}(x)
	= \sup_{y\in [x-\delta_k,x+\delta_k]}f_*(y)
 &>\sup_{y\in [x-\delta_k,x+\delta_k]}\inf_{\delta\in[-\delta_k/4,\delta_k/4]}f_{\delta_l}^{-,\delta_k/4}(y+\delta)-\frac{\delta_k}{4}\\
&	\geq f_{\delta_l}^{+,\delta_k/4}(x)-\frac{\delta_k}{4}
		\geq f_{\delta_l}^{+,\delta_l}(x)-\frac{\delta_k}{4}.
\end{align*}
Hence, the second inequality of \eqref{eq:f&f*} holds. Therefore,
\begin{align*}
    \pp_x^{\rm BM}\left( \tau_y \geq f^{-,\delta_k}_*(y)-f^{+,\delta_k}_*(x)-\delta_k\quad\forall y \in \R \right)
    \geq \pp_x^{\rm BM} \left( \tau_y \geq f^{-,\delta_l}_{\delta_l}(y)-f^{+,\delta_l}_{\delta_l}(x)-\delta_l
		\quad \forall  y \in \R \right).
\end{align*}
Combining this with  \eqref{edpl2} and \eqref{edpl}, we obtain for all $k$,
\begin{align*}
	\tilde{E}^+_{\delta_k,M}(f_*)
	\leq \tilde{E}^+_{\delta_l,M}(f_{\delta_l})
	\leq \lim_{\delta \to 0} \inf_{f\in \kC_{M,\eta}(\xi)}\tilde{E}^+_{\delta, M}(f)+\e
	<\inf_{f\in \kC_{M,\eta}(\xi)}E_{M}(f)-\e. 
 \end{align*}
This  completes the proof by letting $k \to \infty$.
\end{proof}

We are now in a position to prove Lemma~\ref{lem:EpE}.

\begin{proof}[\bf Proof of Lemma~\ref{lem:EpE}]
Fix $\xi_0,\eta>0$.
We first prove the first claim, i.e., there exists $M_0\in\N$ such that for any $M\geq M_0$ and $\xi\leq\xi_0$,
\begin{align*}
	\inf_{f \in \kC_M(\xi)} E_M(f)  \geq \inf_{f \in \kC(\xi)} E(f)-M^{-2}.
\end{align*}
We take $M_0=M_0(\xi_0)\in\N$ sufficiently large such  that for any $M\geq M_0$,
\begin{align*}
    \pp_0^{\rm BM}(\tau_{M_0} \leq \xi_0) \leq \frac{1}{2},\qquad
    \frac{4}{\sqrt{2\pi}}\int_M^\infty \int_{x/\sqrt{\xi_0}}^\infty e^{-t^2/2}\,\dd t \,\dd x \leq M^{-2}.
\end{align*}
Suppose that $M\geq M_0$ and $\xi\leq \xi_0$.  For any $f \in \kC_M(\xi)$, 
since $\xi_0 \geq f \geq 0$, and $f$ is a constant function on  $(-\infty,-M]$ and $[M,\infty)$, we have
\begin{align*}
    E_M(f)- E(f)
    &= \int^{-M}_{-\infty} \log \pp_x^{\rm BM}(\tau_y \geq f_M(y)-f_M(x)\quad\forall y\geq 0) \,\dd x\\
    &\geq \int^{-M}_{-\infty} \log \pp_x^{\rm BM}(\tau_0 \geq \xi_0 )\,\dd x  \geq \int_{M}^{\infty} \log \pp_0^{\rm BM}(\tau_x \geq \xi_0 ) \,\dd x.
\end{align*}
By Lemma~\ref{lem:maxbrown} and the fact that $\log(1-t) \geq -2t$ for $0\leq t \leq 1/2$, we have for all $x \geq M$,
\begin{align*}
    \log \pp_0^{\rm BM}(\tau_x \geq \xi_0 )
    = \log \left\{ 1-\pp_0^{\rm BM}(\tau_x <\xi_0 ) \right\}
    \geq -2\pp_0^{\rm BM}(\tau_x < \xi_0 )
    = -\frac{4}{\sqrt{2\pi}}\int_{x/\sqrt{\xi_0}}^\infty e^{-t^2/2}\,\dd t.
\end{align*}
Combining the last three displayed equations, we reach
\begin{align*}
	E_M(f)-E(f)
	&\geq -\frac{4}{\sqrt{2\pi}}\int_M^\infty \int_{x/\sqrt{\xi_0}}^\infty e^{-t^2/2}\,\dd t \,\dd x \geq -M^{-2}.
\end{align*}
This estimate holds for  all $f \in \kC_M(\xi)$ and thus  the first claim follows.

Let us next prove the second claim, i.e., for all $M\geq 1$,
\begin{align*}
	\lim_{\delta\to 0} \inf_{f\in \kC_{M,\eta}(\xi)}E^+_{\delta,M}(f)=\inf_{f\in \kC_{M,\eta}(\xi)}E_M(f).
\end{align*}
Since $\tilde{E}^+_{\delta,M}(f) \leq E^+_{\delta,M}(f) \leq E_{M}(f)$  for all $f \in \kC_{M,\eta}(\xi)$, it suffices to show that for all $M \geq 1$,
\begin{align}\label{ldedf}
	\lim_{\delta\to 0} \inf_{f\in \kC_{M,\eta}(\xi)}\tilde{E}^+_{\delta,M}(f) \geq \inf_{f\in \kC_{M,\eta}(\xi)}E_{M}(f).
\end{align}
Suppose, towards a contradiction, that 
\begin{align*}
	\lim_{\delta\to 0} \inf_{f\in \kC_{M,\eta}(\xi)}\tilde{E}^+_{\delta,M}(f) < \inf_{f\in \kC_{M,\eta}(\xi)}E_{M}(f).
\end{align*}
By Lemma~\ref{prop:EpE_contra}, we can find  a sequence $(\delta_k)_{k=1}^\infty \downarrow 0$  and $f_* \in \kC_{M,\eta}(\xi)$ such that 
\begin{align*}
	\lim_{k \to \infty} \tilde{E}^+_{\delta_k,M}(f_*) <\inf_{f\in \kC_{M,\eta}(\xi)}E_{M}(f).
\end{align*}
Once we prove 
\begin{align}\label{approx 2}
	 E_{M}(f_*)\leq \lim_{k\to\infty }\tilde{E}^+_{\delta_k,M}(f_*) ,
\end{align}
the following inequalities hold:
\begin{align*}
	\inf_{f\in \kC_{M,\eta}(\xi)}E_{M}(f)
	\leq E_{M}(f_*)
	\leq \lim_{k\rightarrow \infty} \tilde{E}^+_{\delta_k,M}(f_*)
	< \inf_{f \in \kC_{M,\eta}(\xi)}E_{M}(f),
\end{align*}
which derives a contradiction.
Therefore, one has \eqref{ldedf}.
It remains to check \eqref{approx 2}.
Let $\Gamma_*$ be the set of all the points of discontinuity of $f_*$.
Since $f_* \in \kC_{M,\eta}(\xi)$, $\Gamma_*$ is a finite subset of $[-M,M]$. 
Hence,  by the monotone convergence theorem  and  $\lim_{k \to \infty} f_*^{\pm,\delta_k}(z)=f_*(z)$ holds for any $z \in [-M,M] \cap \Gamma_*^c$,
\begin{align} \label{letdk}
	-\tilde{E}^+_{\delta_k,M}(f_*)
	&= \int_{[-M,M]} \log \left[ \pp^{\rm BM}_x\left( \tau_y \geq
		f_*^{-,\delta_k}(y)-f_*^{+,\delta_k}(x)-\delta_k\quad\forall y \in \R  \right) + \delta_k \right]\dd x \notag \\
	&\leq \int_{[-M,M]} \log \left[ \pp^{\rm BM}_x\left( \tau_y \geq
		f_*^{-,\delta_k}(y)-f_*^{+,\delta_k}(x)-\delta_k\quad\forall y \in \Gamma_*^c \right) + \delta_k \right]\dd x \notag \\
  &\xrightarrow{k \rightarrow \infty} \int_{[-M,M]} \pp^{\rm BM}_x\Bigl( \tau_y \geq f_*(y)-f_*(x)\quad\forall y \in  \Gamma_*^c \Bigr) \,\dd x.
\end{align}
Let $\mathcal{E}_x$ be the event appearing in the last probability.
Fixing $z \in \Gamma_*$, we can find a sequence $(y_i)_{i=1}^\infty$ on $\Gamma_*^c
$ such that  $\lim_{i \to \infty} y_i=z$ and  $\lim_{i \to \infty} f_*(y_i) \geq f_*(z)$.   Moreover, if $\kE_x$ occurs, then $\tau_{y_i} \geq f_*(y_i)-f_*(x)$ for all $i \geq 1$.  This combined with \cite[(8.8)]{KarShr91_book} shows that $\P_x^{\rm BM}$-a.s.~on $\mathcal{E}_x$,
\begin{align*}
	\tau_z=\lim_{i \to \infty} \tau_{y_i} \geq \lim_{i \to \infty} f_*(y_i)-f_*(x)\geq f_*(z)-f_*(x).
\end{align*}
Since $\Gamma_*$ is finite, this implies
\begin{align*}
    \P_x^{\rm BM}(\mathcal{E}_x)
    = \P_x^{\rm BM}(\tau_y \geq f_*(y)-f_*(x)\quad\forall y \in \R).
\end{align*}
Combining the above with \eqref{letdk}, we obtain \eqref{approx 2}.
\end{proof}

\section{Localization phenomenon: Proof of Proposition \ref{prop:ts2} } \label{sec:propts2}  
In this section, we consider the localization of the rare event and show that 
  \ben{ \label{slp}
 \lim_{n\rightarrow \infty} \frac{1}{\sqrt{n}}\log \pp ({\rm T}(0,n)\geq (\mu+\xi) n) = \lim_{n\rightarrow \infty} \frac{1}{\sqrt{n}}\log \pp(\rmT(0,M\sqrt{n}) \geq \xi n) + o_M(1).
  }
 This roughly implies that the best strategy to delay the transmission from $0$ to $n$ is to slow down the infection in a bad interval (in the sense that ${\rm T}(a,b)\geq |a-b|^2$ for the interval $[a,b]$) of size  $\kO(\sqrt{n})$. 
 This section is organized as follows:
 The first two subsections are devoted to the proof of the upper bound in \eqref{slp}, i.e. Proposition~\ref{prop:ts2}-(ii), which is the most difficult part. The last two subsections are devoted to the proofs of Proposition~\ref{prop:ts2}-(i) and (iii). 
\subsection{Proof of Proposition~\ref{prop:ts2}-(ii)}\label{subsect:prop3.1-2}

In this subsection, we aim to show that the large deviation event can be localized around several bad intervals whose total passage time is larger than $\xi n$. 
 Let us summarize here the main ideas of the proof. We shall use a multi-step covering process to localize the bad intervals. 

\textbf{ Step 1} ({\bf Control of  small bad intervals}): We show in this step that the bad intervals of size less than $(\log n)^2$ are harmless to the upper tail large deviation event. More precisely, we divide $[0,n]$  into subintervals:
\be{
[0,n] \subset \bigcup_{i \not \in \mathbf{Red}} [i,i+K_n] \cup \bigcup_{i  \in \mathbf{Red}} [i,i+K_n], 
}
where $ \mathbf{Red}$ denotes the set of red intervals, i.e. bad intervals of size $K_n=\lfloor (\log n)^2 \rfloor$:
\be{
\mathbf{Red}:=\{i \in K_n \N\cap [0,n]: \rmT(i,i+K_n ) \geq K_n^2 \},
}
and then prove that 
\be{
\pp \biggl( \sum_{i \not \in \mathbf{Red}} \rmT(i,i+K_n ) \geq (\mu+o(1))n\biggr) \leq  \exp(-n^{2/3}). 
}
\textbf{Step~2 (First covering of red intervals)}: We aim to aggregate these intervals into larger ones that are far from each other. Precisely, we seek for a covering such that
\begin{itemize}
\setlength{\itemsep}{0.5em}
    \item $\displaystyle \bigcup_{i\in \mathbf{Red}} [i,i+  K_n] \subset \bigcup_{j=1}^\ell [S_j,T_j+\kL_j]$,
    \item ${\rm T}(S_j, T_j+\kL_j)\leq 16 \kL_j^2$ for $j=1,\dots,\ell$,
    \item $(x_j-\tfrac{\kL_j}{3},x_j+\tfrac{\kL_j}{3})\cap (x_k-\tfrac{\kL_k}{3},x_k+\tfrac{\kL_k}{3})=\varnothing$ for $1 \leq j<k \leq \ell$,
\end{itemize}
where $x_j$ is a focal point of $[S_j,T_j+\kL_j]$. Roughly speaking, $\kL_1$ is the largest bad radius in dyadic scale: $\kL_1:= \max \{2^k: \exists \, x \in [0,n], \rmT(x,x+2^k) \geq 2^{2k}\}$ and take $[S_1,T_1]:=[x_1-\kL_1,x_1+\kL_1]$ with $x_1$  a maximizer in the definition of $\kL_1$.  The second interval is obtained by the same process applied to the remaining  $[0,n] \setminus [S_1,T_1]$. We continue this procedure until all red intervals are covered. 

 Furthermore, by construction of such a covering, applying Perles's type arguments, we are able to control the number of bad intervals at a given size. In particular, we can show that the moderate bad intervals of size $o(\sqrt{n})$ are controllable, i.e. there exists a constant $c>0$ such that  for any $\varepsilon, \delta>0$
 \be{
  \pp \left(\sum_{j=1}^{\ell}  \rmT(S_j, T_j + \kL_j) \1\{ \kL_j \leq \sqrt{\varepsilon n} \} \geq \delta n\right) 
\leq \exp(-c \delta \sqrt{n/\e}).
 }
\textbf{Step~3 (Refined covering of bad intervals)}: We apply an additional aggregation process to cover the large bad interval  
\be{
\bigcup_{j: \kL_j \geq \sqrt{\varepsilon n}} [S_j, T_j + \kL_j]  \subset \bigcup_{i=1}^m [s_i, t_i],
}
 where $([s_i, t_i])_{i=1}^m$ are intervals satisfying 
 \be{
 {\rm d}([s_i,t_i], [s_j,t_j])>M^3\sqrt{n}\text{, and } \sum_{i=1}^m |t_i-s_i|\leq M^6\sqrt{n},
 }
    with some $m \leq M$.  Combining the constructed coverings, we arrive at 
\be{
[0,n] \subset \bigcup_{i \not \in \mathbf{Red}} {\rm T}(i,i+K_n) \cup \bigcup_{j: \kL_j \leq \sqrt{\varepsilon n}} [S_j, T_j + \kL_j]  \cup  \bigcup_{i=1}^m [s_i, t_i].
}
Remark that since the intervals  $([s_i, t_i])_{i=1}^m$  are sufficiently far from each other, by nearly independence of the passage times on these intervals, we can show that with $\varepsilon=\varepsilon(\delta, M)$ chosen suitably,
\al{
\pp  \left(\sum_{i=1}^{m}  \rmT(s_i,t_i)  \geq (\xi-\delta)n  \right) &\lesssim \sum_{(n_i)_{i=1}^m\in \N^m}\prod_{i=1}^m \pp  \left(\rmT(s_i,t_i)  \geq  n_i  \right) \mathbf{1}_{\sum_i n_i\geq (\xi-\delta)n },
}
which implies Proposition~\ref{prop:ts2}-(ii).

To this end, let us prepare a lemma and a corollary.
\begin{lem}\label{lem:keyup}
For any $\delta,A>0$, there exists $M_1=M_1(\delta, A)\in\N$ such that for any $M\geq M_1$, $\xi\geq 0$, and for any $n\in\N$ large enough,
\begin{align*}
    \pp(\rmT(0,n)\geq (\mu+\xi)n) \leq \exp(-A \sqrt{n}) + n^{2M} \max_{(s_i,t_i)_{i=1}^m \in \kS_M(n)} \pp \left( \sum_{i=1}^m \rmT(s_i,t_i) \geq (\xi- \delta)n \right),
\end{align*}
where
\begin{align*}
    \kS_M(n):=\left\{ (s_i,t_i)_{i=1}^m:
    \hspace{-2em}
    \begin{minipage}{22em}
        \begin{itemize}
        \setlength{\itemsep}{0.3em}
            \item $m \in \Iintv{1,M}$,
            \item $s_i,t_i \in \Iintv{0,n}$ and $s_i<t_i$ for all $i \in \Iintv{1,m}$,
            \item ${\rm d}([s_i,t_i], [s_j,t_j])>M^3\sqrt{n}$ for all $i<j$,
            \item $\sum_{i=1}^m |t_i-s_i|\leq M^6\sqrt{n}$.
        \end{itemize}
    \end{minipage}
    \right\}
\end{align*}
Recall that ${\rm d}([s_i,t_i],[s_j,t_j])$ stands for the Euclidean distance between two intervals $[s_i,t_i]$ and $[s_j,t_j]$, see Section~\ref{subsect:organization}.
\end{lem}

We postpone the proof of Lemma~\ref{lem:keyup}  to the next subsection.
Although the next corollary is a direct consequence of Lemma~\ref{lem:keyup}, it is useful to control the probability that the first passage time extremely deviates upward from the time constant.
We use the corollary in not only the proof of (ii) but also that of (i) in Proposition~\ref{prop:ts2}.

\begin{cor}\label{corollary}
For any $A>0$, there exists $M_2=M_2(A)\in \N$ such that for any $M\geq M_2$, and for any $n\in\N$ large enough,
\begin{align*}
    \pp(\rmT(0,n) \geq M n) \leq \exp(-A\sqrt{n}).
\end{align*}
\end{cor}
\begin{proof}
Fix $A>0$ and let $c>0$ be a universal constant as in Lemma \ref{lem:taln}-(ii).
We use Lemma~\ref{lem:keyup} with $\delta:=1$ and $2A$ in place of $A$:
there exists $M_1=M_1(1,2A) \in \N$ such that for all $M \geq M_1$, $\xi \geq 0$ and for all $n$ large enough,
\begin{align*}
    \pp(\rmT(0,n)\geq (\mu+\xi)n) \leq \exp(-2A \sqrt{n}) + n^{2M} \max_{(s_i,t_i)_{i=1}^m \in \kS_M(n)} \pp \left( \sum_{i=1}^m \rmT(s_i,t_i) \geq (\xi-1)n \right).
\end{align*}
Take $L:=L(A)>(M_1+\mu+1)^6$ large enough to have $L\exp(-cL\sqrt{n}) \leq \exp(-2A\sqrt{n})$ for all $n \in \N$, and let $\xi=L^4-\mu$ and $M=M_1$.
Then, since $\xi-1 \geq L^3$,
\begin{align}\label{pl13n}
    \pp(\rmT(0,n)\geq L^4n)
    \leq \exp(-2A \sqrt{n}) + n^{2M_1} \max_{(s_i,t_i)_{i=1}^m \in \kS_{M_1}(n)} \pp \left( \sum_{i=1}^m \rmT(s_i,t_i) \geq L^3n \right).
\end{align}
Note that if $(s_i,t_i)_{i=1}^m \in \kS_{M_1}(n)$, then $m \leq M_1 \leq L$ and $|t_i-s_i| \leq M_1^6\sqrt{n} \leq L \sqrt{n}$ for all $i \in \Iintv{1,n}$.
Hence, Lemma~\ref{lem:taln}-(ii) with $\alpha=L$ and $\beta=L^2$ shows that there exists a universal constant $c>0$ such that for all large $n \in \N$,
\begin{align*}
    \max_{(s_i,t_i)_{i=1}^m \in \kS_{M_1}(n)} \pp \left(\sum_{i=1}^{m} \T(s_i,t_i) \geq L^{3}n \right) 
    &\leq \max_{(s_i,t_i)_{i=1}^m \in \kS_{M_1}(n)} \sum_{i=1}^m \pp(\T(s_i,t_i) \geq L^{2}n) \\
    & \leq L \times \pp\bigl( \T(0,\lfloor L \sqrt{n} \rfloor) \geq L^{2}n \bigr)\\
    &\leq L \exp(-cL \sqrt{n})
    \leq \exp(-2A\sqrt{n}).
\end{align*}
This combined with \eqref{pl13n} yields that for all large $n \in \N$,
\begin{align*}
    \pp(\rmT(0,n)\geq L^4n)
    \leq (n^{2M_1}+1)\exp(-2A\sqrt{n})
    \leq \exp(-A\sqrt{n}),
\end{align*}
and the corollary follows by taking $M_2:=L^4$.
\end{proof}

We are now in a position to prove Proposition~\ref{prop:ts2}-(ii).

\begin{proof}[\bf Proof of Proposition~\ref{prop:ts2}-(ii)]
Fix $\delta,A>0$.
Lemma \ref{lem:keyup} with $A$ replaced by $2A$ implies that there exists $M_1=M_1(\delta,2A) \in \N$ such that for all $L \geq M_1$, $\xi \geq 0$ and for all large $n \in \N$,
\begin{align}\label{m12a}
    \pp(\T(0,n) \geq (\mu+\xi)n)
    \leq \exp(-2A \sqrt{n}) + n^{2L} \max_{(s_i,t_i)_{i=1}^m \in \kS_L(n)} \pp \left( \sum_{i=1}^m \T(s_i,t_i) \geq (\xi- \delta)n \right).
\end{align}
For each $(s_i,t_i)_{i=1}^m \in \kS_L(n)$,
\begin{align}\label{tisi1}
\begin{split}
    \pp\left( \sum_{i=1}^{m} \T(s_i,t_i) \geq (\xi -\delta)n \right)
    &\leq \pp\left(\sum_{i=1}^{m} \T(s_i,t_i) \geq L^2n \right)\\
    &\quad +\sum_{(h_i)_{i=1}^m \in \N^m}\pp\big( \T(s_i,t_i) \geq h_i \quad \forall  i \in \Iintv{1,m} \big)
    \times \1\bigl\{ {\textstyle (\xi-\delta)n \leq \sum_{i=1}^m h_i \leq L^2 n} \bigr\}.
\end{split}
\end{align}
We first treat the first term in the right-hand side of \eqref{tisi1}.
Note that if $(s_i,t_i)_{i=1}^m \in \kS_L(n)$, then $m \leq L$ and $\max_{1 \leq i \leq m}|t_i-s_i| \leq L^6\sqrt{n} \leq n$ for all large $n$.
Hence, Corollary~\ref{corollary} with $A$ replaced by $2A$ implies that for any $L \geq M_2(2A)$, if $n \in \N$ is large enough, then
\begin{align}\label{tisi2}
    \max_{(s_i,t_i)_{i=1}^m \in \kS_L(n)}
    \pp \left(\sum_{i=1}^{m} \T(s_i,t_i) \geq L^2 n \right)
    \leq L \pp(\T(0,n) \geq Ln)
    \leq L \exp(-2A\sqrt{n}).
\end{align}

Let us next estimate the second term in the right-hand side of \eqref{tisi1}.
Fix $(s_i,t_i)_{i=1}^m \in \kS_L(n)$.
Then, the intervals $([s_i-(L^3/2)\sqrt{n},t_i])_{i=1}^m$ are disjoint. Thus,
\begin{align*}
    \pp\big( \T(s_i,t_i) \geq h_i \quad \forall i \in \Iintv{1,m} \big)
    &\leq \pp\Big( \T_{[s_i-(L^3/2)\sqrt{n},t_i]}(s_i,t_i) \geq h_i \quad \forall i \in \Iintv{1,m} \Big)\\
    &= \prod_{i=1}^m \pp\Big( \T_{[s_i-(L^3/2)\sqrt{n},t_i]}(s_i,t_i)\geq h_i\Big).
\end{align*}
 Furthermore, Lemma~\ref{Task3} and the fact that $\max_{1 \leq i \leq m}|t_i-s_i| \leq L^6 \sqrt{n}$ yields that there exists a universal constant $\alpha>0$ such that
\begin{align*}
    \pp\Bigl( \T_{[s_i-(L^3/2)\sqrt{n},t_i]}(s_i,t_i)\geq h_i \Bigr)
    &= \pp\Bigl( \T_{[-(L^3/2)\sqrt{n},t_i-s_i]}(0,t_i-s_i)\geq h_i \Bigr)\\
    &\leq \exp\left(\frac{2\alpha h_i}{L^3\sqrt{n}}\right) \pp( \T(0,t_i-s_i)\geq h_i)\\
    &\leq  \exp\left( \frac{2\alpha h_i}{L^3\sqrt{n}}\right) \pp\bigl( \T(0,\lfloor L^6\sqrt{n} \rfloor)\geq h_i \bigr).
\end{align*}
With these observations, for all $n \in \N$ and for all $(s_i,t_i)_{i=1}^m \in \kS_L(n)$,
\begin{align}\label{ssti}
\begin{split}
    &\sum_{(h_i)_{i=1}^m \in \N^m}\pp\big( \T(s_i,t_i) \geq h_i \quad \forall  i \in \Iintv{1,n} \big)
    \,\1\bigl\{ {\textstyle (\xi-\delta)n \leq \sum_{i=1}^m h_i \leq L^{2} n} \bigr\}\\
    &\leq \exp\left( \frac{2\alpha \sqrt{n}}{L}\right)
    \sum_{(h_i)_{i=1}^m \in \N^m} \1\bigl\{ {\textstyle (\xi-\delta)n \leq \sum_{i=1}^m h_i \leq L^2 n} \bigr\} \prod_{i=1}^m \pp(\T(0, \lfloor L^6\sqrt{n} \rfloor) \geq h_i).
\end{split}
\end{align}
Due to \eqref{tisi1}, \eqref{tisi2} and \eqref{ssti}, for any $L \geq M_2(2M)$, if $n \in \N$ is large enough, then
\begin{align*}
    &\max_{(s_i,t_i)_{i=1}^m \in \kS_L(n)} \pp \left( \sum_{i=1}^m \T(s_i,t_i) \geq (\xi- \delta)n \right)\\
    &\leq L\exp(-2A\sqrt{n})
    +\exp\left( \frac{2\alpha \sqrt{n}}{L}\right)
    \sum_{m=1}^M \sum_{(h_i)_{i=1}^m \in \N^m} \1\bigl\{ {\textstyle (\xi-\delta)n \leq \sum_{i=1}^m h_i \leq L^2n} \bigr\}
    \prod_{i=1}^m \pp(\T(0, \lfloor L^6\sqrt{n} \rfloor) \geq h_i).
\end{align*}
Replace $L$ with $M^{1/6}$ in \eqref{m12a} and the above expression and take $M_0=M_0(c,\delta,A,\xi) \in \N$ large enough to have
\begin{align*}
    M_0 \geq (M_1(\delta,2A)+M_2(2A)+(4\alpha/c))^6+\xi+2.
\end{align*}
It follows that for any $c,\delta,A,\xi>0$ and for any $M \geq M_0$, if $n \in \N$ is large enough, then
\begin{align*}
    &\pp(\T(0,n) \geq (\mu+\xi)n)\\
    &\leq (Mn^{2M}+1)\exp(-2A \sqrt{n})
    +n^{2M} \exp(c\sqrt{n}/2) \sum_{m=1}^M \sum_{(h_i)_{i=1}^m \in \mathcal{H}_{m,n}^\delta}
    \prod_{i=1}^m \pp \bigl( \T(0,\lfloor M\sqrt{n} \rfloor) \geq h_i \bigr).
\end{align*}
Therefore, we obtain the desired conclusion since $(Mn^{2M}+1)\exp(-2A \sqrt{n}) \leq \exp(-A \sqrt{n})$ and $n^{2M}\exp(c\sqrt{n}/2) \leq \exp(c\sqrt{n})$ hold for all large $n \in \N$.
\end{proof}

\subsection{Proof of Lemma~\ref{lem:keyup}}   \label{polku}
 
We define for $n \in \N$,
\begin{align*}
    N:= \lceil 2 \log_2 (\log n)\rceil, \qquad \kN = 2^{N} \Z \cap [0, n-2^{N}].
\end{align*}
Divide the interval $[0,n]$ into subintervals $\{[i,i+2^N]\}_{i \in \kN}$ and classify them to two colors: blue and red. Given $i \in \kN$, if $\T(i,i+2^N) >  2^{2N}$, then we write   $i\in  \mathbf{Red}$; otherwise (i.e., $\T(i,i+2^N) \leq  2^{2N}$), we write  $i\in \mathbf{Blue}$.
Let us now cover the interval $[0,n]$ with red and blue intervals as follows:
\begin{align}\label{deon}
    [0,n] \subset \bigcup_{i \in \kN} [i, i+2^{N}] =  \bigcup_{i \in \mathbf{Blue}} [i, i+2^{N}] \,    \cup \, \bigcup_{i\in \mathbf{Red}} [i, i+2^{N}].
\end{align}
First, Section~\ref{subsubsect:blue} takes care of the total passage time of blue intervals.
Next, in Section~\ref{subsubsect:blue}, we estimate the contribution from red intervals to the first passage time, and prove Lemma~\ref{lem:keyup} by combining estimates for blue and red intervals.
   

\subsubsection{The total passage time of blue intervals}\label{subsubsect:blue}

\begin{lem} \label{lem:syi}
 
For any fixed $\delta >0$ and $n\in\N$ sufficiently large,
\begin{align}\label{p2}
    \pp \left( \sum_{i\in \mathbf{Blue}} \rmT(i,i+2^{N}) \geq  (\mu+2\delta) n  \right) \leq \exp(-n^{2/3}).
\end{align}
\end{lem}
\begin{proof} 
Given $\delta >0$, we define
\begin{align*}
    K=\lceil C + 8 / \delta \rceil,
\end{align*}
where $C$ is the constant as in Lemma~\ref{lem:tgcn}.
We divide blue intervals into three classes $\mathbf{Lblue}$ (light blue), $\mathbf{Mblue}$ (moderate blue) and $\mathbf{Dblue}$ (dark blue) as follows:
\begin{align*}
    &\mathbf{Lblue}=\bigl\{ i \in \kN: \rmT(i,i+2^N) \leq (\mu + \delta) 2^N \bigr\},\\
    &\mathbf{Mblue}=\bigl\{ i \in \kN: (\mu + \delta) 2^N \leq \rmT(i,i+2^N) \leq K 2^N  \bigr\},\\
    &\mathbf{Dblue}=\bigl\{ i \in \kN: K 2^N \leq \rmT(i,i+2^N) \leq  2^{2N} \bigr\}.
\end{align*}
Then, using the fact that $|\mathbf{Lblue}| \leq |\kN| \leq n/2^N$, we have for all $n$ sufficiently large,
\begin{align}\label{blue}
\begin{split}
    \sum_{i\in \mathbf{Blue}} \T(i,i+2^{N})
    &\leq (\mu+\delta)2^N |\mathbf{Lblue}| + K2^N |\mathbf{Mblue}|+ 2^{2N} |\mathbf{Dblue}|\\ &\leq (\mu+\delta)n + K2^N |\mathbf{Mblue}|+ 2^{2N} |\mathbf{Dblue}|.
\end{split}
\end{align}
Hence, our task is now to prove that for all $n$ sufficiently large,
\begin{align}\label{eq:MDblue}
    \pp\biggl( K2^N|\mathbf{Mblue}| \geq \frac{\delta}{2}n \biggr)
    +\pp\biggl( 2^{2N}|\mathbf{Dblue}| \geq \frac{\delta}{2}n \biggr)
    \leq \exp(-n^{2/3}).
\end{align}

We first treat the probability for $|\mathbf{Mblue}|$.
The translation invariance and \eqref{eq:shape_thm} yield that for all $n$ sufficiently large and for any $i \in \kN$,
\begin{align*}
    \pp(i \in \mathbf{Mblue})
    \leq \pp\bigl( \rmT(0,2^N) \geq (\mu + \delta)2^N \bigr)
    \leq 1/K^2.
\end{align*}
Divide $\kN$ into $4K$ disjoint groups as follows:
\begin{align*}
    \kN = \bigcup_{j=0}^{4K-1} \kM_j,
    \quad \kM_j:=\big\{i \in \kN : \tfrac{i}{2^{N}} \equiv j \ (\textrm{mod } 4K) \big\}.
\end{align*}
Notice that the event $\{i \in \mathbf{Mblue}\}$ depends only on frogs $\{(S^x_{\cdot})\}_{|x-i|\leq K 2^{N}}$.
Moreover, by the definition, for each $j=0, \ldots, 4K-1$, we have $|i-i'| \geq 4K 2^{N}$ for all distinct $i,i' \in \kM_j$.
Thus, these events $(\{i \in \mathbf{Mblue}\})_{i \in \kM_j}$ are independent. Therefore, $|\kM_j \cap \mathbf{Mblue}|$ is stochastically dominated by the Binomial distribution $\mathbf{Bin}(|\kM_j|, K^{-2})$. Hence, using Chernoff's bound, we have for all $j=0, \ldots, 4K-1$,
\begin{align*}
    \pp\bigl( |\kM_j \cap \mathbf{Mblue}| \geq n/(2^N K^3) \bigr)
    \leq \exp\bigl( -n/(2^{N+1} K^3) \bigr).
\end{align*}
Hence, by $K \geq 8/\delta$ and the union bound, for all $n$ large enough,
\begin{align}\label{mblue}
\begin{split}
    \pp\biggl( K2^N|\mathbf{Mblue}| \geq \frac{\delta}{2}n \biggr)
    &\leq \pp\bigl( |\mathbf{Mblue}| \geq n/(2^{N-2} K^2) \bigr)\\
    &\leq \exp\bigl( -n/(2^{N+2} K^3) \bigr)
    \leq \frac{1}{2}\exp(-n^{2/3}).
\end{split}
\end{align}

Finally, we estimate the probability for $|\mathbf{Dblue}|$.
By Lemma \ref{lem:tgcn}, since $K \geq C$ (the constant as in this lemma),  for $n$ sufficiently large 
\be{
 \pp(i \in \mathbf{Dblue}) \leq \pp(\rmT(0,2^N) \geq C2^N) \leq \exp(-c2^{N/4}) \leq 2^{-4N},
 }
with $c$ a positive constant. Divide $\kN$ into $2^{N+2}$ disjoint groups as follows:
 	\[ \kN = \bigcup_{j=0}^{2^{N+2}-1} \kD_j, \quad \kD_j:=\big  \{i \in \kN : \tfrac{i}{2^{N}} \equiv j \pmod{2^{N+2}} \big\}. \]
Observe also that the event $\{i \in \mathbf{Dblue}\}$ depends only on frogs $(S^x_{\cdot})_{|x-i|\leq  2^{2N}}$. Thus for each $j \in \kD_j$,  the events  $(\{i \in \mathbf{Dblue}\})_{i \in \kD_j}$ are independent, and hence  $|\kD_j \cap \mathbf{Dblue}|$ is stochastically dominated by $\mathbf{Bin}(|\kD_j|, 2^{-4N})$. Therefore, Chernoff's bound proves that for each $j=0,\dots,2^{N+2}-1$,
\begin{align*}
    \pp(|\kD_j \cap \mathbf{Dblue}| \geq n/2^{5N} ) \leq \exp(-n/2^{5N+1} ).
\end{align*}
This together with the union bound implies that for all $n$ sufficiently large,
\begin{align}\label{dblue}
\begin{split}
    \pp\biggl( 2^{2N}|\mathbf{Dblue}| \geq \frac{\delta}{2}n \biggr)
    &\leq \pp(|\mathbf{Dblue}| \geq n/2^{4N-2})\\
    &\leq \exp(-n/2^{5N+2})
    \leq \frac{1}{2}\exp(-n^{2/3}).
\end{split}
\end{align}
With these observations, \eqref{eq:MDblue} follows from \eqref{mblue} and \eqref{dblue}, and the proof is complete.
\end{proof}
  
\subsubsection{A covering of  red intervals}\label{subsubsect:red}
We will use a   covering process of disjoint boxes to aggregate the red intervals whose first passage time larger than the square of its distance.
Initially, we define
$$\kL_1:=\max\bigl\{ 2^k:\exists x\in \Iintv{0,n};~{\rm T}(x,x+2^k)\geq 2^{2k} \bigr\}.$$
By Lemma~\ref{lem:taln}-(ii), $\pp(\rmT(0,2^k) \geq 2^{2k}) \leq \exp(-c2^k)$. Then we have $\kL_1 < \infty$ a.s. We take $x_1\in \Iintv{0,n}$ such that ${\rm T}(x_1,x_1+\kL_1)\geq \kL_1^2$ with a deterministic rule breaking ties. Let $I_1=(x_1-\kL_1,x_1+\kL_1)$ and we define $S_1,{\rm T}_1$ such that $[S_1,{\rm T}_1]=[x_1-\kL_1,x_1+\kL_1]\cap [0,\infty)$. Inductively, we define for $j\geq 1$
$$\kL_{j+1}:=\max\{2^k:\exists x\in \Iintv{0,n} \setminus I_j;~{\rm T}(x,x+2^k)\geq 2^{2k}\}.$$
We take $x_{j+1}\in \Iintv{0,n} \setminus I_j$ such that ${\rm T}(x_{j+1},x_{j+1}+\kL_{j+1})\geq \kL_{j+1}^2$ with a deterministic rule breaking ties. Let
\bea{
\tilde{I}_{j+1}&:=&[x_{j+1}-\kL_{j+1},x_{j+1}+\kL_{j+1}] \setminus I_j, \\
 I_{j+1}&:=&I_j\cup (x_{j+1}-\kL_{j+1},x_{j+1}+\kL_{j+1}).
}
Note that $I_j$ is the union of some intervals whose lengths are all larger than that of $[x_{j+1}-\kL_{j+1},x_{j+1}+\kL_{j+1}]$, since $\kL_i\geq \kL_{j+1}$ for any $i\leq j$. Therefore  $\tilde{I}_{j+1}$ is an interval (if not, then there is an interval included in $[x_{j+1}-\kL_{j+1},x_{j+1}+\kL_{j+1}]$ and so $\kL_{j+1} > \kL_i$ for some $i<j$). Hence,  we can define $S_{j+1}\leq T_{j+1}$ such that $[S_{j+1},T_{j+1}]=\tilde{I}_{j+1}\cap [0,\infty)$. Recall $N=\lceil 2\log_2{(\log{n})}\rceil$ and let us define
$$ \ell:= \max\{j:\kL_j \geq 2^{N} \}.$$ 
\begin{lem} \label{lem:sti1d}
	The following hold:
	\begin{itemize}
		\item [(i)] We have
		$$ \bigcup_{i\in \mathbf{Red}} [i,i+2^{N}]\subset \bigcup_{j=1}^{\ell} [S_j,T_j+\kL_j].$$ 
		\item[(ii)] For any   $1 \leq j\leq \ell$,
		$${\rm T}(S_j, T_j+\kL_j)\leq 16 \kL_j^2.$$
		\item[(iii)]  For any $1 \leq i\neq j\leq \ell$,
		$$(x_i-\kL_i/3,x_i+\kL_i/3)\cap (x_j-\kL_j/3,x_j+\kL_j/3)=\varnothing.$$
	\end{itemize}
\end{lem}
\begin{proof}  Suppose that $i\in \mathbf{Red}$, i.e., ${\rm T}(i,i+2^{N})\geq 2^{2N}$. If $i\notin I_{\ell}$, then $\kL_{\ell+1}\geq 2^{N}$, which contradicts  the definition of $\ell$. Thus we get $i\in I_{\ell}$, and hence  there exists $j\leq \ell$ such that $i\in [S_j ,\rmT_j]$. Therefore,  $2^{N} \leq \kL_j$ and so 
	$$[i,i+2^{N}]\subset [S_j, T_j+\kL_j],$$
	and (i) follows. 	For (ii), notice that $ S_j \not \in I_{j-1}$ since $\tilde{I}_j\cap I_{j-1}=\emptyset$, and $T_j-S_j \leq 2 \kL_j$. Hence, thanks to the maximal property of  $\kL_j$,   we have
	$${\rm T}(S_j, T_j+\kL_j)\leq {\rm T}(S_j, S_j+4\kL_j)\leq 16\kL_j^2.$$
Finally we consider (iii). Assume that $i<j$.  Then  $x_j\notin (x_i-\kL_i,x_i+\kL_i)$ since $x_j\notin I_{j-1}$. Hence,  $|x_i-x_j|\geq \kL_i =\max\{\kL_i,\kL_j\}$ since $\kL_i\geq \kL_j$, and  (iii) follows.	
\end{proof}
We introduce some notations: for any $ k \geq 1$ and for $\alpha >0$,
\al{
a_k&:=\#\{j\leq \ell : \kL_j=2^k\};\\
N_{\alpha}&:={\lceil \log_2(\alpha n)/2\rceil.}
}		
Fix $\varepsilon>0$, which is chosen later (see \eqref{eda} below).
Then,
\ben{ \label{ired}
\bigcup_{i\in \mathbf{Red}} [i,i+2^{N}]\subset \bigcup_{j=1}^{\ell} [S_j,T_j+\kL_j] = \bigcup_{j: \kL_j \leq 2^{N_{\e}}} [S_j, T_j + \kL_j] \cup  \bigcup_{j: \kL_j > 2^{N_{\e}}} [S_j, T_j + \kL_j], 
}

The lemma below helps control the passage time of intervals $[S_j, T_j+\kL_j]$ with $\kL_j \leq 2^{N_{\e}}$.         
            \begin{lem} \label{lem:sak}
  There exists a universal constant $c_0>0$ such that the following statements hold: 
  \begin{itemize}
  	\item [(i)] For  any $\delta>0$ and $\e>0$, and  for all $n\in\N$ large enough,
  \bea{
  \pp \left(\sum_{j=1}^{\ell}  \rmT(S_j, T_j + \kL_j) \1\{ \kL_j \leq 2^{N_{\e}} \} \geq \delta n\right) 
  & \leq&   \pp \left(\sum_{i=1}^{\ell}\kL_i^2 \1\{ \kL_i \leq 2^{N_{\e}} \} \geq \delta n/16\right) \\
  & \leq& \exp(-c_0\delta \sqrt{n/\e}).
  }
\item[(ii)] For any $K\geq 1$ and $\e>0$, and  for all $n\in\N$ large enough,
\be{
\pp \left( a_k = 0 \, \forall \, k\geq N_K; \, \sum_{k=N_{\e}}^{N_K} a_k \leq K  \right) \geq 1 -  \exp(-c_0K\sqrt{\e n}). 
}
  \end{itemize}
  \end{lem}
\begin{proof} We first recall that  for each $j$ we can find $x_j \in \Iintv{0,n}$ such that $\rmT(x_j, x_j + \kL_j) \geq \kL_j^2$. Hence, by the definition of $(a_k)_{k \geq1}$, for each $k$ there is a sequence of points $(x_j^k)_{j=1}^{a_k}$ such that 
	\ben{ \label{txjk2}
		\rmT(x_j^k, x_j^k+2^k) \geq 2^{2k}.	
	} 
	For $x \in \Z$ and $t \in \R$, we write 
 $B(x,t):=[x-t,x+t].$ Moreover, 	by Lemma~\ref{lem:sti1d}-(iii),
	\ben{ \label{bxjkp}
		B(x^k_j,2^k/3)\cap B(x^{k'}_{j'},2^{k'}/3) =\varnothing, \, \, \forall \, (k,j) \neq (k',j').
	}
In addition, by Lemma~\ref{lem:taln}-(i), there exists an universal constant $c$ such that for all  $k\in \N$,
\ben{ \label{ptyj}
 \pp\Big({\rm T}_{B(0,2^k/3)}(0,2^k )\geq  2^{2k} \Big) \\
 \leq  \exp\left(-c2^{k}\right).
}
We fix $\delta, \e >0$. The first inequality in (i) directly follows from Lemma~\ref{lem:sti1d}-(ii).

For simplicity of notation, we set $\delta':=\delta/32$. We observe that if $a_k \leq \delta' n 2^{-(k+N_{\e})}$ for any $k\in \Iintv{N,N_\e}$, then 
\be{
\sum_{i=1}^{\ell}\kL_i^2 \1\{ \kL_i \leq 2^{N_{\e}} \} = \sum_{k=N}^{N_\e} a_k 2^{2k} \leq\delta n/16.
}
Therefore, using the union bound,
	\ben{ \label{akne}
 \pp \left(\sum_{i=1}^{\ell}\kL_i^2 \1\{ \kL_i \leq 2^{N_{\e}} \} \geq \delta n/16\right)  \leq N_{\e}	\max_{N \leq k \leq N_\e}\pp \left(a_k \geq \delta' n 2^{-(k+N_{\e})} \right).
}
To estimate the last probability in \eqref{akne}, we fix $N \leq k \leq N_\e$ and define
\be{
	\kB_{k,\e,\delta} := \Iintv{\delta' n 2^{-(k+N_{\e})},  n}.
} 
Given $b_k \in \kB_{k,\e,\delta}$, we define
\be{
	\kB(b_k)  :=  \Big\{ \by=(y_j)_{j=1}^{b_k} \subset \Iintv{0,  n}: 
	B(y_j,2^k/3)\cap B(y_{j'},2^{k}/3) = \varnothing \,\, \forall \, \, 1 \leq  j\neq j' \leq b_k \Big\}.
}
Then we have for all $b_k \in \kB_{k,\e,\delta}$,
\ben{ \label{nobbk}
	\# \kB(b_k) \leq (n+1)^{b_k} \leq \exp(2b_k \log n).
}
Remark that $a_k \leq n$, and thus using \eqref{txjk2} and \eqref{bxjkp}
\bea{
	\pp(a_k \geq \delta' n 2^{-(k+N_{\e})})
	&\leq& \sum_{b_k \in \kB_{k,\e,\delta}} \, \sum_{\by \in \kB(b_k)}   \pp\left(\forall \, 1 \leq j \leq b_k,\, \rmT(y_j,y_j+2^k) \geq 2^{2k}  \right).
}
Observe that $\rmT_A$ is independent of $\rmT_B$ if $A \cap B =\varnothing$. Therefore, for each  $\by \in \kB(b_k)$, 
\bea{
	\pp\Big( \forall j\leq b_k, \,
	~{\rm T}(y_i,y_j+2^k )\geq  2^{2k} \Big) &\leq& 		\pp\Big( \forall j\leq b_k, \,
	~{\rm T}_{B(y_j,2^k/3)}(y_i,y_j+2^k )\geq  2^{2k} \Big) \\
	&=&  \prod_{j=1}^{b_k} \pp\Big({\rm T}_{B(y_j,2^k/3)}(y_j,y_j+2^k )\geq  2^{2k} \Big) \\
	& \leq & \exp\left(-c b_k2^{k}\right), 
}
by using \eqref{ptyj}.   Combining the last two inequalities, we arrive  at
\bea{
	\pp(a_k \geq \delta' n 2^{-(k+N_{\e})}) &\leq& \sum_{b_k \in \kB_{k,\e,\delta}} \, \# \kB(b_k)\, \exp\Big(-c b_k2^{k}\Big)\\
	 &\leq& \sum_{b_k \in \kB_{k,\e,\delta}} \,  \exp\Big(-c b_k2^{k-1}\Big)\leq \exp\left(- 2^{-7}c \delta  \sqrt{n/\e}\right),
}
where   we have used \eqref{nobbk} with $2^k\geq (\log{n})^2$ for $k\geq N$ in the second inequality, and that   $\# \kB_{k,\e,\delta} \leq n+1$ 
 and $b_k2^k \geq \delta n2^{-N_\e -5}\geq 2^{-6} \delta  \sqrt{n/\e}$ for all $b_k \in \kB_{k,\e,\delta}$ in the last line. Combining this estimate with \eqref{akne}, we obtain (i).

We next consider (ii). Observe that  
\begin{align}\label{pak1}
\begin{split}
    \P(\exists \, k \geq N_K: \, a_k \geq 1)
    &\leq \sum_{k \geq N_K}\P(a_k \geq 1)  \\
    &\leq \sum_{k \geq N_K}\P(\exists i\in \Iintv{0,n};~{\rm T}(i,i+2^k)\geq 2^{2k})\\
    &\leq \sum_{k\geq N_K} (n+1) \pp(\T(0,2^k)\geq 2^{2k})\\
    & \leq (n+1) \sum_{k\geq N_K}\exp(- c 2^k) \leq 2(n+1) \exp(-c\sqrt{Kn}),
\end{split}
\end{align}
by using \eqref{ptyj} and  $2^{N_K} \geq \sqrt{Kn}$. Now we consider the event that $\sum_{k=N_{\e}}^{N_K} a_k \geq K$. Define 
\be{
	\kB_{\e,K}:=\Big\{ \boldsymbol{b}=(b_k)_{k=N_{\e}}^{N_K} \subset  \Iintv{0,n}: \, \, \sum_{k=N_{\e}}^{N_K} b_k \geq K \Big \},
}
and for any $\bb \in \kB_{\e,K}$, we set 
\be{
	\kB(\bb) := \Big \{ \boldsymbol{y} = (y^k_j)_{ \substack{1 \leq j \leq b_k \\ N_{\e} \leq k \leq N_K}} \subset \Iintv{0,n}: B(y^k_j,2^k/3)\cap B(y^{k'}_{j'},2^{k'}/3) = \varnothing \,\, \forall (k,j) \neq (k',j')  \Big \}.
}
It is straightforward that 
\ben{ \label{nobek}
	\# \kB_{\e,K} \leq (n+1)^{N_K} \leq \exp((\log n)^3),
}
and
\ben{ \label{nob}
	\#\kB(\bb) \leq \prod_{k=N_\e}^{N_K}(n+1)^{b_k} \leq \exp \Big(2 \sum_{k=N_\e}^{N_K} b_k \log n \Big).
}
Using the same argument for Part (i), for each $\bb \in \kB_{\e,K}$ and $\by\in\kB(\bb)$, we have 
\bea{
	&&	\pp\left( \forall \, N_{\e} \leq k \leq N_K, \,\forall \, j\leq b_k, \,
	~{\rm T}(y^k_j,y^k_j+2^k )\geq  2^{2k} \right)
	\\
	&\leq&\pp\left( \forall \, N_{\e} \leq k \leq N_K, \,\forall \, j\leq b_k, \,
	~{\rm T}_{B(y^k_j,2^k/3)}(y^k_j,y^k_j+2^k )\geq  2^{2k} \right)
	\\
	&=&\prod_{k=N_{\e}}^{N_K} \prod_{j=1}^{b_k} \pp\left( 
	~{\rm T}_{B(y^k_j,2^k/3)}(y^k_j,y^k_j+2^k )\geq  2^{2k} \right) \\
	&\leq&  \exp\left(-c \sum_{k=N_{\e}}^{N_K}b_k2^{k} \right). 
}
Therefore,  by using the union bound and \eqref{nob},
\bea{ \label{pbkk}
	\pp \left( \sum_{k=N_{\e}}^{N_K}a_k  \geq K\right) 
	&\leq& \sum_{\bb \in \kB_{\e,K}} \sum_{\by \in \kB(\bb)}  \pp\left( \forall \, N_{\e} \leq k \leq N_K, \,\forall \, j\leq b_k, \,
	~{\rm T}(y^k_j,y^k_j+2^k )\geq  2^{2k} \right) \notag 
	\\
	&\leq&   \sum_{\bb \in \kB_{\e,K}} \# \kB(\bb) \exp\left(-c \sum_{k=N_{\e}}^{N_K}b_k2^{k} \right) \notag\\
	&\leq& \sum_{\bb \in \kB_{\e,K}} \exp\left(-c \sum_{k=N_{\e}}^{N_K}b_k2^{k-1}  \right).
}
Moreover, using $N_{\e}= 
\lceil \log_2(\sqrt{\e n}) \rceil$, we have $\sum_{k=N_{\e}}^{N_K}b_k2^{k-1} \geq 2^{N_\e-1} \sum_{k=N_{\e}}^{N_K}b_k \geq K 2^{N_\e-1} \geq K\sqrt{\e n}/2$ for any $\bb \in \kB_{\e, K}$. Hence, the last display equation together with \eqref{nobek} implies that 
\bea{
\pp \left( \sum_{k=N_{\e}}^{N_K}a_k  \geq K\right) 	\leq \#  \kB_{\e,K}  \exp\left(-cK\sqrt{\e n}/2   \right) \leq  \exp\left(-cK\sqrt{\e n}/4\right).
}
Combining this  estimate with \eqref{pak1}, we obtain (ii). 
\end{proof}

We prepare a lemma that tells us how to group intervals.
\begin{lem} \label{lem:int}
  For any $R>0$ and a sequence of intervals $([x_i,y_i])_{i=1}^m$, 
  there exists a sequence of intervals $([s_i,t_i])_{i=1}^{m'}$ with $m'\leq m$ such that 
  \begin{itemize}
  \item  $(s_i)^{m'}_{i=1}\subset (x_j)_{j=1}^m$ and $(t_i)_{i=1}^{m'}\subset (y_j)_{j=1}^m$,
   \item $ \sum_{i=1}^{m'} |t_i-s_i|\leq 2mR+\sum_{i=1}^{m}|y_i-x_i|,$
  \item $ {\rm d}([s_i,t_i],[s_j,t_j])\geq R$ for all  $1 \leq i\neq j\leq m'$,
  \item $\cup_{i=1}^m [x_i,y_i] \subset \cup_{i=1}^{m'} [s_i,t_i].$
    \end{itemize}
  \end{lem}
\begin{proof} We write $A_i:=[x_i,y_i]$. We define an equivalent relation on $\{1,\ldots,m\}$ as follows.   Given $1 \leq i,j\leq m$, we write  $i\sim j$ if there exist $(i_k)_{k=1}^r\subset [m]$ with $i_1:=i,i_r:=j$ such that $\max_{k\in [r-1]}{\rm d}( A_{i_k},A_{i_{k+1}})\leq R$. It is not hard to check that  $\sim$ is an equivalent relation. Given $p\in C:= \{1,\ldots,m \}/\sim$, we define
\begin{align*}
    s_{p}:=\min\{x_i:~i\in p\},\quad t_{p}:=\max\{y_i:~i\in p\}.
\end{align*}
Note that by construction,
\begin{align}\label{basic eq}
    B_p:=[s_p,t_q]\subset \bigcup_{i\in p} [x_i-R,y_i+R].
\end{align}
We will prove that $([s_p,t_p])_{p\in C}$  satisfies the desired properties. Note that $m':=|C|\leq m$.  By construction, since
\begin{align*}
    \bigcup_{i=1}^m A_i \subset \bigcup_{p \in C} B_p\subset \bigcup_{i=1}^m [x_i-R,y_i+R],
\end{align*}
the first, second and fourth conditions follow.  We prove the third one. Let $p\neq q$. Without loss of generality, we suppose $t_p\leq t_q$. Let $i\in p$ be such that $t_p\in A_i$. If $s_q< t_p+R$, by \eqref{basic eq}, then there exists $x'\in A_j$ with $j\in q$ such that $x'\in [t_p,t_p+R]$, which implies  $A_i\sim A_j$ and derives a contradiction. Thus, we have $s_q\geq t_p+R$ and ${\rm d}(B_p,B_q)\geq R.$
  \end{proof}
In the next lemma, we show that with overwhelmed probability, we can find a  covering composing of elements in $\kS_M$ (defined in Lemma \ref{lem:keyup} with some $M$ large)  for the  intervals $[S_j,T_j+\kL_j]$ with $\kL_j\geq 2^{N_{\e}}$.  

\begin{lem} \label{lem:akl}
  For any $A>1$ and $\e>0$, there exists $M_4=M_4(\e, A)$ such that for $M\geq M_4$ and $n\in\N$ sufficiently large, 
  \be{ 
 \pp(\kE_{\bf cov}) \geq 1-\exp(-A\sqrt{n}); 
 }
where
\be{
 \kE_{\bf cov}:= \biggl\{ \exists (s_i,t_i)_{i=1}^m \in \kS_M(n); \bigcup_{j: \kL_j \geq 2^{N_{\e}}} [S_j, T_j + \kL_j]  \subset \bigcup_{i=1}^m [s_i, t_i]  \biggr\},
}
and $\kS_M(n)$ is given in Lemma \ref{lem:keyup}, that is,
\begin{align*}
    \kS_M(n):=\left\{ (s_i,t_i)_{i=1}^m:
    \hspace{-2em}
    \begin{minipage}{22em}
        \begin{itemize}
        \setlength{\itemsep}{0.3em}
            \item $m \in \Iintv{1,M}$,
            \item $s_i,t_i \in \Iintv{0,n}$ and $s_i<t_i$ for all $i \in \Iintv{1,m}$,
            \item ${\rm d}([s_i,t_i], [s_j,t_j])>M^3\sqrt{n}$ for all $i<j$,
            \item $\sum_{i=1}^m |t_i-s_i|\leq M^6\sqrt{n}$
        \end{itemize}
    \end{minipage}
    \right\}.
\end{align*}
\end{lem}
\begin{proof} Fix $A>1$ and $\e>0$. Let $c_0$ be a positive constant as in Lemma \ref{lem:sak}. We set 
	\be{
M_4:=M_4(\e,A):= c_0A/\e^2.	
}
 By Part (ii) of this lemma,  for $M\geq M_4$, if  $n\in\N$ is  large enough, then
	\ben{ \label{pea}
		\pp(\kE) \geq 1 -\exp(-c_0M \sqrt{\e n}) \geq 1- \exp(-A\sqrt{n}),  
	}
where 
\be{
 \kE := \{a_k =0 \,\forall \, k \geq N_M \} \cap \Bigl \{ \sum_{k=N_{\e}}^{N_M} a_k \leq M \Bigr\}.
}
Hence, it suffices to show that 
$ \kE \subset \kE_{\bf cov}.$ 
Let $M \geq M_4(\e,A)$. Suppose that $\kE$ occurs. Then $\#\{i: \kL_i \geq 2^{N_\e}\} \leq M$. Thus applying Lemma \ref{lem:int} with $R= M^3 \sqrt{n}$ to the sequence of intervals $([S_j,T_j+\kL_j])_{j:~\kL_j\geq 2^{N_\e}}$,  we can find $([s_i,t_i])_{i=1}^m$ with $m \leq M$ satisfying
\be{
 {\rm d}(B_i,B_j) \geq M^3\sqrt{n} \quad \forall i\neq j; \qquad \bigcup_{j:  \kL_j \geq 2^{N_\e}} [S_j,T_j +\kL_j] \subset \bigcup_{i=1}^m [s_i,t_i].
}
Moreover, since $N_M=\lceil \log_2(\sqrt{Mn})\rceil$, on the event  $\kE$,
\bea{
  \sum_{i=1}^m |t_i-s_i| &\leq& 2mM^3\sqrt{n} + \sum_{j: 
  \kL_j \geq 2^{N_\e}} |T_j+\kL_j-S_j| \\
&\leq&M^5\sqrt{n} + 2\sum_{j: 
	\kL_j \geq 2^{N_\e}} \kL_j\leq M^5\sqrt{n} +2M 2^{N_M}
\leq M^6\sqrt{n}.
}
 Hence, $\kE_{\bf cov}$ occurs and  we have $ \kE \subset \kE_{\bf cov}.$ 
\end{proof}

\begin{proof}[\bf Proof of Lemma~\ref{lem:keyup}]
Let $c_0$ be a positive constant as in Lemma  \ref{lem:sak}. We set 
\ben{ \label{eda}
	\e:=\e(\delta,A) := (c_0 \delta/64A)^2.
}
Using  Lemma  \ref{lem:sak}-(i), we have 
\ben{ \label{p33}
	\pp \left( 	\sum_{i: \kL_i <2^{N_\e}} {\rm T}(S_i,{\rm T}_i+\kL_i) > \delta n \right) \leq \exp(-2A\sqrt{n}).
}
Let us define   
\be{
\kE_{\bf red}: = \kE_{\bf cov}\cap \left\{ 	\sum_{i: \kL_i <2^{N_\e}} {\rm T}(S_i,{\rm T}_i+\kL_i)  \leq \delta n \right\},
}
where $\kE_{\bf cov}$ is the event in Lemma \ref{lem:akl}, and 
\ben{ \label{eyellow}
\kE_{\mathbf{blue}}:= \Big \{ \sum_{i\in \mathbf{Blue}} \rmT(i,i+2^{N})\leq (\mu+2\delta) n \Big \}.
}
 Using \eqref{p33}, Lemma \ref{lem:akl} and Lemma~\ref{lem:syi},   there exists $M_4=M_4(\e,2A)$ such that for $M\geq M_4$, if $n$ is large enough,
\begin{align}\label{p34}
    \pp(\kE_{\bf red}^c)+\pp(\kE^c_{\mathbf{blue}})
    \leq 2\exp(-2A \sqrt{n})+\exp(-n^{2/3})
    \leq \exp(-A\sqrt{n}).
\end{align}
We remark that by \eqref{ired}, on the event $\kE_{\bf red}$ with $(s_i,t_i)_{i=1}^m$,
\al{ 
	[0,n]& \subset \bigcup_{i \in \kN} [i, i+2^{N}] =  \bigcup_{i \in \mathbf{Blue}} [i, i+2^{N}] \,  \cup \, \bigcup_{i\in \mathbf{Red}} [i, i+2^{N}]\\
	&\subset  \bigcup_{i \in \mathbf{Blue}} [i, i+2^{N}]  \,  \cup \bigcup_{i:\kL_i < 2^{N_\e}} [S_i,{\rm T}_i+\kL_i] \cup \bigcup_{i=1}^m [s_i,t_i].
} 
Hence on the event $\kE_{\bf red} \cap \kE_{\mathbf{blue}}$, ${\rm T} (0,n)\leq n(\mu+3\delta)+ \sum_{i=1}^m {\rm T}(s_i,t_i).$ 
Therefore, we have
\begin{align}\label{p1234}
\begin{split}
    \pp(\T(0,n) \geq (\mu+\xi)n)
    &\leq \pp\left(\exists (s_i,t_i)_{i=1}^m \in \kS_M(n); \sum_{i=1}^m \rmT(s_i,t_i) \geq (\xi-3 \delta)n \right)\\
    &\quad+ \P\left(\kE_{\bf red}^c\right)+\P(\kE_{\mathbf{blue}}^c)\\
    &\leq \sum_{(s_i,t_i)_{i=1}^m \in \kS_M(n)} \P\left(\sum_{i=1}^m \rmT(s_i,t_i) \geq (\xi-3 \delta)n\right)+\exp(-A \sqrt{n}).
\end{split}
\end{align}
Thus, since $|\kS_M(n)|\leq n^{2M}$, Lemma \ref{lem:keyup} follows by taking $M_1=M_4(\e,2A)$.
\end{proof}

\subsection{Proof of Proposition~\ref{prop:ts2}-(iii)}\label{subsect:prop3.1-3}
Applying Lemma~\ref{lem:keyup} with $A=1$ and $\delta=\xi/2$, letting $M=M_1(1,\xi/2)$ as in  this lemma,  we have
\begin{align*}
    \P(\T(0,n)>(\mu+\xi)n)
    &\leq e^{-n}+n^{2M}\sum_{x\in\Iintv{0,n}}\P\Bigl( \T(x,x+ \lfloor M^6\sqrt{n} \rfloor)>\xi n/(2M) \Bigr)\\
    &\leq e^{-n}+n^{3M}\P\Bigl( \T(0, \lfloor M^6\sqrt{n} \rfloor)>\xi n/(2M) \Bigr).
\end{align*}
By Lemma~\ref{lem:taln}-(ii), the last probability is bounded from above by $e^{-c\sqrt{n}}$ with some positive constant $c=c(\xi,M)$, which yields the claim. \qed

\subsection{Proof of Proposition~\ref{prop:ts2}-(i)}\label{subsect:prop3.1-1}
We first prove that there exists a universal constant $c>0$ such that for any $m\in\N$,
\ben{ \label{cljp}
\pp(t(0,m)=\rmT (0,m)) \leq \exp(-cm^{2/3}),
}
Let $M_2=M_2(1)$ as in Corollary~\ref{corollary}. Since $\pp(t(0,x)\leq h) \leq \exp(-c_0x^2/h)$ with some universal constant $c_0>0$ by Lemma~\ref{lemrw}-(iii) and $\rmT(0,m) \leq \rmT(0, \lfloor m^{4/3} \rfloor)$,
\bea{
\pp(t(0,m)=\rmT(0,m)) &\leq& \pp(t(0,m)\leq M_2 m^{4/3})+ \pp(\rmT(0,m)\geq M_2 m^{4/3})\\
&\leq&\exp(-c_0 m^{2/3}/M_2) + \pp(\rmT(0, \lfloor m^{4/3} \rfloor)\geq M_2 m^{4/3}),
}
By Corollary~\ref{corollary},
\be{
\pp(\rmT(0, \lfloor m^{4/3} \rfloor)\geq M_2 m^{4/3}) \leq \exp(-m^{2/3}).
}
Hence, combining the last two display equations, we get \eqref{cljp}. 

 We take $\e:=\frac{\delta}{3\mu}$ so that 
 \ben{ \label{muee}
 	\mu(1-\e)=\mu-\delta/3.
 }
 Define 
 \bea{
\kE_{\Delta} :=  \{ \textrm{All the optimal paths from $0$ to $n$ must visit $\Delta$}\}, \text{ with }\Delta :=\{ k\in\Z: ~M\sqrt{n} \leq k \leq  \e n\}.
}
On the event $\kE_{\Delta}^c$, there is $x \leq M\sqrt{n}$ and $y\geq \e n$ such that $t(x,y)=\rmT(x,y)$. Thus,
 \bean{ \label{peclb}
\pp(\kE_{\Delta}^c) &\leq& \pp(\rmT(0,n)\geq n^2) + \pp(\kE_{\Delta}^c; \rmT(0,n) \leq n^2) \notag \\ 
&\leq& \pp(\rmT(0,n)\geq n^2) +  \pp(\exists  x \in \Iintv{-n^2,M\sqrt{n}}, \,\exists y\in\Iintv{\e n ,  n};~ t(x,y)=\rmT(x,y)) \notag \\
&\leq& \pp(\rmT(0,n)\geq n^2) + \sum_{-n^2 \leq x \leq M\sqrt{n}}\,\,\sum_{\e n \leq y\leq n} \pp(t(x,y)=\rmT(x,y))\notag \\
&\leq& \exp(-c_1 n) + \exp(-c_1 n^{2/3}),
}
with some $c_1=c_1(\e)>0$, by using Lemma \ref{lem:taln}-(ii) and \eqref{cljp}. 

By the lower tail large deviation \cite[Theorem 1]{BR}, for any $\varepsilon>0$, there exists $c_2>0$ such that for any $m\in\N$,
\be{ 
\pp(\rmT(0,m) < (1-\e)\E[\rmT(0,m)]) \leq \exp(-c_2 m),
} 
Moreover,   by \eqref{muee}, for all $k \in \Delta$, we have
 $$(1-\e)\E[\T(k,n)] \geq  (1-\e)\E[\T(\lfloor \e n \rfloor, n)] =\mu(1-\e)^2n-o(n) \geq (\mu-\delta)n.$$
 Therefore,  we reach for some $c_3=c_3(\delta)>0$,
\ben{\label{lldm}
\pp(\rmT(k,n)< (\mu-\delta)n) \leq \exp(-c_3 n).
} 
Finally, we observe that  
    \bea{
 \pp(\T(0,n) \geq (\mu+\xi)n) &\geq& \pp(\T(0,n) \geq (\mu+\xi)n, \,\kE_{\Delta}) \\
  & \geq  & \pp(\T(0, \lfloor M\sqrt{n} \rfloor) \geq (\xi+\delta)n, \T(k, n)\geq (\mu-\delta)n  \, \forall \, k \in \Delta, \, \kE_{\Delta}) \\
  &\geq & \pp(\T(0, \lfloor M\sqrt{n} \rfloor) \geq (\xi+\delta)n) - \sum_{k \in \Delta} \pp( \T(k, n)< (\mu-\delta)n) -\pp(\kE_{\Delta}^c) \\
  & \geq & \pp(\T(0, \lfloor M\sqrt{n} \rfloor) \geq  (\xi+\delta)n) -\exp(-c_4 n^{2/3}),
 }
for some $c_4=c_4(\delta,\e,\mu)>0$,  where we have used \eqref{peclb} and  \eqref{lldm}. \hfill $\square$ 

\section{Energy approximation by step functions: Proof of Proposition \ref{prop:s}} \label{sec:props}
For the convenience, we  recall the definition of the energy functional
\be{
	E(f) := -\int_{\R} \log \theta_f(x) \, \dd x, \qquad \theta_f(x) :=\pp_x^{\textrm{BM}} (\tau_y \geq f(y)-f(x) \, \forall \, y \in \R),
}
and our goal is to prove that 
\ben{ \label{gos6}
\inf_{f \in \kC(\xi)} E(f) = \inf_{f \in \kC^{\rm Step}(\xi)} E(f).
}
We  also recall a result from Lemma~\ref{thm2} that will be used frequently in this section:
\ben{\label{lem 1.2 revisited}
E(f)=\sqrt{\xi}E(f_\xi),\quad\text{$f_\xi(x):=\xi^{-1}f(\sqrt{\xi}x)$ for any $f\in \kC(1)$ and $\xi>0$}.
}
Hence, we only need to prove the claim \eqref{gos6} with $\xi=1$.
Given parameters $\varepsilon, \delta >0$, our aim is to  deform a function $f \in \kC(1)$ to a step function $g \in \kC^{\rm Step}(1 - \varepsilon)$ such that $E(g) \geq E(f) -\delta$. Then  letting $\varepsilon, \delta \rightarrow 0$, we can validate the claim of Proposition \ref{prop:s}. The primary strategy involves the integration of two types of transformations: soft and hard deformations. To illustrate, suppose that we have to deform a function $f$ over a finite interval $I\subset \R$.  Then the hard deformation  simply forces the function $f$ to the minimum value over $I$, that is
\be{
	g(x) = \begin{cases}
		f(x) & \textrm{ if } x \not \in I\\
		\min_{y \in I} f(y) & \textrm{ if } x \in I.
	\end{cases}
} 
This deformation does not change the maximum value of $f$ and a lower bound of the change of energy is given by
\ben{ \label{engg1}
	E(f) - E(g)  \geq \int_I \log \frac{\theta_g(x)}{\theta_f(x)} \dd x, 
}
since $g(x) \leq f (x)$ for all $x \in \R$ and $f(x) =g(x)$ for all $x \not \in I$. In particular, if the length of $I$ is $1/n$, then
\ben{ \label{engg2}
	E(f) - E(g)  \geq (\kO(1)+\log n)/n. 
}
The soft deformation defined in Lemma \ref{lem: deformation of function} below   is more complicated, which gradually flatten the function $f$ to get a function $g$ with lower energy: $	E(g) \leq E(f)$, and controllable height (and so the maximum): for all $x \in \R$
\ben{ \label{goh}
	g(x) \geq f(x) - \Delta_I(f), 
}
where we define for each interval  $I \subset R$, 
\ben{
\Delta_I(f) := M_I(f)-m_I(f), \qquad  M_I(f) :=\sup_{y \in I} f(y), \quad m_I(f):= \inf_{y \in I} f(y).
}
We will partition the primary interval $[-M,M]$, where $M$ is appropriately large based on $\delta$ and $\varepsilon$, into subintervals of length $1/n$. These subintervals will be categorized into three types based on their height:

\begin{enumerate}
\item[1.] {\bf Large height}, if $\Delta_I(f) \geq C_* \log n/n$,
\item[2.] {\bf Moderate height}, if $\Delta_I(f) \in [c_* \log n/n, C_* \log n/n)$, 
\item[3.] {\bf Small height}, if $\Delta_I(f) \leq c_* \log n/n$.
\end{enumerate}
Here, $C_*$ and $c_*$ are appropriately selected constants to be chosen later.

For intervals with large height, we apply the hard deformation and manage the total energy gaps using  \eqref{engg2}. Note that the number of intervals with large height is at most $ n/ C_*\log n$. The details of this part will be presented in Proposition \ref{prop:largeheight}. On the other hand,  we apply a soft deformation to the function over intervals with small heights. By its definition, the newly formed function possesses a lower energy. Using Eq. \eqref{goh}, we can manage the difference in height post-deformation. This will be addressed in Proposition \ref{prop:smallheight}. 

The intervals of moderate height present the most significant challenge. It is not immediately clear whether to apply a soft or hard deformation to each. Rather than making this decision for each individual interval, we will group these moderate height intervals into larger clusters. The choice of transformation for each group will then depend on certain criteria related to the size of its group (i.e., the number of intervals it contains) and height (i.e., the difference in the values of the function over the group). This approach will be detailed in Proposition \ref{prop:rellarge}.

The structure of this section is outlined as follows. The subsequent subsection will present a summary of some preliminary results. The proof of Proposition \ref{prop:s} is provided in subsection \ref{subsect:approx_pf}, and it relies on several other results, specifically Propositions \ref{prop:truncation} through \ref{prop:smallheight}. The proofs for these propositions can be found in subsections \ref{subsect:pfc-l} and \ref{subsect:pfrellarge}.

\subsection{The two deformations and preliminaries}\label{subsect:useful_est}

We present two key deformations allowing us flatten the function with controllable  energy and height. Let $\mathcal{J}= \{I_i\}_{i=1}^{\ell}$ be finite disjoint intervals, where $I_i$ is of form $(a_i,b_i)$,  $[a_i,b_i)$, $(a_i,b_i]$, or $[a_i,b_i]$ with $a_i<b_i$ ($a_i$ and $b_i$ may take values $\pm\infty$.) \\

\begin{figure}
	\includegraphics[width=.4\linewidth]{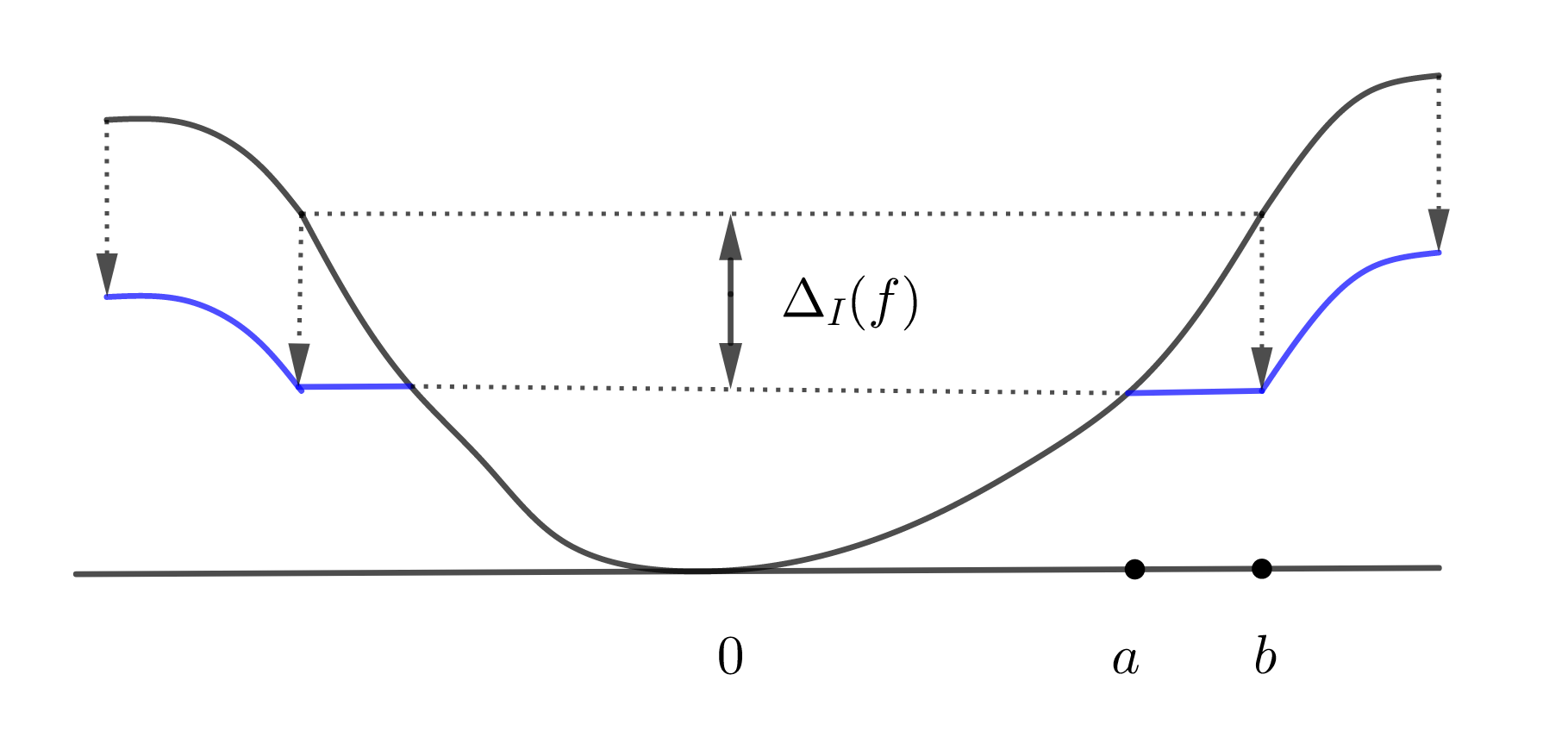}\hspace{4mm}
	\includegraphics[width=.38\linewidth]{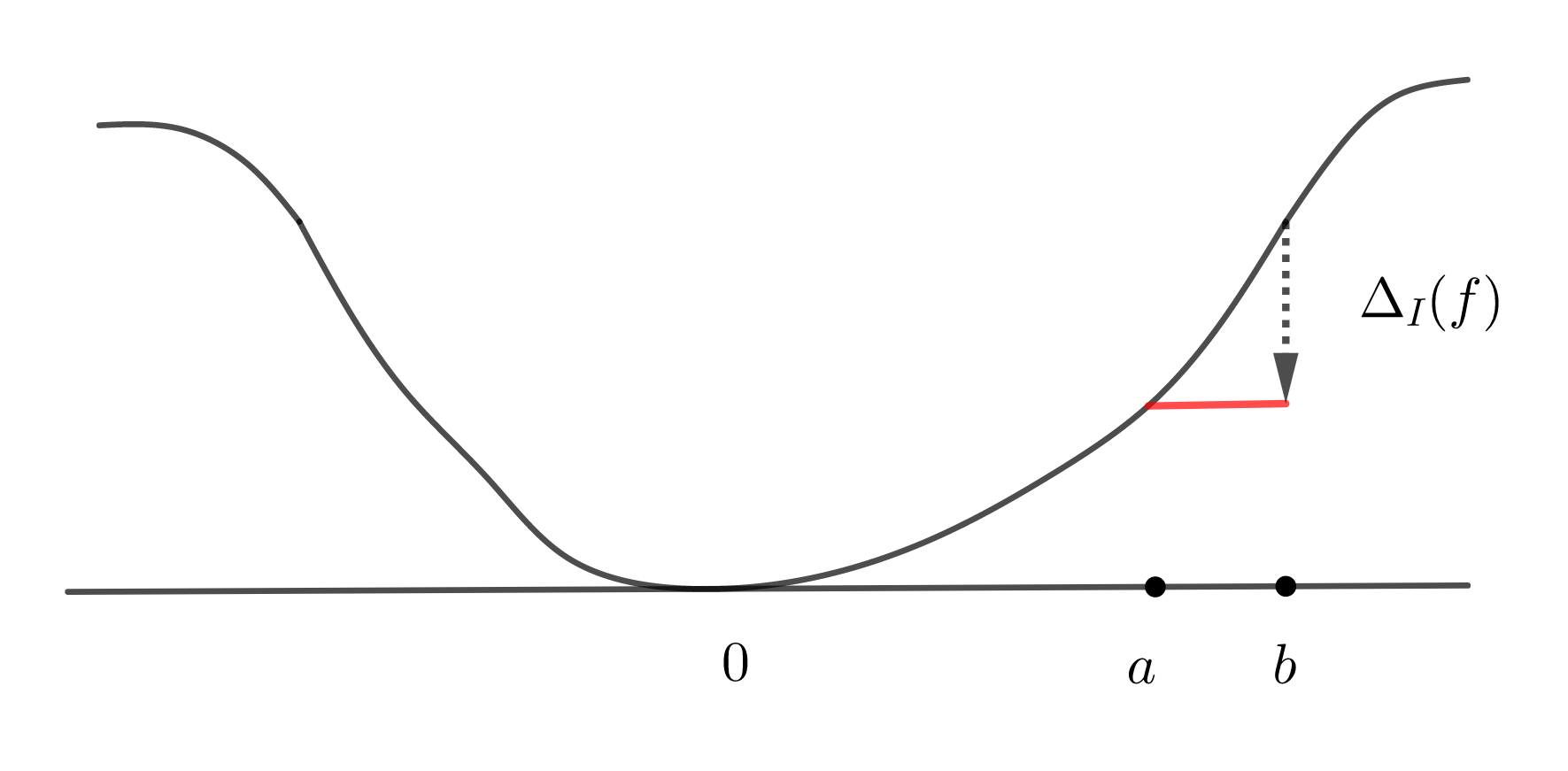}
	\caption{Soft deformation over $I=[a,b]$;\qquad \qquad Hard deformation over $I=[a,b]$\qquad}
\end{figure}
\noindent {\bf The hard deformation} of $f\in \kC(1)$ over $\kJ$, denoted by $f^{\textrm{hd},\kJ}$, is a function given as 
\ben{
	\fhd(x) :=\begin{cases} 
		\inf_{y \in I} f(y) & \textrm{ if } x \in \cup_{I\in \kJ} I, \\
		f(x) & \textrm{ otherwise}.
	\end{cases}
}
\begin{lem} \label{lem:hd}
The following hold:
	\begin{itemize}
		\item [(i)] If 	 $(g(x)-g(y))_+ \leq (f(x)-f(y))_+$ for all $x$, $y \in \R$ then $E(f) \geq E(g)$.
		\item[(ii)] There exists a positive constant $C$ such that for all $f \in \kC(1)$ and $x$, $y \in \R$  with $f(y)>f(x)$ 
		\be{
			\theta_f(x) \leq  \frac{C|x-y|}{\sqrt{f(y)-f(x)}}.
		}
  \item[(iii)]  If $g(x) \leq f(x)$ for all $x \in \R$ then
	\be{
		 E(g)-E(f) \geq \int_{\{g<f\}} \, \log \frac{\theta_g(x)}{\theta_f(x)} \,  \dd x.
	} 
	As a consequence, 
	\be{
		E(\fhd)-E(f)\geq \sum_{I \in \kJ}  \int_I \, \log \frac{\theta_{\fhd}(x)}{\theta_f(x)} \,  \dd x.
	}
 \end{itemize}
\end{lem}
\begin{proof}
	Parts (i) and (iii) directly follows from the definition of $E(f)$ and $\fhd\leq f$. 
  Using Lemma \ref{lem:maxbrown}
	\be{
		\theta_f(x) \leq \pp_x^{\textrm{BM}}(\tau_y \geq  f(y) -f (x))	=  \pp_0^{\textrm{BM}}(\tau_{|x-y|} \geq  f(y) -f (x)) \asymp \frac{|x-y|}{\sqrt{ f(y) -f (x)}},
	}
and thus (ii) follows.
\end{proof}
\noindent {\bf The soft deformation} $\fsd$ of $f\in \kC(1)$ over $\kJ$ will be defined  inductively as follows.   Set $f^{[0]}:=f$.
By induction in $k \geq 1$, $f^{[k]}$ is defined as:
\begin{align*}
	f^{[k]}(x):=
	\begin{cases}
		m_{I_k}(f^{[k-1]}), & \text{ if } f^{[k-1]}(x)\in [m_{I_k}(f^{[k-1]}),M_{I_k}(f^{[k-1]})],\\
		f^{[k-1]}(x), & \text{ if } f^{[k-1]}(x)<m_{I_k}(f^{[k-1]}),\\
		f^{[k-1]}(x)- \Delta_{I_k} (f^{[k-1]}) & \text{ if } f^{[k-1]}(x)>M_{I_k}(f^{[k-1]}).
	\end{cases}
\end{align*}
We set $\fsd:=f^{[\ell]}$ with $\ell:=|\kJ|$. 

\begin{lem}\label{lem: deformation of function}
	Suppose that $f \in \kC(1)$. The following hold:
	\begin{enumerate}
		\item [(i)]
		$\fsd|_{I} \equiv const$ for all $I \in \kJ$.
		\item [(ii)]
		For any $x,y\in\R$,
		\begin{align*}
			&\fsd(x)\geq f(x)-\sum_{I \in \kJ} \Delta_I(f) ,\quad\text{and}\\
			&(\fsd(x)-\fsd(y))_+\leq (f(x)-f(y))_+.
		\end{align*}
		As a consequence, $E(\fsd) \leq E(f)$.
	\end{enumerate}
\end{lem}

\begin{proof} Part (i) directly follows from the definition of  $\fsd$.  For Part (ii), by Lemma \ref{lem:hd}-(i),  it suffices to show that for any $k \leq \ell$ and $x,y\in\R$,
	\begin{align}
	&	f^{[k]}(x)\geq f^{[k-1]}(x)-\Delta_{I_k}(f),\quad\text{and}\label{induction one-point}\\
	&	(f^{[k]}(x)-f^{[k]}(y))_+ \leq (f^{[k-1]}(x)-f^{[k-1]}(y))_+.\label{induction two-point}
	\end{align}
	Since $\Delta_{I_k}(f^{[k-1]})\leq \Delta_{I_k}(f)$ by \eqref{induction two-point}, we have \eqref{induction one-point}. 
	Hence, our task is  to prove \eqref{induction two-point}. For simplicity of notation, we write $m_{k}:=m_{I_k}(f^{[k-1]})$ and $M_k:=M_{I_k}(f^{[k-1]})$.

	When $f^{[k-1]}(y)<m_k$, by $f^{[k]}(x) \leq f^{[k-1]}(x)$ and $f^{[k]}(y) = f^{[k-1]}(y)$, 
	\begin{align*}
		(f^{[k]}(x)-f^{[k]}(y))_+
		\leq (f^{[k-1]}(x)-f^{[k-1]}(y))_+.
	\end{align*}

	When $f^{[k-1]}(y)\in [m_k,M_k]$,
	\begin{align*}
		(f^{[k]}(x)-f^{[k]}(y))_+&=
		\begin{cases}
			0, & \text{if } f^{[k-1]}(x) \leq M_k,\\
			(f^{[k-1]}(x)-M_k)_+, & \text{otherwise},
		\end{cases}\\
		&\leq (f^{[k-1]}(x)-f^{[k-1]}(y))_+.
	\end{align*}

	We finally suppose $f^{[k-1]}(y)>M_k$. If $f^{[k-1]}(x) \leq M_k$, then \eqref{induction two-point} follows since $(f^{[k]}(x)-f^{[k]}(y))_+=0$. Otherwise, if $f^{[k-1]}(x)>M_k$, then $$f^{[k]}(x)=f^{[k-1]}(x)-(M_k-m_k),\quad f^{[k]}(x)=f^{[k-1]}(y)-(M_k-m_k).$$ 
Therefore, we have
	\begin{align*}
		(f^{[k]}(x)-f^{[k]}(y))_+=(f^{[k-1]}(x)-f^{[k-1]}(y))_+.
	\end{align*}
	Consequently, in all cases, \eqref{induction two-point} holds, and the lemma follows.
\end{proof}

The following  technical lemmas will be proved in Appendix.

\begin{lem}\label{lem:ning}
	
	For any  $\delta>0$, there exist  $c , \tilde{c} \in (0,1)$   such that for any interval $I \subset [\delta,\infty)$ with $|I| \leq 1$ and   $f \in \kC(1)$ satisfying $f|_{[-\delta, \delta]} \equiv 0$ and $f|_{I} \equiv const$, we have for any $x \in I$,
	\bea{
		\theta_f(x)
		\geq  c\, \pp^{\rm BM}_x (\tau_y \geq f(y)-f(x) \quad\forall y\geq \sup I) \geq  \tilde{c} \, (\sup I -x).
	}
\end{lem}
\begin{lem}\label{lem:comp} Let 
	 $b>a>0$ and  $f,\tilde{f}$ two increasing functions on  $[0,\infty)$ satisfying $\tilde{f}(x) \leq f(x)$ for any $x\geq a$. Let $\ell_{b,a}:=f(b)-f(a)$ and $\tilde{\ell}_{b,a}:=f(b)-\tilde{f}(a)$. It holds:
 \begin{align*}
(\tilde{\ell}_{b,a})^{3/2}\mathbb {P}_a^{\rm BM}(\tau_x \geq \tilde{f}(x) -\tf(a)\quad\forall x\geq b)
		\geq (\ell_{b,a})^{3/2}  \mathbb{P}_a^{\rm BM}(\tau_x\geq f(x) -f(a)\quad\forall x \geq b).
	\end{align*}
\end{lem}

\subsection{Proof of Proposition \ref{prop:s}}\label{subsect:approx_pf}
 Since the inequality
$
r_* \leq \inf\{E(f):f \in \mathcal{C}^{\rm Step}(1)\}
$
is always true, 
we now focus on proving  the converse inequality, that is 
\begin{align}\label{eq:xi=1}
	r_* \geq \inf \{ E(f):f \in \mathcal{C}^{\rm Step}(1)\}.
\end{align}
Given $M,\eta>0$, we denote by $\overline{\mathcal{C}}(M,\eta)$ the set of all functions $g:\R \to [0,\infty)$ satisfying 
\begin{itemize}
    \item $g$ is increasing in $[0,\infty)$ and  is decreasing in $(-\infty,0]$,
    \item $g|_{(-\infty,-M]}\equiv const, \, g|_{[M,\infty)} \equiv const$, 
			  $g|_{[-\eta,\eta]} \equiv 0$,  $\| g \|_\infty:=\sup_{x \in \R}g(x) \leq 1.$
\end{itemize}
From now on, fix an arbitrary $\epsilon>0$ and take $f \in \mathcal{C}(1)$ such that
\begin{align}\label{eq:choice_f}
	E(f) \leq r(1) +\epsilon.
\end{align}
Let  $\delta \in (0,1/4)$ be small enough so that	\begin{align}\label{eq:choice_delta}
		\frac{E(f)+4\delta}{\sqrt{1-4\delta}} \leq r(1) +2\epsilon,
	\end{align}
We prepare several claims that will be proved in the subsequent sections. The following proposition says that $f$ can be approximated by a function in $\bigcup_{M,\eta>0}\overline{\mathcal{C}}(M,\eta)$ with lower energy.
\begin{prop}\label{prop:truncation}
	There exist $M,\eta>0$ and $g_0 \in \overline{\mathcal{C}}(M,\eta)$  depending on $\delta$ and $f$ such that $$E(g_0) \leq E(f)\text{, and } f(x)-\delta \leq g_0(x) \leq f(x)\quad \forall x  \not \in [-M,M].$$ 
\end{prop}

 Let $n\in\N$. Dividing the interval $[-M,M]$ into  subintervals of length $1/n$, we define 
\be{
	\kI:= \kI^+ \cup \kI^-:=\left\{ \left[ \frac{i-1}{n},\frac{i}{n} \right):~i\in \Iintv{1,Mn} \right\}\cup \left\{ \left( \frac{-i}{n},\frac{-i+1}{n} \right]:~i\in \Iintv{1,Mn}  \right\},
}
Let $C_* := 4/\delta.$ We also define
\begin{align}\label{dc1}
	\mathcal{L}^\pm:=\left\{ I \in \mathcal{I}^\pm: \Delta_I(g_0) \geq C_* \frac{\log n}{n} \right\}, \qquad \kL :=\kL^+ \cup \kL^-.
\end{align}
\begin{prop}\label{prop:largeheight}
	For $n\in\N$ large enough, there exists $g_1 \in \overline{\mathcal{C}}(M,\eta)$ so that the following hold:
	\begin{enumerate}
		\renewcommand{\labelenumi}{(\alph{enumi})}
		\item\label{item:largeheight_a}
		$g_1|_{I} \equiv const$ for  $I \in \kL$,  $g_1|_I = g_0|_I$ for  $I \in \kI \setminus \kL$, and  $\Delta_I(g_1) < (C_* \log n)/n$ for $I \in \kI$.
		\item\label{item:largeheight_b}
		 $g_1(x) \leq g_0(x)$ for all $x \in \R$ and $g_1(x)=g_0(x)$ for $x \not\in [-M,M]$.
		\item\label{item:largeheight_c}
		$E(g_1) \leq E(g_0)+\delta$.
	\end{enumerate}
\end{prop}
We define
\ben{ \label{dom}
	\kM :=\left\{I \in \kI:~
\frac{c_* \log n}{n}  \leq \Delta_I(g_1) < \frac{C_* \log n}{n} \right\}, \quad \text{ with}\quad c_*:=\frac{\delta}{16(E(f)+1)}<C_*.
}

\begin{prop}\label{prop:rellarge}
For $n\in\N$ large enough, there exists $g_2 \in \overline{\mathcal{C}}(M,\eta)$ so that the following hold:
	\begin{enumerate}
		\renewcommand{\labelenumi}{(\alph{enumi})}
		\item $g_2|_I \equiv const$ for $I \in \kM \cup \kL$ and  $\Delta_I(g_2)<(c_*\log n)/n$ for $I \in \mathcal{I}$.
		\item $g_2(x) \leq g_1(x)$ for all $x \in \R$ and $g_2(x) \geq g_1(x)-\delta$ for $x \not\in [-M,M]$.
		\item $E(g_2) \leq E(g_1)+\delta$.
	\end{enumerate}
\end{prop}
Finally, we flatten $g_2$ over the remaining intervals  $I \in \kI$ with small height $\Delta_I(g_1) < c_* (\log n)/n$.

\begin{prop}\label{prop:smallheight}
	If $n\in\N$ is large enough, then there exists a step function $g_3 \in \overline{\mathcal{C}}(M,\eta)$ such that $E(g_3) \leq E(g_2)$ and $g_2(x)-\delta \leq g_3(x) \leq g_2(x)$ for all $x \in \R$.
\end{prop}
Assuming these propositions, we first prove \eqref{eq:xi=1}.
\begin{proof} [\bf Proof of \eqref{eq:xi=1}] 
	Let $n \in \N$ be sufficiently large and fixed.
	By \eqref{eq:choice_f} and Propositions~\ref{prop:truncation}--\ref{prop:smallheight}, $g_3$ is a step function on $\R$, increases in $[0,\infty)$, decreases in $(-\infty,0]$, and satisfies that $\lim_{x \rightarrow 0}g_3(x)=0$ by $g_3\in \bar{\kC}(M,\eta)$ and
	\begin{align}\label{eq:prop_g3}
		\text{$E(g_3) \leq E(f)+4\delta$ and $f(x)-4\delta \leq g_3(x) \leq f(x)$ for all $x \not\in [-M,M]$.}
	\end{align}
	Let  $\alpha:=\sup_{y\geq 0} g_3(y)\geq 1-4\delta$. We consider the function
	\begin{align*}
		\phi(x):=\alpha^{-1}g_3(\sqrt{\alpha}x),\qquad x \in \R.
	\end{align*} 
	Lemma \ref{thm2}, \eqref{eq:choice_delta} and \eqref{eq:prop_g3} imply that $\phi\in \mathcal{C}^{\rm Step}(1)$ and 
	\begin{align*}
		E(\phi)
		= \frac{1}{\sqrt{\alpha}}E(g_3)
		\leq  \frac{E(f)+4\delta}{\sqrt{1-4\delta}}
		\leq r(1)+2\epsilon.
	\end{align*}
	Consequently, one has
	\begin{align*}
		\inf\{ E(f): f \in \mathcal{C}^{\rm Step}(1) \} \leq r(1)+2\epsilon.
	\end{align*}
	Since $\epsilon$ is arbitrary, \eqref{eq:xi=1} follows by letting $\epsilon \searrow 0$.
\end{proof}

\subsection{Proofs of Propositions~\ref{prop:truncation} and  \ref{prop:largeheight} }\label{subsect:pfc-l}

\subsubsection{Proof of Proposition~\ref{prop:truncation}}

Let $\delta>0$ and $f\in \mathcal{C}(1)$.
Since $\lim_{x \to \infty} f(x) = 1$ and $\lim_{x\rightarrow -\infty} f(x)$ exists, we can take   $L>0$ such that
\begin{align*}
f(L) \geq 1- \frac{\delta}{3}, \quad 	f(-L) \geq \lim_{x \to -\infty} f(x)-\frac{\delta}{3}.	
\end{align*}
We next take $\eta\in (0,\delta/3)$ small enough so that $ \max_{|x| \leq \eta} f(x) <\frac{\delta}{3}$,
which is possible thanks to $\lim_{x \rightarrow 0} f(x)=0$.
Apply Lemma~\ref{lem: deformation of function} with $\mathcal{J}:=\{ (-\infty,-L],[-\eta,\eta],[L,\infty) \}$ to obtain the deformation 
\begin{align*}
	g_0(x):=\fsd(x), \qquad M:=L+\eta.
\end{align*}
By the constructions of $\fsd$, one has 
\be{
f(x)-\delta \leq  g_0(x)  \leq f(x) \, \, \forall  \, x \in \R, \qquad  E(g_0) \leq E(f).
}
Moreover, by construction, $g_0$ belongs to $ \overline{\mathcal{C}}(M,\eta)$. \qed 

\subsubsection{Proof of Proposition~\ref{prop:largeheight}}
 Recall the notations $\mathcal{L}^\pm$ from \eqref{dc1}. We consider the hard deformation
\begin{align*}
	g_1(x) := g_0^{\textrm{hd},\kL}(x) :=
	\begin{cases}
		\inf_{y \in I} g_0(y), & \text{if $x \in I$ for some $I \in \mathcal{L}$}\\
		g_0(x), & \text{otherwise}.
	\end{cases}
\end{align*}
Clearly, $g_1 \in \overline{\mathcal{C}}(M,\eta)$ holds since $g_0 \in \overline{\mathcal{C}}(M,\eta)$. Moreover, Properties~(a) and (b) of Proposition~\ref{prop:largeheight} are trivial from the construction. Finally, we check Property (c). Since $g_0$ is equal to zero on $[-\eta,\eta]$, we have $I \subset \R_+ \setminus [0,\eta/2]$ for all  $I \in \mathcal{L}^+$ when $n\in\N$ is large enough depending on $\eta$.
Hence, by Lemma~\ref{lem:ning}-(ii), there exists  $c_3=c_3(\eta) \in (0, \infty)$  such that for all $I \in \mathcal{L}^+$ and  $x \in I$,
\be{
	\theta_{g_1}(x) \geq c_3 (\sup I -x).	
}
Thus for all $n\in\N$ large enough depending on $c_3$,
\be{
	\int_I \log \theta_{g_1}(x )\, \dd x\geq \int_0^{1/n}\log{(c_3\, x)}\dd x
	\geq -2 (\log n)/n.
}
By considering $h(x)=g_1(-x)$, we  obtain the same estimate for all $I \in \mathcal{L}^-$. Hence, by Lemma \ref{lem:hd},
\ben{\label{eq:Eg1-Eg2}
	E(g_0)-E(g_1)
	\geq \sum_{I \in \mathcal{L}} \int_I
	\log \theta_{g_1}(x) \, \dd x \,
	\geq -2(\#\mathcal{L}^++\#\mathcal{L}^-)\frac{\log n}{n}.
}
Note that since $0 \leq g_0 \leq 1$ and $g_0$ is monotone in $\R_-$ and $\R_+$, we have
\be{
	\sum_{I \in \kL^+}	\Delta_I (g_0) \leq 1, \quad \sum_{I \in \kL^-}	\Delta_I (g_0) \leq 1.
}
Moreover, $\Delta_I (g_0) \geq C_* (\log n)/ n$ for all $I \in \kL$. Therefore,
\begin{align*}
	\#\mathcal{L}^\pm \leq \frac{n}{C_* \log n}.
\end{align*}
This, combined with \eqref{eq:Eg1-Eg2} and  the choice of $C_*$ as in \eqref{dc1}, gives
\begin{align*}
	E(g_0)-E(g_1) \geq -\frac{4}{C_*} \geq -\delta,
\end{align*}
and Property~(c) follows.
\hfill $\square$

\subsection{Proof of Proposition~\ref{prop:rellarge}}\label{subsect:pfrellarge}
Our goal is to flatten $g_1$ over all the intervals belonging to
\begin{align*}
	\kM :=\kM^{+} \cup \kM^-, \quad  \textrm{ where } \quad 	\mathcal{M}^{\pm}:=\left\{ I \in \mathcal{I}^{\pm}: c_*(\log n)/n \leq \Delta_I(g_1) < C_* (\log n)/n\right\}.
\end{align*}


\subsubsection{Clustering of moderate intervals}
We define 
\ben{ \label{chok}
	K:=2+ \big \lfloor 	8 \delta^{-1} \big(E(g_0)+\delta+4/c_* \big)  \big \rfloor. 
} 
Given a non-empty set $\kA \subset \kM^+$, we enumerate $\mathcal{A}=\{ I_1,\dots,I_\lambda \}$ with $\inf I_1>\dots>\inf I_\lambda$, and define
\be{
\lambda (\kA) := \max \{j \in \{1, \ldots, \lambda \}:  \inf I_i >  \sup I_1 - (i/n)K  \quad \forall i\in\Iintv{1 ,j}\}.
}
Similarly, for $\kA \subset \kM^-$, we enumerate $\mathcal{A}=\{ I_{-1},\dots,I_{-\lambda} \}$ with $\sup I_{-1}<\dots<\sup I_{-\lambda}$, and define 
\be{
	\lambda(\kA) := \max\bigl\{ j \in \{1,\dots,\lambda\}:\sup I_{-i}> \inf I_{-1}-(i/n)K \quad \forall i\in\Iintv{1 ,j} \bigr\}.
}
Let 
\begin{align*}
	\Gamma(\mathcal{A}):= \begin{cases} \{ I_1,\dots,I_{\lambda(\mathcal{A})} \} & \textrm{ if } \kA \subset \kM^+ \\
		\{ I_{-1},\dots,I_{-\lambda(\mathcal{A})} \} & \textrm{ if } \kA \subset \kM^-.
	\end{cases}
\end{align*}
We set $\mathcal{M}_1:=\Gamma(\mathcal{M}^+)$ and $\mathcal{M}_{-1}:=\Gamma(\mathcal{M}^-)$, and  define
inductively,
\begin{align}
	\mathcal{M}_{i+1}:=\Gamma\Big( \mathcal{M}^+ \setminus \bigcup_{j=1}^i\mathcal{M}_j \Big) \quad \textrm{ for } i\geq 1, \qquad \mathcal{M}_{i-1}:=\Gamma\Big( \mathcal{M}^- \setminus \bigcup_{j=i}^{-1}\mathcal{M}_j \Big) \quad \textrm{ for } i\leq -1.
\end{align}
Define also 
\be{
	\ell^+ := \min \Big\{ i \geq 1: \mathcal{M}^+ \setminus \bigcup_{j=1}^i\mathcal{M}_j=\varnothing \Big \}, \qquad \ell^- := \max \Big \{i\leq -1: \mathcal{M}^- \setminus \bigcup_{j=i}^{-1}\mathcal{M}_j=\varnothing \Big \}.
}
Finally, for $-\ell^- \leq i \leq \ell^+$ , we set 
\be{
	n_i := \# \kM_i, \qquad 
	t_i := \begin{cases}
		\sup_{I \in \kM_i} \sup I& \textrm{ if } i\geq 1,\\
		\inf_{I \in \kM_i} \inf I& \textrm{ if } i\leq -1,  
	\end{cases}
	\qquad 
	F_i := \begin{cases}
		[t_i - \tfrac{K n_i}{n}, \, t_i) & \textrm{ if } i\geq 1,\\
		(t_i, \, t_i+\tfrac{K n_i}{n}] & \textrm{ if } i\leq -1,
	\end{cases}
}
with the convention that $n_0:=0$ and $F_0:=\varnothing$.

\begin{lem}\label{lem:property_F}
	For $n\in\N$ large enough, the following hold:
	\begin{enumerate}
		\item [(i)]
		It holds
		\begin{align*}
			\sum_{i=1}^{\ell^+}	n_i= \#\mathcal{M}^+ \leq \frac{n}{c_*\log n}, \qquad 			\sum_{i=\ell^-}^{-1}	n_i= \#\mathcal{M}^- \leq \frac{n}{c_*\log n}.
		\end{align*}
		\item [(ii)] The intervals $(F_i)_{i=\ell^-}^{\ell^+}$ are disjoint and do not intersect $[-\eta/4, \, \eta/4]$.
		\item [(iii)]
		For all $1 \leq i \leq \ell^+$ and $x \in [t_i-Kn_i/n, \, t_i -1/n]$,
		\begin{align*}
			g_1(t_i)-g_1(x)
			\geq \frac{c_*|t_i-x|}{2K}\log n;
		\end{align*}
		for all $\ell^- \leq i \leq -1$ and  $x \in [t_i+1/n, \, t_i+ Kn_i/n]$,
		\begin{align*}
			g_1(t_i)-g_1(x)
			\geq \frac{c_*|t_i-x|}{2K}\log n.
		\end{align*}		
	\end{enumerate}
\end{lem}
\begin{proof}
By symmetry,	
we give a proof only for positive parts, i.e.,  $1 \leq i \leq \ell^+$. 	

(i): The equation is trivial since   $(\mathcal{M}_i)_{1 \leq i \leq \ell^+}$ is a partition of $\kM^+$. 
	On the other hand, since $g_1$ is increasing in $[0,\infty)$ and bounded by $1$, the inequality follows from
	\begin{align*}
		1 \geq \sum_{I \in \mathcal{M}^+} \Delta_I(g_1) \geq \left( c_*\frac{\log n}{n} \right)\#\mathcal{M}^+.
	\end{align*}
	
	(ii): Suppose the contrary that  there exists $i\in \Iintv{1,\ell^+}$ such that $\inf F_i \leq \eta/4$, or equivalently    $t_i-Kn_i/n \leq \eta/4$.
	Due to (i), if $n$ is large enough, then 
	\begin{align*}
		t_i \leq \frac{Kn_i}{n}+\frac{\eta}{4} \leq \frac{K}{c_*\log n}+\frac{\eta}{4} \leq \frac{\eta}{2},
	\end{align*}
	 On the other hand, 	since $g_1|_{[-\eta,\eta]} \equiv 0$, we have 
	\be{
		\eta\leq \inf_{I \in \kM^+} \sup I\leq \sup_{I \in \kM_i} \sup I \leq t_i,
	}
	which  is a contradiction. Therefore,  $\inf F_i>\eta/4$ holds for all $1 \leq i \leq \ell^+$. We now show that $(F_i)_{i=1}^{\ell^+}$ are disjoint intervals by proving that $\inf F_i >\sup F_{i+1} = t_{i+1}$ for $1 \leq i \leq \ell^+-1$. Fixing  an index $1 \leq i \leq \ell^+-1$, we denote $\mathcal{A}:=\mathcal{M}^+ \setminus \bigcup_{j=1}^{i-1}\mathcal{M}_j$. We enumerate $\mathcal{A}$ as $\kA=\{ I_1,\dots,I_\lambda \}$ with $\inf I_1>\dots>\inf I_\lambda$. Since $i<\ell^+$, $\lambda(\mathcal{A})<\lambda.$
	The definition of  $\lambda(\mathcal{A})$ gives that
	\be{
		\inf I_{\lambda(\mathcal{A})+1}
		 \leq \sup I_1-\frac{\lambda(\mathcal{A})+1}{n}K=\inf F_i -\frac{K}{n}.
	}
 
	Moreover, since $I_{\lambda(A)+1}$ is the first interval in  $\kM_{i+1}$,
	\be{
		t_{i+1} = \sup I_{\lambda(A)+1} = \inf I_{\lambda(A)+1} + \frac 1 n.
	}
	Therefore, we get that $\inf F_i > t_{i+1}=\sup F_{i+1}$. 
	
	(iii): We claim that for all $1 \leq i \leq \ell^+$ and $x \in [t_i -K n_i/n, t_i-1/n]$,
	\begin{align}\label{eq:number_M}
		\#\left\{ I \in \mathcal{M}_i: I \subset [x,t_i] \right\} \geq \frac{n}{K}|t_i-x|-1.
	\end{align}
	Assuming the above, we first conclude (iii). We fix $1 \leq i \leq \ell^+$.  If $t_i-K n_i/n\leq x\leq  t_i-2K/n$, by the definition of moderate intervals, then
	\begin{align*}
		g_1(t_i)-g_1(x)
		&\geq c_*\frac{\log n}{n} \,\#\left\{ I \in \mathcal{M}_i:I \subset [x,t_i] \right\}\\
		&\geq c_*\frac{\log n}{n}  \left( \frac{n}{K}|t_i-x|-1 \right)\\
		&= \frac{c_*|t_i-x|}{K}\log n-c_*\frac{\log n}{n} \geq \frac{c_*|t_i-x|}{2K}\log n.
	\end{align*}
	If $t_i -2K/n <x \leq t_i-1/n$, by  $g_1(t_i-1/n) \geq g_1(x)$ and $[t_i-1/n,t_i] \in \kM^+$, then 
 \begin{align*}
		g_1(t_i)-g_1(x)
		\geq g_1(t_i)-g_1(t_i-1/n) \geq c_*\frac{\log n}{n}
		\geq \frac{c_*|t_i-x|}{2K}\log n.
	\end{align*}
		Now it remains to prove \eqref{eq:number_M}.
	We  fix $1 \leq i \leq \ell^+$ and enumerate $\mathcal{A}:=\mathcal{M}^+ \setminus \bigcup_{j=1}^{i-1}\mathcal{M}_j $ as $\kA= \{ I_1,\dots,I_\lambda \}$ with $\inf I_1>\dots>\inf I_\lambda$.
	For all $1 \leq j \leq \lambda(\mathcal{A})$, since $\inf I_j>\sup I_1-\frac{j}{n}K=t_i-\frac{j}{n}K$,
	\begin{align}\label{eq:topj}
		j>\frac{n}{K}(t_i-\inf I_j).
	\end{align}
	If $\inf I_{\lambda(\mathcal{A})}<x \leq t_i-1/n$, then there exists a unique $2 \leq j(x) \leq \lambda(\mathcal{A})$ with $\inf I_{j(x)}<x \leq \inf I_{j(x)-1}$. Then, by \eqref{eq:topj}, 
	\begin{align*}
		\#\left\{ I \in \mathcal{M}_i: I \subset [x, t_i] \right\}
		&\geq 
  j(x)-1> \frac{n}{K}(t_i-\inf I_{j(x)})-1 > \frac{n}{K}|t_i-x|-1.
	\end{align*}
	If $t_i -n_iK/n \leq x \leq \inf I_{\lambda(\mathcal{A})}$, then one has
	\begin{align*}
		\#\left\{ I \in \mathcal{M}_i: I \subset [x, t_i] \right\}
		= \lambda(\mathcal{A}),\quad |t_i-x|
		\leq \frac{n_i}{n}K
		= \frac{\lambda( \mathcal{A})}{n}K.
	\end{align*}
	Therefore,
	\begin{align*}
		\#\left\{ I \in \mathcal{M}_i: I  \subset [x, t_i] \right\}
		\geq \frac{n}{K}|t_i -x|.
	\end{align*}
\end{proof}
\begin{lem}\label{lem:choice_K}
	For all $n$ large enough,
	\begin{align*}
		\sum_{i=\ell^-}^{\ell^+} \frac{n_i}{n} \log\frac{n_i}{n} \geq -\frac{\delta}{4},
	\end{align*}
	with the convention that $n_0=0$ and $0 \log 0 =0$.
\end{lem}
\begin{proof}
	By Lemma~\ref{lem:property_F}-(iii), if $n\in\N$ is large enough, then for any $1 \leq i \leq \ell^+$ and $ x \in [t_i - Kn_i/n, \,   t_i -1/n]$,
	\begin{align*}
		\log 	\theta_{g_1}(x) \leq 	\log\mathbb{P}_x^{\rm BM}(\tau_{t_i} \geq g_1(t_i)-g_1(x))
		&\leq \log\mathbb{P}_x^{\rm BM} \left( \tau_{t_i} \geq \frac{c_*|t_i-x|}{2K}\log n \right)\\
		& = \frac{1}{2}\log|t_i-x| - \frac{1}{2} \log \log n + \kO(\log K) \\
		& \leq \frac{1}{2}\log|t_i-x|,
	\end{align*}
 where we have used $\pp_x^{\rm BM}(\tau_a \geq b) \asymp |a-x| /\sqrt{b}$ for $a, x \in \R$ and $b>0$, by Lemma \ref{lem:maxbrown}. 	Similarly, for any $\ell^- \leq i \leq -1$ and $x \in [t_i+1/n,t_i +Kn_i/n]$,
	\begin{align*}
		\log \theta_{g_1}(x) \leq 	\log\mathbb{P}_x^{\rm BM}(\tau_{t_i} \geq g_1(t_i)-g_1(x))
		\leq \frac{1}{2}\log|t_i-x|.
	\end{align*}
	Therefore, since $F_i$'s are disjoint by Lemma~\ref{lem:property_F}-(ii), we have
	\begin{align*}
		-E(g_1)
		&\leq \sum_{i=1}^{\ell^+}\int_{t_i-K n_i/n}^{t_i-1/n}
		\log \theta_{g_1}(x) \,\dd x \, + \, \sum_{i=\ell^-}^{-1}\int^{t_i+K n_i/n}_{t_i+1/n}
		\log \theta_{g_1}(x) \,\dd x \\
		&\leq \frac{1}{2}  \sum_{i\in \{\ell^-,\cdots,\ell^+\}\setminus\{0\}}\int_{1/n}^{Kn_i/n} \log t\,\dd t.
	\end{align*}
	A straightforward calculation shows that
	\begin{align*}
		\int_{1/n}^{K n_i/n} \log t\,\dd t&=\frac{Kn_i}{n}\log{\frac{Kn_i}{n}}-\frac{Kn_i}{n}-\frac{1}{n}\log{\frac{1}{n}}+\frac{1}{n}\\
		&\leq \frac{Kn_i}{n}\log\frac{n_i}{n}+\frac{Kn_i}{n}\log K+ \frac{\log n}{n},
	\end{align*}
 Moreover,  by Part~(i) of Lemma~\ref{lem:property_F},$$\ell^+ +\ell^-\leq \sum_{i=\ell^-}^{\ell^+}n_i =\#\mathcal{M} \leq \frac{2 n}{c_*\log n},$$ 
 for $n\in\N$  large enough. Therefore, 
	\begin{align*}
		-E(g_1)
		&\leq \frac{K}{2} \sum_{i=\ell^-}^{\ell^+}\frac{n_i}{n}\log\frac{n_i}{n} +\frac{K \log K + \log n}{n}\cdot \frac{2 n}{c_*\log n}\\
		&\leq\frac{K}{2}\sum_{i=\ell^-}^{\ell^+}\frac{n_i}{n}\log\frac{n_i}{n} 		+\frac{4}{c_*},
	\end{align*}
	and hence by the choice of $K$ as in \eqref{chok},
	\begin{align*}
		\sum_{i=\ell^-}^{\ell^+}\frac{n_i}{n}\log\frac{n_i}{n}
		&\geq -\frac{2}{K}(E(g_1)+4c_*^{-1})\\
		&\geq -\frac{2}{K}(E(f)+2\delta+4c_*^{-1}) \geq -\frac{\delta}{4},
	\end{align*}
{\NK and the lemma follows.}
\end{proof}
Recall that the function $g_1$ obtained in Proposition \ref{prop:largeheight}  satisfies 
\begin{itemize}
	\item[(a1)] $g_1 \in \overline{\mathcal{C}}(M,\eta)$
	\item [(b1)] $g_1$ is constant in each interval $I \in \kL$, and $\Delta_I (g_1) \leq C_* (\log n)/n$ for the other intervals. 
\end{itemize}
Let us  define 
\bea{
	\kF_{\rm h} &:=& \big \{ F_i: \ell^- \leq i \leq \ell^+, \, \Delta_{F_i}(g_1) \leq C_*K^2 \frac{n_i}{n} \log n \big \}\\
	\kF_{\rm s} &:=& \big \{ F_i: \ell^- \leq i \leq \ell^+, \, \Delta_{F_i}(g_1) > C_*K^2 \frac{n_i}{n} \log n \big \},
}
and 
\bea{
	\kM_{\rm h} &:=&\{I \in \kM: I \subset \kF_i \textrm{ for some } F_i \in \kF_{\rm h} \} \\
	\kM_{\rm s} &:=&\{I \in \kM: I \subset \kF_i \textrm{ for some } F_i \in \kF_{\rm s} \}.
}
Note that 
\be{
	\kF_{\rm h} \cup \kF_{\rm s} = \kF := \{F_i: \ell^- \leq i \leq \ell^+\}, \quad \kM = \kM_{\rm h} \cup \kM_{\rm s}. 
}
Our strategy is to first apply the hard deformation with   $g_1$ over the intervals in $\kF_{\rm h}$, and then apply the soft deformation over the intervals in $\kM_{\rm s}$. We finally confirm that the final function   will satisfy all of the desired conditions.

We consider the hard deformation of $g_1$ over $\kF_{\rm h}$ as 
\begin{align*}
	\tilde{g}_1(x):= g_1^{\rm hd, \kF_{\rm h}} (x)=
	\begin{cases}
		g_1(t_i-\tfrac{Kn_i}{n}) & \text{if $x \in F_i$ for some $F_i \in \kF_{\rm h}$ with $i\geq 1$},\\
		g_1(t_i+\tfrac{Kn_i}{n}) & \text{if $x \in F_i$ for some $F_i \in \kF_{\rm h}$ with $i\leq -1$},\\
		g_1(x), & \text{otherwise}.
	\end{cases}
\end{align*}

\begin{lem}\label{prop:tilde_g}
	For $n$ large enough, the following hold:
	\begin{itemize}
		\item [($\tilde{a}$1)] $\tilde{g}_1 \in \overline{\mathcal{C}}(M,\eta)$,
		\item [($\tilde{b}$1)] $\tilde{g}_1$ is constant on each interval $I \in \kF_{\rm h} \cup \kL$, and  $\Delta_I(\tilde{g}_1) \leq C_* \log n/n$ for the other intervals.
		\item [($\tilde{c}$1)] $\tilde{g}_1(x) \leq g_1(x)$ for all $x$ and $\tilde{g}_1(x) = g_1(x)$ for $x \not \in [-M,M]$.
		\item [($\tilde{d}$1)]  $			E(\tilde{g}_1)-E(g_1) \geq -\delta.$
		
	\end{itemize}
\end{lem}
\begin{proof}
	The first three properties of $\tilde{g}_1$ are trivial from the its definition and Properties (a1), (b1) of $g_1$. We now prove the last property. 
	By Lemma \ref{lem:hd}, 
	\begin{align}\label{eq:divide_g}
		\begin{split}
		E(g_1)-	E(\tilde{g}_1)
			&\geq \sum_{\substack{i \geq 1 \\ F_i \in \kF_{\rm h}}} \int_{F_i}\log \frac{\theta_{\tilde{g}_1}(x)}{\theta_{g_1}(x)}\, \dd x + \sum_{\substack{i \leq -1 \\ F_i \in \kF_{\rm h}}} \int_{F_i}\log \frac{\theta_{\tilde{g}_1}(x)}{\theta_{g_1}(x)}\, \dd x.
		\end{split}
	\end{align}
 We decompose the first term as 
	\be{
		\sum_{\substack{i \geq 1 \\ F_i \in \kF_{\rm h}}} \int_{F_i}\log \frac{\theta_{\tilde{g}_1}(x)}{\theta_{g_1}(x)}\, \dd x \geq  \textrm{(I)}+\textrm{(II)},
	}
	where 
	\be{ 	\textrm{(I)}	:=\sum_{\substack{i \geq 1 \\ F_i \in \kF_{\rm h}}} \int^{t_i}_{t_i-n_i/n}\log \theta_{\tilde{g}_1}(x)\, \dd x, \qquad  \textrm{(II)} :=\sum_{\substack{i \geq 1 \\ F_i \in \kF_{\rm h}}} \int_{t_i-Kn_i/n}^{t_i-n_i/n}\log \frac{\theta_{\tilde{g}_1}(x)}{\theta_{g_1}(x)}\, \dd x.
	}
 For any $i \geq 1$, since $\tilde{g}_1|_{F_i}\equiv const$ and $|F_i|\leq 1$  by Lemma~\ref{lem:property_F}-(i),  Lemma \ref{lem:ning} implies that there exist positive constants $C_1, C_2\in (0,1)$ depending on $\eta$ such that for all $x \in F_i$,
	\ben{\label{eq:tilde_prob}
		1\geq \theta_{\tilde{g}_1}(x)
		\geq C_1\mathbb{P}_x^\textrm{BM}(\tau_y \geq \tilde{g}_1(y)-\tilde{g}_1(x) \text{ for all $y \geq t_i$}) \geq C_2 (t_i-x).
	}
	Therefore, 
	\bean{ \label{eoii}
		\textrm{(I)}
		\geq \sum_{i = 1}^{\ell^+} \int^{t_i}_{t_i-n_i/n}\log (C_2(t_i-x))\, \dd x 
		=  \sum_{i=1}^{\ell^+}\frac{n_i}{n}\log\frac{n_i}{n} + (\log C_2 -1) \sum_{i=1}^{\ell^+}\frac{n_i}{n}.\notag
	}
	For (II), by using \eqref{eq:tilde_prob} and Lemma~\ref{lem:comp}, we obtain
	\be{
		\frac{\theta_{\tilde{g}_1}(x)}{\theta_{g_1}(x)} \geq  C_1\frac{\mathbb{P}_x^\textrm{BM}(\tau_y \geq \tilde{g}_1(y)-\tilde{g}_1(x) \text{ for all $y \geq t_i$})}{\mathbb{P}_x^\textrm{BM}(\tau_y \geq g_1(y)-g_1(x) \text{ for all $y \geq t_i$})} \geq C_1 \left( \frac{g_1(t_i)-g_1(x)}{g_1(t_i)-\tilde{g}_1(x)} \right)^{3/2}.
	}
 Moreover, if  $x \in [t_i- Kn_i/n, t_i-n_i/n]$, Lemma~\ref{lem:property_F}-(iii) gives that
	\begin{align*}
		g_1(t_i)-g_1(x)
		\geq \frac{c_*|t_i-x|}{2K}\log n
		\geq \frac{c_*}{2K} \times \frac{n_i}{n}\log n.
	\end{align*}
	Using  the definition of $\tilde{g}_1$, and $F_i \in \kF_{\rm h}$,
	\begin{align*}
		g_1(t_i)-\tilde{g}_1(x) = 	g_1(t_i)-g_1\left(t_i-Kn_i/n \right) \leq C_*K^2 \frac{n_i}{n}\log n.
	\end{align*}
Therefore, we arrive at 
	\be{
		\frac{\theta_{\tilde{g}_1}(x)}{\theta_{g_1}(x)} \geq C_3:=  C_1 \left(\frac{c_*}{2C_*K^3}\right)^{3/2},
	}
	which together with Lemma~\ref{lem:property_F}-(i) implies that for all $n\in\N$ large enough,
	\begin{align*}
		\textrm{(II)}
		&\geq  \sum_{i =1}^{\ell^+} \frac{(K-1)n_i\log C_3}{n}  \geq \frac{(K-1)\log C_3}{{c_*\log n}}
		\geq -\frac{\delta}{4}.
	\end{align*}
Using the above estimate and \eqref{eoii}, we get 
	\begin{align*}
		\sum_{\substack{i \geq 1 \\ F_i \in \kF_{\rm h}}} \int_{F_i}\log \frac{\theta_{\tilde{g}_1}(x)}{\theta_{g_1}(x)}\, \dd x \geq 	-\frac{\delta}{4}+\sum_{i=1}^{\ell^+}\frac{n_i}{n}\log\frac{n_i}{n} + (\log C_2 -1) \sum_{i=1}^{\ell^+}\frac{n_i}{n}.
	\end{align*}
	We have the same for the negative part. Thus, by \eqref{eq:divide_g}, Lemma~\ref{lem:property_F}-(i), Lemma~\ref{lem:choice_K}, for $n$ large enough, then
	\begin{align*}
		E(g_1)-E(\tilde{g}_1)
		\geq -\frac{\delta}{2}+ \sum_{i=\ell^-}^{\ell^+}\frac{n_i}{n}\log\frac{n_i}{n} + (\log C_2 -1) \sum_{i=-\ell^-}^{\ell^+}\frac{n_i}{n}
		\geq -\delta.
	\end{align*}
\end{proof}

\begin{proof}[\bf Proof of Proposition~\ref{prop:rellarge}]
	We consider the soft deformation:
	\begin{align*}
		g_2:=(\tilde{g}_1)^{\textrm{sd},\kM_{\rm s}}.
	\end{align*}
	Clearly, $g_2 \in \overline{\mathcal{C}}(M,\eta)$ and $g_2$ satisfies the following conditions:
	\begin{itemize}
		\item $g_2(x) \leq g_1(x)$ for all $x$.
		\item $g_2|_I \equiv const$ for all $I \in \kM \cup \kL$ and  $\Delta_I(g_2) \leq c_* (\log n)/n$ for all $I \in \kI \setminus \kM \cup \kL $.
		\item $E(g_2) \leq E(\tilde{g}_1) \leq  E(g_1)+\delta$.
	\end{itemize}
	Since $\tilde{g}_1(x) =g_1(x)$	 for all $x \not \in [-M,M]$, to show Property (b) of Proposition \ref{prop:rellarge}, it  remains to prove 
	\begin{align} \label{g21de}
		g_2(x) \geq \tilde{g}_1(x)-\delta.
	\end{align}
	By the definition of $\kM_{\rm s}$ and $\kF_{\rm s}$, we have 
	\ben{
		\#\kM_{\rm s} \leq \sum_{ F_i \in \kF_{\rm s}} n|F_i| = \sum_{ F_i \in \kF_{\rm s}} Kn_i.
	}
	Combined  with the fact that $\Delta_I(\tilde{g}_1) \leq C_* (\log n)/n$ for all $I\in \kM_{\rm s}$, this yields
	\be{
		\sum_{I \in \kM_{\rm s}} \Delta_I(\tilde{g}_1) \leq \frac{C_*K \log n}{n} \sum_{ F_i \in \kF_{\rm s}} n_i.
	}
	Since $(F_i)_{i=\ell^-}^{\ell^+}$ are disjoint intervals that do not contain $0$, and $g_1$ is a monotone function both on 
  $(-\infty,0)$ and $(0, \infty)$  with $0 \leq g_1 \leq 1$, by using $\Delta_{F_i}(g_1) \geq  C_* K^2 n_i (\log n)/n$ if $F_i \in \kF_{\rm s}$,
	\be{
		2 \geq \Delta_{(-\infty,0)}(g_1) + \Delta_{(0,\infty)}(g_1) 
  \geq  \sum_{i=-\ell^-}^{\ell^+} \Delta_{F_i}(g_1)  \geq  \sum_{F_i \in \kF_{\rm s}} \Delta_{F_i}(g_1) \geq \frac{C_*K^2 \log n}{n} \sum_{F_i \in \kF_{\rm s}} n_i,
	}
	 It follows from the last two estimates that 
	\be{
		\sum_{I \in \kM_{\rm s}} \Delta_I(\tilde{g}_1) \leq 2/K.
	}
	Combining this   with Lemma~\ref{lem: deformation of function} yields that for all $x \in \R$,
	\begin{align*}
		g_2(x)
		&\geq \tilde{g}_1(x)
		-	\sum_{I \in \kM_{\rm s}} \Delta_I(\tilde{g}_1)
		\geq \tilde{g}_1(x)-2K^{-1} \geq \tilde{g}_1(x)-\delta,
	\end{align*}
and \eqref{g21de} follows.
\end{proof}

\subsection{Proof of Proposition~\ref{prop:smallheight}} Let $g_2$ be the function constructed in  Proposition~\ref{prop:rellarge}.  We have $g_2|_I \equiv const$ for all $I \in \kM \cup \kL$, so it remains to flatten $g_2$ to a step function using the soft deformation.  Define 
\begin{align*}
	\mathcal{S}(1)&:=\left\{ I \in \kI: \Delta_I(g_1)<n^{-3/2} \right\},\\
	\mathcal{S}(2)&:=\biggl\{  I \in \kI: n^{-3/2} \leq \Delta_I(g_1) < c_* (\log n)/n\biggr\}, \\
	\mathcal{S}&:= \mathcal{S}(1) \cup \mathcal{S}(2),		
\end{align*}
and consider
	$g_3:=g_2^{\textrm{sd},\mathcal{S}}.$
Thanks to Lemma~\ref{lem: deformation of function},  the condition $E(g_2) \geq E(g_3)$ immediately follows.
Hence, it suffices to check that $g_3(x) \geq g_2(x)-\delta$ for all $x \in \R$.
Use Lemma~\ref{lem: deformation of function} again to obtain that for all $x \in \R$,
\begin{align}\label{f-g}
	\begin{split}
		g_3(x)-g_2(x)
		&\geq -n^{-3/2}\,\#\mathcal{S}(1)-\frac{c_*\log n}{n}\,\#\mathcal{S}(2)\\
		&\geq -(2M+1)n^{-1/2}-\frac{c_*\log n}{n}\,\#\mathcal{S}(2),
	\end{split}
\end{align}
since $\#\mathcal{S}(1) \leq (2M+1)n$.
To estimate $\#\mathcal{S}(2)$, note that if $n\in\N$ is large enough, then for any $I=[a,a+1/n) \in \mathcal{S}(2) \cap \mathcal{I}^+$,  since $g_2(x)=0$ on $[-\eta,\eta]$, one has $I-1/n:=[a-1/n,a) \in \mathcal{I}^+$. Moreover, by Lemma \ref{lem:hd}-(ii) with a universal positive constant $C$, for all $x 
\in I-1/n$, we have 
\be{
	\theta_{g_2}(x) \leq  \frac{C(\sup I -x)}{\sqrt{g_2(\sup I)-g_2(x)}}	\leq \frac{2C/n}{\sqrt{n^{-3/2}}} =2 C n^{-1/4}.
} 
 The same inequality holds for  all $x \in I+1/n:=(a,a+1/n] $ with  $I=(a-1/n,a] \in \mathcal{S}(2) \cap \mathcal{I}^-$. Therefore,
if $n\in\N$ is large enough, then
\begin{align*}
	-E(g_2)
	&\leq \sum_{I \in \mathcal{S}(2) \cap \mathcal{I}^+}
	\int_{I-1/n} \log \theta_{g_2}(x)\,\dd x +\sum_{I \in \mathcal{S}(2) \cap \mathcal{I}^-}
	\int_{I+1/n} \log \theta_{g_2}(x)\,\dd x\\
	&\leq \#\mathcal{S}(2)\,	\int_0^{1/n} \log{(2Cn^{-1/4})}\,\dd x\\
	&\leq -\frac{\log n}{8n}\,\#\mathcal{S}(2).
\end{align*}
The above estimate, combined with Propositions~\ref{prop:truncation}--\ref{prop:rellarge}, implies that
\begin{align*}
	\#\mathcal{S}(2) \leq 8E(g_2)\frac{n}{\log n}
	\leq 8(E(f)+2\delta)\frac{n}{\log n}.
\end{align*}
Combining this  with  \eqref{f-g} and the choice of $c_*$ (see \eqref{dom}) yields that for all $x \in \R$,
\begin{align*}
	g_3(x)-g_2(x)
	\geq -(2M+1)n^{-1/2}-8(E(f)+2\delta)c_*
	\geq -\delta.
\end{align*}
\hfill $\square$

\section*{Appendix}
In this appendix, for simplicity, we write $\pp_x$ for $\pp_x^{\rm BM}$, and write $\pp$ for $\pp_0^{\rm BM}$.
\begin{proof} [\bf Proof of Lemma \ref{lem:ning}] 
	We claim that for any $\delta>0$ there exists $c=c(\delta)>0$ such that for any $x\geq \delta$, $a\geq 0$ and  $f \in \kC(1)$ satisfying $f|_{[-\delta,\delta]}\equiv 0$,  
	\ben{ \label{keyning}
		\pp_x(\tau_y \geq f(y)-f(x) \, \forall y \in (-\infty,0]\cup [x+a,\infty)) \geq c \, \pp_x(\tau_y \geq f(y)-f(x) \, \forall y \geq x+a).
	}
	Assuming this claim for a moment, we finish the proof of Lemma \ref{lem:ning}. Let $I \subset \R_+ \setminus [0, \delta]$ with $|I| \leq 1$ and assume that $f|_{I} \equiv \text{const}$. Then, since $f|_{I} \equiv \text{const}$ and $0 \leq f \leq 1$ there exists  $\tilde{c}=\tilde{c}(\delta)>0$ such that for any $x \in I$ we have
	\bea{
		\theta_f(x)&=& \pp_x(\tau_y \geq f(y)-f(x) \quad \forall y  \in \R ) \\
		&=& \pp_x(\tau_y \geq f(y)-f(x) \quad \forall y \in (-\infty,0]\cup [\sup I,\infty))\\
	&\geq& c \, \pp_x(\tau_y \geq f(y)-f(x) \quad \forall y \geq \sup I)\\
		&\geq& c \,\pp_x(\tau_{\sup I} \geq 1)\geq \tilde{c} \, (\sup I-x),
	}
 where we have used  \eqref{keyning} in the third line and  Lemma~\ref{lem:maxbrown} in the last line.
	Now we focus on proving \eqref{keyning}. Let
	\be{
		\kA:=\{ \tau_{-\delta} \geq 1 \}; \quad \kB:=\{\tau_y \geq  f(y)-f(x) \, \, \forall \, y\geq x+a \}; \quad \kC:=\{ \tau_{x+a} \geq 1 \}.
	}
	We claim and prove later an FKG type inequality that 
	\ben{ \label{condfkg}
		\pp_x( \kC^c \cap \kA \mid \kB) \geq \pp_x( \kA \mid \kB) \pp_x(\kC^c \mid \kB).
	}
	This inequality implies that 
	\ben{\label{empock}
		\pp_x(\kC^c\cap  \kA\cap  \kB)=\pp_x(\kC^c\cap  \kA\mid  \kB)\P(\kB) \geq \pp_x(\kA\mid \kB) \pp_x(\kC^c\cap\kB).
	}
Since $\kC \subset \kB$, 
	 \al{
&  \pp_x(\kA\cap\kB)\pp_x(\kC) - \pp_x(\kA\cap\kC)\pp_x(\kB)\\
&= \pp_x(\kA\cap\kC)\pp_x(\kC)+\pp_x(\kA\cap\kB\cap \kC^c)\pp_x(\kC) - \pp_x(\kA\cap\kC)\pp_x(\kC)-\pp_x(\kA\cap\kC)\pp_x(\kB\cap \kC^c) \\
  &=\pp_x(\kC^c\cap\kA\cap\kB)\pp_x(\kC) - \pp_x(\kC^c\cap\kB)\pp_x(\kA \cap\kC).
  }Therefore, by \eqref{empock}, we have 
	\bea{
		\pp_x(\kA \mid \kB) -\pp_x(\kA \mid \kC) &=& \frac{\pp_x(\kA\cap\kB)\pp_x(\kC) - \pp_x(\kA\cap\kC)\pp_x(\kB)}{\pp_x(\kB)\pp_x(\kC)} \\
		&=& \frac{\pp_x(\kC^c\cap\kA\cap\kB)\pp_x(\kC) - \pp_x(\kC^c\cap\kB)\pp_x(\kA \cap\kC)}{\pp_x(\kB)\pp_x(\kC)}\\
		&\geq&  \frac{\pp_x(\kA\mid \kB)\pp_x(\kC^c\cap\kB) \pp_x(\kC)-\pp_x(\kC^c\cap \kB)\pp_x(\kA \cap\kC)}{\pp_x(\kB)\pp_x(\kC)}\\
		&=&\pp_x(\kC^c\mid \kB) \big(\pp_x(\kA\mid \kB)-\pp_x(\kA\mid \kC)\big).
	}
Hence, since $1>\pp_x(\kC^c\mid \kB)$,	\ben{\label{pabc} 
		\pp_x(\kA \mid \kB) \geq \pp_x(\kA \mid \kC).	
	}
	Since $f$ is bounded by $1$ and equals $0$ in $[-\delta,\delta]$, we have
	\bean{ \label{ptafa}
		\pp_x(\tau_y \geq f(y)-f(x) \, \forall y \geq x+a,\, \textrm{and} \, y\leq 0) 	&\geq& \pp_x(\tau_{-\delta} \geq 1, \tau_y \geq f(y)-f(x) \, \forall y \geq x+a) \notag \\
		& =&\pp_x(\kA \cap \kB) =	\pp_x(\kA \mid \kB) \pp_x(\kB) \geq \pp_x(\kA \mid \kC) \pp_x(\kB),
	}
	where for the last inequality we have used \eqref{pabc}.  Moreover,  for all $x>0$,
	\ben{ \label{tda}
		\pp_x(\kA \mid \kC) =\pp(\tau_{-\delta-x} \geq 1 \mid \tau_{a} \geq 1) \geq
		\inf_{b>0} \pp(\tau_{-\delta} \geq 1 \mid \tau_b \geq 1) = \inf_{b>0} \frac{\pp(\tau_{-\delta} \wedge \tau_b  \geq 1)}{\pp(\tau_b \geq 1)}.
	}
	By the strong Markov property,
	\bea{
		\pp(\tau_{-\delta} \wedge \tau_b >1) &\geq& \pp(\tau_{-\delta} \wedge \tau_b \geq 1, \tau_{-\delta/2}< \tau_{b/2})\\
		&=& \pp(\tau_{-\delta} \wedge \tau_b \geq 1 \mid \tau_{-\delta/2}< \tau_{b/2}) \pp (\tau_{-\delta/2}< \tau_{b/2})\\
		&\geq& \pp_{-\delta/2} (\tau_0 \wedge \tau_{-\delta} \geq 1) \pp (\tau_{-\delta/2}< \tau_{b/2})\\
		&=&   \pp_{-\delta/2} (\tau_0 \wedge \tau_{-\delta}>1)\frac{b}{b +\delta},
	}
where we have used that $\pp(\tau_u< \tau_v)=\tfrac{v}{v+|u|}$ if $u<0<v$. 
	Observe that by Lemma~\ref{lem:maxbrown}, 
	\be{
		\pp(\tau_{b} \geq 1)   \asymp  (b\wedge 1).
	}
Thus,	$\displaystyle\inf_{b>0} \frac{\pp(\tau_{-\delta} \wedge \tau_b  \geq 1)}{\pp(\tau_b \geq 1)}$ is a positive constant depending on $\delta$. Together with \eqref{ptafa}  and \eqref{tda}, we have  \eqref{keyning}.
\end{proof}

\begin{proof}[ {\bf  Proof of  \eqref{condfkg} (Conditional FKG inequality)}] 
For simplicity of notation, we set $x=0$ and $f(0)=0$ and focus on proving that 
\ben{
\label{fkg}
		\pp( \kA \cap \kD \mid \kB) \geq \pp( \kA \mid \kB) \pp(\kD \mid \kB),
	}
where 
\be{
\kA:= \{\tau_{-\delta} \geq  1\}; \quad \kB:=\{\tau_y \geq f(y) \,\, \forall \,  y \geq a\};\quad \kD:=\kC^c := \{\tau_{a}<1\};  
} 
with $(B_s)_{s\geq 0}$ being the standard Brownian motion.  Observe that it is sufficient to consider the case where  $f$ is a step function. Indeed, for a non-decreasing function $f$ on $[0,\infty)$ and $\ell\in\N$, we define
$$f_\ell(x):=    \lfloor f(x) 2^\ell \rfloor 2^{-\ell}.$$
Since $f$ is non-decreasing, $f_\ell$ is a step function and $f_\ell(x)$ increases to $f(x)$ as $\ell\to\infty$ for any $x\geq 0$.
Define $\kB_{\ell}$ the corresponding event of $f_{\ell}$, i.e. $\kB_{\ell}=\{\tau_y \geq f_{\ell}(y)  \, \, \forall \, y \geq a \}$. Then $(\kB_{\ell})_{\ell \geq 1}$ is a sequence of  decreasing events that converges to $\kB$. If  $\pp(\cdot\mid \kB_{\ell})$ satisfies the inequality \eqref{fkg} for any $\ell\in\N$, by the dominated  convergence theorem, then so does the measure $\pp(\cdot \mid \kB)$. 
	Therefore, we assume that $f$ is a non-decreasing step function  on $[0,\infty)$ bounded by $1$. 
With some $0< b_1 < \ldots<b_k \leq 1$ and $a  \leq  a_1 < \ldots<a_k$ and $k \in \N$ determined by the step function $f$, we write
	\be{
		\kB=  \left\{\max_{0 \leq s \leq b_1} B_s \leq a_1, \ldots, \max_{0 \leq s \leq b_k} B_s \leq a_k \right \}\quad \textrm{a.s.}
	}	
	
 Let us consider the Gaussian random walk $(S_m)_{m\geq 0}$ with $S_0=0$ and $S_m=X_1+\ldots+X_m$ for $m\geq 1$ where $(X_i)_{i\geq 1}$ is a sequence of i.i.d. standard normals.  We take $n\in \N$ that finally goes to infinity, and we define for $i=\Iintv{1,k}$,
	\be{
		n_i= \lfloor nb_i \rfloor.
	}
	Given  $\beta>0$, 
 let $\pp_{n,\beta}$ be a probability measure on $\R^n$ with the probability  density $p(s) $ which is proportional to 
	\be{
		 q(s):= \exp \left(  \beta \prod_{i=1}^k  \prod_{m=1}^{n_i} \1\left\{ \frac{s_m}{\sqrt{n}} \leq a_i \right\} -\frac{1}{2} \sum_{i=1}^{n} (s_i-s_{i-1})^2 \right), \qquad s=(s_i)_{i=1}^n \in \R^n,
	}
	with the convention  $s_0:=0$.   Since $q$ is integrable, the measure  $\pp_{n,\beta}$ is well defined.  On $\R^n$ we consider the following partial order $s=(s_i)_{i=1}^n \leq s'=(s'_i)_{i=1}^n$ if $s_i \leq s'_i$ for all $i=1,\ldots,n$. Moreover, for $s=(s_i)_{i=1}^n, s'= (s'_i)_{i=1}^n \in \R^n$, we define 
	\be{
		s\vee s' = (s_i\vee s'_i)_{i=1}^n,  \qquad s\wedge s' = (s_i\wedge s'_i)_{i=1}^n,
	}
and
	\bea{
		l(s) :=  \log q(s)= \beta \prod_{i=1}^k  \prod_{m=1}^{n_i} \1\left\{ \frac{s_m}{\sqrt{n}} \leq a_i \right\} -\frac{1}{2} \sum_{i=1}^{n} (s_i-s_{i-1})^2 
		=:\beta l_1(s) + l_2(s).
	} We check that $q$ (or equivalently $p$) satisfies the log-suppermodular inequality, that is for all $s,s' \in \R^n$,
	\ben{ \label{lgsumo}
		l(s\vee s') + l(s \wedge s') \geq l(s) +l(s').
	}
	Indeed,  if $l_1(s)=l_1(s')=1$, then $l_1(s\vee s')=l_1(s\wedge s')=1$, and if $l_1(s)+l_1(s')=1$ then $l_1(s \wedge s') =1$. Hence, in all cases, $l_1(s\vee s') + l_1(s\wedge s') \geq l_1(s)+l_1(s')$. Next, for each $i$, we consider
	\be{
		r_i := (s_{i+1}\vee s'_{i+1}-s_{i}\vee s'_{i})^2 +   (s_{i+1}\wedge s'_{i+1}-s_{i}\wedge s'_{i})^2 - (s_{i+1}-s_i)^2 -(s'_{i+1}-s'_i)^2.
	}
	If either $s_{i+1} \geq s'_{i+1}$ and $s_{i}\geq s'_i$ or $s_{i+1} \leq s'_{i+1}$ and $s_{i}\leq s'_i$, then we have $r_i=0$. If $s_{i+1}\geq s'_{i+1}$ and $s_i \leq s'_i$ then
	\be{
		r_i= (s_{i+1}- s'_{i})^2 +   (s'_{i+1}-s_{i})^2 - (s_{i+1}-s_i)^2 -(s'_{i+1}-s'_i)^2 = 2(s_{i+1}-s'_{i+1})(s_i-s'_i) \leq 0.
	}
	Similarly, $r_i \leq 0$ when  $s_{i+1}\leq s'_{i+1}$ and $s_i \geq s'_i$. In all cases, we have $r_i \leq 0$, and thus 
	\be{
		l_2(s\vee s') +l_2(s\wedge s') -l_2(s) -l_2(s') = -\frac{1}{2} \sum_{i=1}^n r_i \geq 0.
	}
	Therefore, we have \eqref{lgsumo}. Then, by \cite[Proposition 1]{FKG71}, $\pp_{n,\beta}$ satisfies the FKG inequality.  	 Note that 
		$\pp_{n,\beta}$ converges weakly toward $\pp_n ( \cdot \mid  \kB_n )$ as $\beta \rightarrow \infty$,  where  $\pp_n$ is the probability measure of $(S_i)_{i=1}^n$ and 
 \be{
 \kB_n:= \left\{ \max_{0 \leq m \leq n_i} \frac{S_m}{\sqrt{n}} \leq  a_i \, \, \forall \, i=1, \ldots,k \right\}.
 }
	As a consequence, since $\pp_{n,\beta}$ satisfies FKG, by the dominated convergence theorem, so does $\pp_n ( \cdot   \mid  \kB_n)$.  Define 
 \be{
 \kA_n= \left\{ \min_{0 \leq m \leq n} \frac{S_m}{\sqrt{n}} \geq  -\delta \right\}; \qquad \kD_n= \left\{ \max_{0 \leq m \leq n} \frac{S_m}{\sqrt{n}}  \geq a  \right\}.
 }
Observe that the two events $\kA_n$ and $\kD_n$ are increasing, and thus 
\be{
\pp_n(\kA_n \cap \kD_n \mid \kB_n) \geq \pp_n(\kA_n \mid \kB_n) \pp_n(\kD_n \mid \kB_n).
}
	  Recalling the events $\kA, \kB, \kD$ and   $n_i= \lfloor nb_i \rfloor$, since $\kA= \left\{\min_{0 \leq s \leq 1} B_s \geq -\delta \right \}$ and $\kD= \left\{ \max_{0 \leq s \leq 1} B_s \geq a \right \}$ a.s.,  by the  Donsker's invariance principle, $\pp_n(\kA_n \cap \kD_n \mid \kB_n) \rightarrow \pp(\kA \cap \kD \mid \kB)$, $\pp_n(\kA_n  \mid \kB_n) \rightarrow \pp(\kA \mid \kB)$ and $\pp_n(\kD_n  \mid \kB_n) \rightarrow \pp(\kD \mid \kB)$ as $n \rightarrow \infty$. Therefore, \eqref{fkg} holds.	
\end{proof}

\begin{proof}[\bf Proof of Lemma \ref{lem:comp}] Recall that $\ell_{b,a}:=f(b)-f(a)$ and $\tilde{\ell}_{b,a}:=f(b)-\tilde{f}(a)$. We assume that  $\tilde{\ell}_{b,a}$ and $\mathbb{P}_a(\tau_y \geq f(y) -f(a)\quad\forall y \geq b)$ are both positive since otherwise the lemma is trivial. 	As $\tilde{f}(y) \leq f(y)$ for $y\geq b$, it suffices to check that
	\begin{align}\label{gcp}
		\begin{split}
			\frac{\mathbb{P}_a(\tau_y \geq f(y)-\tilde{f}(a)\quad\forall y \geq b \big)}{\mathbb{P}_a(\tau_y \geq f(y)-f(a)\quad\forall y \geq b \big)}
			\geq \left( \frac{\ell_{b,a}}{\tilde{\ell}_{b,a}} \right)^{3/2}.
		\end{split}
	\end{align}
We remark that for all $y \in \R$ and $\ell >0$,
\be{
\{ \tau_y \geq \ell  \} = \{M_{\ell} \leq y   \} \quad \textrm{a.s.,\quad where} \quad  M_{t}:= \max_{0 \leq s \leq t} B_s \, \textrm{ for }t>0.
}
 Therefore, by the Markov property and the fact that $\ell_{y,a}-\ell_{b,a}=f(y)-f(b)$, we have 
	\begin{align*}
		&\mathbb{P}_a(\tau_y \geq f(y)-f(a)\quad\forall y \geq b)\\
		&= \mathbb{E}_a\left[ \1\{ M_{\ell_{b,a}} \leq b\}\,
		\mathbb{P}_a \left( \max_{\ell_{b,a} \leq s \leq \ell_{y,a}} B_s \leq y\quad\forall y\geq b \middle| B_{\ell_{b,a}} \right) \right]
	\\
		&= \mathbb{E} \left[ \1\{ M_{\ell_{b,a}} \leq b-a\}\,
		\mathbb{P}_{B_{\ell_{b,a}}}(M_{f(y)-f(b)} \leq y\quad\forall y\geq b) \right].
	\end{align*}
  By \cite[Proposition~{8.1}]{KarShr91_book},
	for $t>0$, $\beta \geq 0$ and $\alpha \leq \beta$,
	\begin{align*}
		\mathbb{P}(B_t \in \dd \alpha,\,M_t \in \dd \beta)
		=\frac{2(2\beta-\alpha)}{\sqrt{2\pi t^3}} \exp\biggl\{ -\frac{(2\beta-\alpha)^2}{2t} \biggr\}\,\dd \alpha\dd \beta.
	\end{align*}
	It then follows  that
	\begin{align*}
		&\mathbb{P}(\tau_y \geq f(y)-f(a)\quad\forall y \geq b)\\
		&= \int_0^{b-a} \int_{-\infty}^{b-a} \1\{ s \leq t \} \,\frac{2(2t-s)}{\sqrt{2\pi\ell^3_{b,a}}} \exp\left\{ -\frac{(2t-s)^2}{2\ell_{b,a}} \right\} \mathbb{P}_s(M_{f(y)-f(b)} \leq y\quad\forall y\geq b)\,\dd s \dd t.
	\end{align*}
 Using the same argument with  $\tilde{f}(a)$ in pleace of $f(a)$,  by  $\tilde{\ell}_{y,a}-\tilde{\ell}_{b,a}=f(y)-f(b)$, we also have
 \begin{align*}
		&\mathbb{P}(\tau_y \geq f(y)-\tilde{f}(a)\quad\forall y \geq b)\\
		&= \int_0^{b-a} \int_{-\infty}^{b-a} \1\{ s \leq t \} \,\frac{2(2t-s)}{\sqrt{2\pi\tilde{\ell}^3_{b,a}}} \exp\left\{ -\frac{(2t-s)^2}{2\tilde{\ell}_{b,a}} \right\} \mathbb{P}_s(M_{f(y)-f(b)} \leq y\quad\forall y\geq b)\,\dd s \dd t.
	\end{align*}
		Since $\ell_{b,a} \leq \tilde{\ell}_{b,a}$,
  \be{
 \exp\left\{ -\frac{(2t-s)^2}{2{\ell}_{b,a}} \right\} \leq  \exp\left\{ -\frac{(2t-s)^2}{2\tilde{\ell}_{b,a}} \right\}. 
  }
  Combining the last three displays, we get the desired estimate  \eqref{gcp}.
\end{proof}

\section*{Acknowledgements}
 N.K.~was supported by JSPS KAKENHI Grant Number JP20K14332. S. Nakajima is supported by JSPS KAKENHI 22K20344.

\end{document}